\documentclass[hidelinks, letterpaper]{article}
\usepackage[margin=1in]{geometry}

\usepackage{amsmath}
\usepackage{amssymb}
\usepackage{amsthm}
\usepackage{setspace}
\usepackage{bm, bbm}
\usepackage{enumitem}
\usepackage{mathrsfs}
\usepackage{hyperref}
\usepackage{xcolor}
\newcommand{\R}{\mathbb{R}}
\newcommand{\W}{\mathscr{W}}

\DeclareMathOperator{\Poisson}{Poisson}
\DeclareMathOperator{\Uniform}{Uniform}

\DeclareMathOperator{\Bernoulli}{Bernoulli}
\DeclareMathOperator{\Var}{Var}
\DeclareMathOperator{\Binomial}{Binomial}
\DeclareMathOperator{\Multinomial}{Multinomial}
\DeclareMathOperator*{\argmax}{argmax}
\DeclareMathOperator*{\dTV}{d_{TV}}

\newtheorem{theorem}{Theorem}
\newtheorem{lemma}{Lemma}
\newtheorem{corollary}{Corollary}
\newtheorem{definition}{Definition}
\newtheorem{proposition}{Proposition}

\theoremstyle{definition}

\hypersetup{
    colorlinks=true,
    linkcolor=blue,
    filecolor=blue,      
    urlcolor=cyan,
    citecolor=red
}

\title{Locally sharp goodness-of-fit testing in sup norm for high-dimensional counts}
\author{Subhodh Kotekal\footnote{Department of Statistics, University of Chicago, United States of America} \and Julien Chhor\footnote{Department of Mathematics and Statistics, Toulouse School of Economics, France} \and Chao Gao\(^*\)}
\date{}

\begin{document}
    \maketitle
    \begin{abstract}
        We consider testing the goodness-of-fit of a distribution against alternatives separated in sup norm. We study the twin settings of Poisson-generated count data with a large number of categories and high-dimensional multinomials. In previous studies of different separation metrics, it has been found that the local minimax separation rate exhibits substantial heterogeneity and is a complicated function of the null distribution; the rate-optimal test requires careful tailoring to the null. In the setting of sup norm, this remains the case and we establish that the local minimax separation rate is determined by the finer decay behavior of the category rates. The upper bound is obtained by a test involving the sample maximum, and the lower bound argument involves reducing the original heteroskedastic null to an auxiliary homoskedastic null determined by the decay of the rates. Further, in a particular asymptotic setup, the sharp constants are identified.
    \end{abstract}
   
    \section{Introduction}
    A canonical problem with a rich history in statistics is goodness-of-fit testing in the context of count data collected across a number of categories. Classically, the problem has been studied in an asymptotic setup with a growing sample size and a fixed number of categories. Pearson's celebrated \(\chi^2\)-statistic is one standard approach in the classical setting. Spurred by technological advancement, it is both practically relevant and theoretically interesting to consider the now typical situation in which the number of categories may be very large. Many categories may exhibit small, if not zero, observed counts. 
    
    In part, we consider count data \(X = (X_1,...,X_p)\) for \(p\) categories following the data generating process  
    \begin{equation}\label{model:poisson}
        X \sim \bigotimes_{j=1}^{p} \Poisson(\lambda_j),
    \end{equation}
    where \(\lambda = (\lambda_1,...,\lambda_p) \in [0, \infty)^p\) denotes the rates for the categories. We adopt a minimax perspective and investigate testing goodness-of-fit in \(\sup\) norm with respect to some reference rates \(\mu = (\mu_1,...,\mu_p)\). It is assumed without loss of generality \(\mu_1 \geq \mu_2 \geq ... \geq \mu_p > 0\). For \(\varepsilon > 0\), define the space of alternatives \(\Lambda(\mu, \varepsilon) := \left\{ \lambda\in [0, \infty)^p : ||\mu - \lambda||_\infty \geq \varepsilon \right\}\). Formally, the goodness-of-fit testing problem is that of deciding between the hypotheses
    \begin{align}
        H_0 &: \lambda = \mu, \label{problem:poisson0}\\
        H_1 &: \lambda \in \Lambda(\mu, \varepsilon) \label{problem:poisson1}. 
    \end{align}
    The minimax testing risk is \(=\mathcal{R}_{\mathcal{P}}(\varepsilon, \mu) = \inf_{\varphi}\left\{ P_{\mu}\left\{\varphi = 1\right\} + \sup_{\lambda \in \Lambda(\mu, \varepsilon)} P_{\lambda}\left\{\varphi = 0\right\} \right\}\) where the infimum runs over all tests (i.e. binary valued measurable functions with the data \(X\) as input). We are concerned with characterizing the fundamental testing limit.   

    \begin{definition}
        For any \(\eta \in (0, 1)\) and $\mu \in [0,\infty)^p$, the local minimax separation rate for (\ref{problem:poisson0})-(\ref{problem:poisson1}) is 
        \begin{align*}
            \varepsilon^*_{\mathcal{P}}(\mu) = \varepsilon^*_{\mathcal{P}}(\mu,\eta) = \inf \left\{\varepsilon>0 : \mathcal{R}_{\mathcal{P}}(\varepsilon, \mu) \leq \eta\right\}.
        \end{align*}
    \end{definition}
    \noindent The minimax separation rate characterizes, up to universal constants, the difficulty of the testing problem. Notably, the rate is said to be \textit{local} as it depends on the choice of null \(\mu\). It will be seen that the problem is harder for some choices of \(\mu\) and easier for other choices. Establishing the tight dependency of $\varepsilon^*_{\mathcal{P}}(\mu,\eta)$ on $\mu$ up to multiplicative constants depending only on $\eta$ will be a major focus of this paper.

    We also study the related multinomial version of the problem, namely, we consider data 
    \begin{align}\label{model:multinomial}
        X \sim \Multinomial(n, q),
    \end{align}
    where \(q \in \Delta_p := \left\{\pi \in [0, 1]^p : \sum_{j=1}^{p}\pi(j) = 1 \right\}\) denotes the vector of probabilities corresponding to the categories. As in the Poisson setting, we are interested in testing the goodness-of-fit in sup norm with respect to some reference distribution \(q_0 \in \Delta_p\). Without loss of generality, assume \(q_0(1) \geq q_0(2) \geq ... \geq q_0(p) \geq 0\). For \(\varepsilon > 0\), define the space \(\Pi(q_0, \varepsilon) := \left\{q \in \Delta_p : ||q - q_0||_\infty \geq \varepsilon\right\}\) and consider the problem
    \begin{align}
        H_0 &: q = q_0, \label{problem:multinomial0}\\
        H_1 &: q \in \Pi(q_0, \varepsilon). \label{problem:multinomial1}
    \end{align}
    The minimax testing risk and the minimax separation rate can be defined analogously to the definitions given in the Poisson setting. Define \(\mathcal{R}_{\mathcal{M}}(\varepsilon, n, q_0) = \inf_{\varphi}\left\{ P_{q_0}\left\{\varphi = 1\right\} + \sup_{q \in \Pi(q_0, \varepsilon)} P_{q}\left\{\varphi = 0\right\} \right\}\). 

    \begin{definition}
        For any \(\eta \in (0, 1)\), $q_0 \in \Delta_p$, and $n\in\mathbb N$, the local minimax separation rate for (\ref{problem:multinomial0})-(\ref{problem:multinomial1})  is 
        \begin{align*}
            \varepsilon^*_{\mathcal{M}}(q_0) = \varepsilon^*_{\mathcal{M}}(n,q_0,\eta) = \inf \left\{\varepsilon>0: \mathcal{R}_{\mathcal{M}}(\varepsilon, n, q_0) \leq \eta\right\}.
        \end{align*}
    \end{definition}

    Again, emphasis will be placed on characterizing how the local rate $\varepsilon^*_{\mathcal{M}}(n,q_0,\eta)$ tightly depends on $q_0$. It is well known that the multinomial model (\ref{model:multinomial}) is strongly connected to the corresponding Poisson model (\ref{model:poisson}) with rates \(\lambda_j = nq(j)\). The connection can be established through the standard and well-known Poissonization trick \cite{balakrishnan_hypothesis_2019,chhor_sharp_2022,canonne_topics_2022}. However, one cannot immediately derive \(\varepsilon^*_{\mathcal{M}}(q_0)\) from \(\varepsilon^*_{\mathcal{P}}(nq_0)\) due to the shape constraint \(q \in \Delta_p\) in \(\Pi(q_0, \varepsilon)\). This extra geometric structure can be exploited to detect signals of smaller magnitudes, whereas the category rates in the Poisson model (\ref{model:poisson}) need not be related to one another in any way. Hence, \(\varepsilon^*_{\mathcal{M}}\) is worthy of study separate from \(\varepsilon_{\mathcal{P}}^*\).  

    \subsection{Related work}    
    In the multinomial setting, the regime \(n \to \infty\) with \(p = O(1)\) was historically the regime of focus, and the most popular approaches to goodness-of-fit testing were Pearson's \(\chi^2\) test and the likelihood ratio test; extensive theory has been developed as the asymptotic environment is classical. In the high-dimensional multinomial setting, Pearson's test is not rate optimal for detecting \(\ell_2\) separated alternatives. However, a \(\chi^2\)-test without normalization can be shown to achieve the optimal rate \(\sqrt{||q_0^{-\max}||_2/n} + n^{-1}\) where \(q_0^{-\max} \in [0, 1]^{p-1}\) is the vector obtained by removing the largest coordinate of \(q_0\). The argument follows a standard line of reasoning. Perhaps the first non-standard result is Paninski's \cite{paninski_coincidence-based_2008}. Paninski showed that the minimax rate in \(\ell_1\) testing (i.e. total variation) for the multinomial model in the case of uniform null \(q_0 = \left(p^{-1},...,p^{-1}\right)\) is \(1 \land \frac{p^{1/4}}{n}\). Notably, \(\ell_1\) separated alternatives can be successfully detected even when the sample size \(n\) is polynomially smaller than the number of categories. The rate-optimal test relies on the well-known ``birthday paradox'' \cite{paninski_coincidence-based_2008}.
    
    The literature often distinguishes between \textit{global} separation rates---corresponding to the most difficult null distribution within a class---and \textit{local} separation rates. 
    In the multinomial setting, the uniform null is the hardest null distribution for $\ell_1$-separated alternatives, and is associated with a separation rate of order $1 \land \frac{p^{1/4}}{\sqrt{n}}$~\cite{valiant_automatic_2017,chhor_sharp_2022}. 
    In contrast, the easiest null distributions in $\ell_1$ separation are the Dirac distributions of the form $q_0(j) = \mathbbm{1}_{\{j = j_0\}}$ for some $j_0 \in [p]$, whose separation rate is as fast as $\frac{1}{n}$. 
    Consequently, local results can substantially improve global ones and often lead to a much more refined characterization of each null distribution's intrinsic difficulty. 
    In comparison to global results~\cite{hoeffding1965asymptotically, fienberg1979use, batu2000testing, kipnis2023minimax, waggoner2015lp, chan2014optimal}, local results have gained significant interest in recent years~\cite{diakonikolas_new_2016, valiant_automatic_2017, diakonikolas2017near,balakrishnan_hypothesis_2019,chhor_goodness--fit_2023, chhor_sharp_2022,berrett2020locally, dubois2021goodness,lam2022local, blais2019distribution} - see also~\cite{balakrishnan_hypothesis_2018} for an excellent survey.
    
    The paper's main focus will be on deriving local rates, which are generally more involved to obtain than global ones. The case of uniform null is theoretically convenient since the data are homoskedastic under the null; the counts \(\sum_{i=1}^{n} \mathbbm{1}_{\{Y_i  = j\}} \sim \Binomial\left(n, p^{-1}\right)\) all share the same variance. With a different choice of \(q_0 \in \Delta_p\), the data become heteroskedastic under the null, introducing substantial difficulty in establishing minimax rates. Difficulties with heteroskedasticity appears to be a common theme beyond multinomials; indeed, most of the literature on minimax testing in Gaussian models studies the homoscedastic case (for example \cite{ingster_nonparametric_2003,donoho_higher_2004,collier_minimax_2017,spokoiny_adaptive_1996,ingster_multichannel_2003,ingster_detection_2010}). The work addressing the heteroskedastic case is much more limited \cite{laurent_non_2012,chhor2024sparse,hall_innovated_2010}, as is the situation in non-Gaussian heteroskedastic models \cite{ingster_nonparametric_2007,collier_estimating_2018,blanchard2024tight,blanchard2024correlated}.
    
    The first comprehensive local result for \(\ell_1\) testing in the multinomial setting is the seminal result of Valiant and Valiant \cite{valiant_automatic_2017} (see also \cite{balakrishnan_hypothesis_2018,balakrishnan_hypothesis_2019,diakonikolas_new_2016,chhor_sharp_2022}). Their results can be adapted to the Poisson context; following the presentation of \cite{chhor_sharp_2022}, the minimax separation rate they established is
    \begin{equation*}
        \varepsilon^*_{\ell_1}(\mu) \asymp 1 + \sqrt{\left(\sum_{j\leq I} \mu_j^{2/3}\right)^{3/2}} + \sum_{j>I} \mu_j.
    \end{equation*}
    where \(I = \min\left\{ J : \sum_{j > J} \mu_j^2 \leq c\right\}\) for a small constant \(c > 0\). Here, and throughout our discussion elsewhere, we treat \(\eta\) as fixed and, when the context is clear, freely absorb it into the notation \(\asymp\) since our focus in this is on the characterization of the local minimax rate's dependence on the null parameter \(\mu\). Following the terminology of \cite{chhor_sharp_2022}, the index \(I\) can be interpreted as the boundary between the ``bulk'' and ``tail'' portions of \(\{\mu_j\}_{j=1}^{p}\). These monikers are appropriate, since \(\mu_j \leq \sqrt{c}\) for \(j > I\) meaning that the typical value of \(X_j\) is small (i.e. zero or one). A rate-optimal test is a combination of two separate tests for the bulk and tail. A weighted \(\chi^2\) statistic with a specific weight tailored to the \(\ell_1\) norm is used for testing the bulk and a linear statistic is used for testing the tail. 
    
    Chhor and Carpentier \cite{chhor_sharp_2022} generalized this result and characterized the minimax testing rate in \(\ell_t\) norm for all \(t \in [1, 2]\), establishing
    \begin{equation*}
        \varepsilon^*_{\ell_t}(\mu) \asymp 1 + \sqrt{\left(\sum_{j \leq I} \mu_{j}^{\frac{2t}{4-t}}\right)^{\frac{4-t}{2t}}} + \left(\sum_{j > I} \mu_j\right)^{\frac{2-t}{t}}.      
    \end{equation*}
    The bulk and tail contributions depend delicately on the value of \(t\), and a weighted \(\chi^2\) statistic with tailored weights combined with a test designed to detect extreme perturbations was shown to be rate optimal. 

    As Chhor and Carpentier note in their discussion, the case \(t > 2\) is open as the situation appears intrinsically different due to the geometry of the \(\ell_t\) norm placing more emphasis on large perturbations. The geometric effect can be appreciated when comparing \(\ell_1\) to \(\ell_2\). Observe that \(1 + \sqrt{\sum_{j \leq I} \mu_j^2} \asymp 1 + ||\mu||_2\), which is to say that the bulk of \(\mu\) completely determines the \(\ell_2\) separation rate. In contrast, the \(\ell_1\) separation rate involves a nontrivial contribution from the tail; the \(\ell_1\) norm highlights smaller perturbations relatively more than the \(\ell_2\) norm. Moving to the \(\ell_\infty\) norm, the more pronounced emphasis on large perturbations requires a finer understanding of the bulk. As will be established, it will turn out that the decay rate is critical.

    \subsection{Main contributions}
    Our main contributions are the sharp characterizations of the local minimax separation rates in the goodness-of-fit testing problems (\ref{problem:poisson0})-(\ref{problem:poisson1}) and (\ref{problem:multinomial0})-(\ref{problem:multinomial1}). To state the results, define the function \(\Gamma : [0, \infty) \to \R\) with 
    \begin{equation}\label{eqn:Gamma}
        \Gamma(x) = 
        \begin{cases}
            \sqrt{x} &\textit{if } x \leq 1, \\
            \frac{x}{\log(ex)} &\textit{if } x > 1.
        \end{cases}
    \end{equation}
    Results for the two problems are discussed in turn. 
    
    \subsubsection{Testing goodness-of-fit in a Poisson model}\label{section:main_contribution_Poisson}
    Recall it is assumed without loss of generality in the Poisson model (\ref{model:poisson}) that \(\mu_1 \geq \dots \geq \mu_p > 0\). The minimax separation rate for the problem (\ref{problem:poisson0})-(\ref{problem:poisson1}) in the Poisson model (\ref{model:poisson}) will be shown to be
    \begin{equation}\label{rate:poisson}
        \varepsilon_{\mathcal{P}}^*(\mu) \asymp 1 + \max_{1 \leq j \leq p} \mu_j \Gamma\left(\frac{\log(ej)}{\mu_j}\right).
    \end{equation}
    Instead of separate contributions from a ``bulk" and a ``tail" in the rate as identified by \cite{chhor_sharp_2022} in the case of \(\ell_t\) for \(t \in [1, 2]\), the decay rate of the \(\{\mu_{j}\}_{j=1}^{p}\) determines the minimax rate. The involvement of \(\Gamma\) is a direct consequence of the tail of the Poisson distribution and will be commented upon in Section \ref{section:poisson_upper_bound}. Roughly speaking, the \(\Gamma(x) = \sqrt{x}\) behavior for small \(x\) corresponds to the subgaussian part of the tail and the \(\Gamma(x) = \frac{x}{\log(ex)}\) behavior for large \(x\) corresponds to the subpoissonian part. From (\ref{rate:poisson}), the \(j\)th category exhibits a Gaussian contribution when \(\mu_j\) is sufficiently large, which matches the intuition that the \(\Poisson(\lambda)\) distribution is well approximated by \(N(\lambda, \lambda)\) when the rate \(\lambda\) is large. This type of phenomenon has been noted in other Poisson testing problems \cite{arias-castro_sparse_2015,donoho_higher_2022}.

    For example, in the case of large rates \(\mu_j \gtrsim \log p\), the subgaussian part of the tail is always in force and (\ref{rate:poisson}) becomes \(\varepsilon^*(\mu) \asymp \max_{1 \leq j \leq p} \sqrt{\mu_j \log(ej)}\). With large rates, the tail of the \(\Poisson(\mu_j)\) distribution behaves like the tail of the \(N(\mu_j, \mu_j)\) distribution. Indeed, this intuition for the above display is confirmed when recalling the well-known fact \cite{van_handel_spectral_2017,talagrand_upper_2021} that the maximum of the collection of independent \(\left\{N(0, \mu_j)\right\}_{j=1}^{p}\) random variables has the same order as \(\max_{1 \leq j \leq p} \sqrt{\mu_j \log(ej)}\) with high probability. Furthermore, this rate was found to be the minimax separation rate for \(\ell_\infty\) testing in a heteroskedastic Gaussian model \cite{chhor2024sparse}.
    
    As another example, in the case of constant order rates \(\mu_j \asymp 1\), it follows from (\ref{rate:poisson}) that \(\varepsilon^*(\mu) \asymp \frac{\log p}{\log\log p}\). The subpoissonian part of the tail is always in force and determines the rate. Indeed, recall the well-known fact that the maximum of \(p\) independent and identically distributed \(\Poisson(1)\) random variables is almost surely asymptotically equivalent to \(\frac{\log p}{\log\log p}\) as \(p \to \infty\) \cite{rosalsky_remark_1983,rosalsky_acknowledgement_1984,chow_probability_1997}.

    \subsubsection{Testing goodness-of-fit in a multinomial model}\label{section:main_contribution_multinomial}
    We also obtain the minimax separation rate for testing a multinomial distribution (\ref{problem:multinomial0})-(\ref{problem:multinomial1}). Recall it is assumed without loss of generality \(q_0(1) \geq q_0(2) \geq ... \geq q_0(p)\). Furthermore, denote \(q_0^{\max} := \max_{1 \leq j \leq p} q_0(j)\) and \(q_0^{-\max} = (q_0(2),...,q_0(p)) \in [0, 1]^{p-1}\). The minimax rate will be shown to be
    \begin{equation}\label{rate:multinomial}
        \varepsilon_{\mathcal{M}}^*(q_0, n) \asymp \frac{1}{n} + \sqrt{\frac{q_0^{\max}(1-q_0^{\max})}{n}} + \max_{j} q_0^{-\max}(j) \Gamma\left(\frac{\log(ej)}{nq_0^{-\max}(j)}\right). 
    \end{equation}
    It is clear (\ref{rate:multinomial}) is not simply given (after scaling by \(n^{-1}\)) by the Poisson rate (\ref{rate:poisson}) with \(\lambda_j = nq_0(j)\). Despite the tight connection between Poissons and multinomials, the rates are not in correspondence due to the shape constraint \(q \in \Delta_p\) in \(\Pi\); no such shape constraint affects \(\Lambda\). The impact of the simplex geometry is most evident in the extremal case \(q_0 = (1, 0,...,0) \in \Delta_p\), in which case (\ref{rate:multinomial}) is of order \(\frac{1}{n}\) whereas (\ref{rate:poisson}) is of order (after rescaling) \(\frac{1}{\sqrt{n}}\). The simplex's influence on fundamental testing limits has been noted elsewhere \cite{chhor_sharp_2022,bhattacharya_sparse_2024}. 
    
    \subsection{Notation}\label{section:notation}
    The following notation will be used throughout the paper. For \(p \in \mathbb{N}\), let \([p] := \{1,...,p\}\). For \(a, b \in \R\), denote \(a \vee b := \max\{a, b\}\) and \(a \wedge b = \min\{a , b\}\). For any $x \in \R$, define $x_+ = x \lor 0$. Denote \(a \lesssim b\) to mean there exists a universal constant \(C > 0\) such that \(a \leq C b\). The expression \(a \gtrsim b\) means \(b \lesssim a\). Further, \(a \asymp b\) means \(a \lesssim b\) and \(b \lesssim a\). When discussing asymptotics, given real-valued functions \(f\) and \(g\), we say \(f \sim g\) as \(x \to \infty\) if \(\lim_{x \to \infty} \frac{f(x)}{g(x)} = 1\). The same notation is used when taking asymptotics differently, e.g. \(x \to 0\) or along natural numbers; how we pass to the limit is typically clear from the context and thus not explicitly stated. The symbol \(\langle \cdot, \cdot \rangle\) denotes the usual inner product in Euclidean space. For \(v \in \R^p\), we write \(||v||_\infty := \max_{1 \leq j \leq p} |v_j|\). The total variation distance between two probability measures \(P\) and \(Q\) on a measurable space \((\mathcal{X}, \mathcal{A})\) is defined as \(\dTV(P, Q) := \sup_{A \in \mathcal{A}} |P(A) - Q(A)|\). The product measure on the product space is denoted as \(P \otimes Q\). If \(Q\) is absolutely continuous with respect to \(P\), the \(\chi^2\) divergence is defined as \(\chi^2(Q \,||\, P) := \int_{\mathcal{X}} \left(\frac{dQ}{dP} - 1\right)^2 \, dP\).

    \section{Minimax testing rates in the Poisson model}\label{section:poisson}
    In this section, we study the problem (\ref{problem:poisson0})-(\ref{problem:poisson1}) in the Poisson model (\ref{model:poisson}). Define the function \(h : [-1, \infty) \to \R\) with \(h(-1) = 1\) and \(h(x) = (1+x)\log(1+x) - x\). Let \(h^{-1}\) denote the inverse of the function \(h\) restricted to \([0, \infty)\).

    \subsection{Upper bound}\label{section:poisson_upper_bound}
    A natural idea to detect signals separated in \(\ell_\infty\) norm is to examine the maximal deviation from the null. Define the test 
    \begin{equation}\label{test:max}
        \varphi = \mathbbm{1}\left\{||X - \mu||_\infty > \max_{1 \leq j \leq p} \mu_j h^{-1}\left(\frac{\log(C' j^2)}{\mu_j}\right)\right\}.
    \end{equation}
    Here, \(C'\) is a constant to be set to achieve a desired level of testing error. The following theoretical guarantee is available.

    \begin{theorem}\label{thm:upper_bound}
        If \(\eta \in (0, 1)\), there exists \(C',C_\eta > 0\) depending only on \(\eta\) such that
        \begin{equation*}
            P_{\mu}\left\{ \varphi = 1 \right\} + \sup_{\lambda \in \Lambda(\mu, C_\eta \psi)} P_{\lambda}\left\{\varphi = 0\right\} \leq \eta
        \end{equation*}
        where \(\varphi\) is the test given in (\ref{test:max}) and \(\psi = 1 + \max_{1 \leq j \leq p} \mu_j h^{-1}\left(\frac{\log(ej)}{\mu_j}\right).\)
    \end{theorem}

    By Lemma \ref{lemma:h_inverse}, the functions \(h^{-1}\) and \(\Gamma\) are equivalent up to universal constants and so the test \(\varphi\) achieves the rate (\ref{rate:poisson}). The form \(\max_{1 \leq j \leq p} \mu_j h^{-1}\left(\frac{\log(ej)}{\mu_j}\right)\) has a straightforward explanation, and follows from understanding the data's behavior under the null. The function \(h\) appears in the tail of the Poisson distribution; Bennett's inequality (Lemma \ref{lemma:Bennett}) gives the exponential inequality \(P\left\{\left|\Poisson(\rho) - \rho\right| \geq \rho u\right\} \leq 2\exp\left(-\rho h(u)\right)\) for \(\rho, u > 0\). Under the null, union bound gives 
    \begin{align*}
        P_\mu\left\{\varphi = 1\right\} \leq \sum_{j=1}^{p} P_\mu\left\{|X_j - \mu_j| \geq u_j\right\} \leq \sum_{j=1}^{p} 2\exp\left(-\mu_jh\left(h^{-1}\left(\frac{\log(C'j^2)}{\mu_j}\right)\right)\right) = \sum_{j=1}^{p} \frac{2}{C'j^2},
    \end{align*}
    where \(u_{j} = \mu_jh^{-1}\left(\frac{\log(C'j^2)}{\mu_j}\right)\). Since \(\sum_{j=1}^{\infty} \frac{1}{j^2} < \infty\), the constant \(C'\) can be picked to ensure the Type I error is smaller than the desired level. This calculation essentially shows that under the null 
    \begin{equation*}
        ||X - \mu||_{\infty} \lesssim \max_{1 \leq j \leq p} \mu_j \Gamma\left(\frac{\log(ej)}{\mu_j}\right),
    \end{equation*}
    with high probability. Consequently, it is intuitive that a signal with an \(\ell_\infty\) norm of larger order is detectable. In actuality, the signal should have magnitude of at least \(1 + \max_{1 \leq j \leq p} \mu_j \Gamma\left(\frac{\log(ej)}{\mu_j}\right)\). At least constant order signal is necessary even in a simple one-dimensional testing problem, namely \(H_0: Y \sim \Poisson(\rho)\) versus \(H_1 : Y \sim \Poisson(\rho + \delta)\). To see why, consider that if \(\rho\) is very small and \(\delta\) is also small, then \(Y = 0\) with high probability under both the null and alternative; this is a consequence of an intrinsic feature of the Poisson distribution.

    \subsection{Lower bound}\label{section:poisson_lower_bound}
    The rate (\ref{rate:poisson}) has two pieces which we prove separately. It is straightforward to show the constant part of (\ref{rate:poisson}) by a two-point construction. The lower bound argument proceeds by examining the testing problem
    \begin{align}
        H_0 &: \lambda = \mu,\label{eq:poisson_trivial_H0}\\
        H_1 &: \lambda = \mu',\label{eq:poisson_trivial_H1}
    \end{align}
    where \(\mu' = (\mu_1 + c, \mu_2,...,\mu_p)\). Note the separation \(||\mu' - \mu||_\infty = c\) is of constant order. The total variation distance between the distributions \(P_\mu\) and \(P_{\mu'}\) can be explicitly bounded, and it turns out the two hypotheses cannot be separated provided \(c\) is sufficiently small.

    \begin{proposition}\label{prop:constant_lower_bound}
        If \(\eta \in (0, 1)\), there exists \(c_\eta > 0\) depending only on \(\eta\) such that \(\mathcal{R}_{\mathcal{P}}(c_\eta, \mu) \geq \eta\). 
    \end{proposition}
            
    It remains to prove the non-constant part of the lower bound.  Note we can assume it is greater than constant order else there is nothing to prove. 
    \begin{theorem}\label{thm:lower_bound}
        Suppose 
        \begin{equation}\label{eqn:psi_large}
            \max_{1 \leq j \leq p} \mu_j h^{-1}\left(\frac{\log(ej)}{\mu_j}\right) \geq 1. 
        \end{equation}
        If \(\eta \in (0, 1)\), then there exists \(c_\eta > 0\) depending only on \(\eta\) such that 
        \begin{equation*}
            \mathcal{R}_{\mathcal{P}}\left(c_\eta \max_{1 \leq j \leq p} \mu_j h^{-1}\left(\frac{\log(ej)}{\mu_j}\right), \mu \right) \geq \eta.
        \end{equation*} 
    \end{theorem}
    A standard lower bound argument following the minimax testing literature would involve the second moment method. A prior \(\pi\) supported on \(\Lambda\left(\mu, c\max_{1 \leq j \leq p} \mu_j h^{-1}\left(\frac{\log(ej)}{\mu_j}\right)\right)\) would be constructed and the testing risk would be lower bounded by \(1-\dTV(P_\pi, P_\mu)\) where \(P_\pi = \int P_\lambda \, \pi(d\lambda)\) is the mixture induced by \(\pi\). Since \(\dTV(P_\pi, P_\mu) \leq \frac{1}{2}\sqrt{\chi^2(P_\pi || P_\mu)}\), it would suffice to bound the \(\chi^2\) divergence, i.e. the second moment of the likelihood ratio \(\frac{dP_\pi}{dP_\mu}\) under the null. 

    It turns out such an argument would only deliver the subgaussian part \(\max_{1 \leq j \leq p} \sqrt{\mu_j \log(ej)}\) of the lower bound, completely missing the subpoissonian regime. The conditional second moment method is needed to get the subpoissonian part. In the usual unconditional approach, a problematic small probability event can cause the \(\chi^2\) divergence to blow up even though the total variation distance between \(P_\mu\) and \(P_\pi\) is small. As its name suggests, the conditional second moment method involves conditioning on a high probability event to exclude the problematic part of the probability space. In the literature, this truncation strategy is typically used only to pin down the sharp constant as the usual second moment method typically delivers the rate. Notably, this is not the case here. 

    Additionally, a key observation behind our construction of a prior \(\pi\) is that the heteroskedastic problem can be reduced to a homoskedastic problem. To elaborate, let 
    \begin{equation}\label{def:jstar}
        j^* = \arg\max_{1 \leq j \leq p} \mu_j h^{-1}\left(\frac{\log(ej)}{\mu_j}\right). 
    \end{equation}
    Suppose we were faced with data \((Y_1,...,Y_{j^*}) \sim \bigotimes_{j=1}^{j^*} \Poisson(\lambda_j)\) and the testing problem 
    \begin{align}
        H_0 &: \lambda_j = \mu_{j^*} \text{ for all } 1 \leq j \leq j^*, \label{problem:homo0}\\
        H_1 &: \max_{1 \leq j \leq j^*} |\lambda_j - \mu_{j^*}| \geq \varepsilon. \label{problem:homo1}
    \end{align}
    Then our target lower bound would be \(\max_{1 \leq j \leq j^*} \mu_{j^*} h^{-1}\left(\frac{\log(ej)}{\mu_{j^*}}\right) = \mu_{j^*} h^{-1}\left(\frac{\log(ej^*)}{\mu_{j^*}}\right)\), where the equality follows from \(h^{-1}\) being an increasing function. Namely, the homoskedastic problem (\ref{problem:homo0})-(\ref{problem:homo1}) has the same putative rate as our original heteroskedastic problem (\ref{problem:poisson0})-(\ref{problem:poisson1}). 
    
    This correspondence suggests using a prior in the lower bound argument which could essentially be applied to both problems. The following construction implements this intuition. To lower bound the testing risk in (\ref{problem:poisson0})-(\ref{problem:poisson1}), we consider the Bayes testing problem \(H_0: \lambda = \mu\) versus \(H_1 : \lambda \sim \pi\) where \(\pi\) is the prior in which a draw \(\lambda \sim \pi\) is obtained by drawing \(J \sim \Uniform(\left\{1,...,j^*\right\})\) and setting 
    \begin{equation}
        \lambda_j = 
        \begin{cases}
            \mu_j + c\psi &\textit{if } j = J, \\
            \mu_j &\text{otherwise},
        \end{cases}\label{eq_def_prior_poisson}
    \end{equation}
    for \(1 \leq j \leq p\), where 
    \begin{equation}\label{def:Poisson_psi}
        \psi = \mu_{j^*} h^{-1}\left(\frac{\log(Cj^*)}{\mu_{j^*}}\right).
    \end{equation}
    Here, \(C \geq e\) and \(c \geq 0\) are constants to be set. Notably, \(\pi\) places the perturbation only in one of the first \(j^*\) coordinates, just as one would stipulate in a prior construction applicable to the problem (\ref{problem:homo0})-(\ref{problem:homo1}). This prior \(\pi\) is used to prove Theorem \ref{thm:lower_bound}.

    \subsection{Interpretation of the subgaussian and subpoissonian regimes}\label{subsec:intuition_SG_SP_regimes}
    
    We provide here an interpretation of the subgaussian and subpoissonian regimes in the Poisson model. 

    \textit{Subgaussian regime.}
    We recall that this regime is defined by the condition $\mu_{j^*} \geq C \log(ej^*)$ for some sufficiently large constant $C>0$. 
    This condition implies a noteworthy property of the data in this regime. 
    Specifically, with high probability under the null hypothesis $H_0$, all coordinates $j \in [j^*]$ should be observed at least once (meaning $X_j \geq 1$ for all $j \in [j^*]$). 
    To demonstrate this, the probability under $H_0$ that at least one coordinate in $[j^*]$ is unobserved can be bounded as follows
    \begin{align*}
        P_{\mu}\left\{\exists j \in [j^*]: X_j = 0\right\} \leq \sum_{j= 1}^{j^*} P_\mu\left\{X_j = 0\right\} = \sum_{j= 1}^{j^*} e^{-\mu_j} \leq j^* e^{-\mu_{j^*}} \leq \exp(-\mu_{j^*} \!+ \log(j^*)) \leq {(ej^*)}^{-C+1},
    \end{align*}
    which can be made arbitrarily small provided $C$ is sufficiently large. 
    This property offers partial insight as to why the normal approximation is valid in the subgaussian regime; the Poisson distribution most severely deviates from a normal distribution for large values, due to its subexponential tail, and for small values due to the restriction of being nonnegative. The result above demonstrates that, in the subgaussian regime, the nonnegativity constraint never plays a role, as all coordinates are observed with high probability.

    \textit{Subpoissonian regime.} 
    In contrast, the subpoissonian regime is characterized by the condition $\mu_{j^*} \leq c\log(ej^*)$. 
    For some sufficiently small $c>0$, the data exhibits the opposite behavior: with high probability under the null hypothesis $H_0$, at least one coordinate $j \in [j^*]$ is unobserved (i.e. $X_j = 0$ for some $j \in [j^*]$), as Lemma~\ref{lemma:intuition_subpoissonian_regime} below demonstrates. 
    In the subpoissonian regime, the nonnegativity constraint is activated, leading to subpoissonian rather than subgaussian concentration of the data.
    \begin{lemma}\label{lemma:intuition_subpoissonian_regime}
    For any constant $c'>0$, there exists a small enough constant $c>0$ such that the following holds.
        Letting $j^* = \argmax_{1 \leq j \leq p}\mu_{j} h^{-1}\left(\frac{\log(ej)}{\mu_j}\right)$, assume $\mu_{j^*} \leq c\log(ej^*)$ and assume the constant rate does not dominate, that is $\frac{\log(j^*)}{\log(\frac{e\log(j^*)}{\mu_{j^*}})} \geq 1$. Then
        \begin{align*}
            P_{\mu}\left\{\forall j \in [j^*]: X_j \geq 1\right\} \leq c'.
        \end{align*}
            \end{lemma}
    Lemma~\ref{lemma:intuition_subpoissonian_regime} is proved in Appendix~\ref{subsec:intuition_subpoissonian_regime}.

    \subsection{Asymptotic constants}\label{section:poisson_sharp_constant}
    It is possible to pin down the sharp constants in a certain asymptotic setup. Throughout this subsection, we consider asymptotics as \(p \to \infty\). Concretely, we consider sequences of testing problems (\ref{problem:poisson0})-(\ref{problem:poisson1}) indexed by \(p\). Following \cite{arias-castro_sparse_2015}, we consider all sequences \(\left\{\mu^{(p)}\right\}_{p=1}^{\infty}\) in which \(\mu^{(p)} \in [1, \infty)^p\) with \(\mu_1^{(p)} \geq ... \geq \mu_p^{(p)}\). For notational ease, the superscript will be dropped when the context is clear. 
    
    For any sequence \((\alpha_p)_{p = 1}^{\infty}\) with \(\lim_{p \to \infty} \alpha_p = \infty\), denote 
    \begin{equation}
        j^* = \arg\max_{1 \leq j \leq p} \mu_{j} h^{-1}\left( \frac{\log(ej\alpha_p\log^2(ej))}{\mu_j}\right).\label{eq_def_jstar_poisson_sharp_constant}
    \end{equation}
    Set 
    \begin{equation}
        \epsilon = \xi \cdot \max_{1 \leq j \leq p} \mu_{j} h^{-1}\left(\frac{\log(ej\alpha_p\log^2(ej))}{\mu_j}\right)\label{eq_def_eps_poisson_sharp_constant}
    \end{equation}
    where \(\xi > 0\) is a fixed constant which does not change with \(p\).
    
    Note the appearance of the \(\alpha_p\log^2(ej)\) terms marks a difference from what is seen in (\ref{rate:poisson}). 
    As the focus is only on sharp first-order asymptotics, the reader should conceptualize this term as a slowly diverging term that does not affect first-order asymptotics (as stated in Lemma~\ref{lem:removing_logs} below) but is a technical necessity to ensure the existence of a consistent test. In other words, it is analogous to the value \(C'\) set for the significance level of the test (\ref{test:max}); the term \(\alpha_p\log^2(ej)\) diverges in order to require that the testing risk vanishes asymptotically (i.e. consistency). Additionally, the term \(\log^2(ej)\) is not fundamental; rather a sequence \(b_j = o(j)\) such that \(\sum_{j=1}^{\infty} \frac{1}{jb_j} < \infty\) could be used to obtain essentially the same result. 
    The lemma below, applied with $u_{p,j} = \log(\alpha_p \log^2(ej))/\log(ej)$ for all $p,j \in \mathbb N$ ensures that the factor $\alpha_p$ and the extra logarithmic factors do not affect the constant in the rate.
    \begin{lemma}\label{lem:removing_logs}
    Assume that $j^* \to \infty$ and let $(u_{p,j})_{p,j \in \mathbb N}$ be a positive sequence such that $u_{p,j^*} = o(1) $ as $p,j^* \to \infty$. 
Then it holds that 
\begin{align*}
    \max_{1 \leq j \leq p} \mu_j \, h^{-1}\! \left(\frac{\log\big(ej\big)(1 + u_{p,j})}{\mu_j}\right) = (1+o(1)) \max_{1 \leq j \leq p} \mu_j \, h^{-1}\! \left(\frac{\log(ej)}{\mu_j}\right).
\end{align*}
\end{lemma}

\begin{proof}[Proof of Lemma~\ref{lem:removing_logs}]
    
Since the function $h$ is increasing over $[0,\infty)$ and $u_{p,j}>0$ for any $j,p \in \mathbb N$, we have
\begin{align*}
    \max_{1 \leq j \leq p} \mu_j \, h^{-1}\! \left(\frac{\log\big(ej\big)(1 + u_{p,j})}{\mu_j}\right) \geq \max_{1 \leq j \leq p} \mu_j \, h^{-1}\! \left(\frac{\log(ej)}{\mu_j}\right).
\end{align*}
We also have $\log\big(ej^*)(1 + u_{p,j^*}) = (1+o(1))\log\big(e{j^*})$. 
Noting that $h$ has polynomial growth, we obtain 
\begin{align*}
    \mu_{j^*} \, h^{-1}\! \left(\frac{\log\big(e{j^*}\big)(1 + u_{p,j^*})}{\mu_{j^*}}\right) 
    &= (1+o(1))
    \mu_{j^*} \, h^{-1}\! \left(\frac{\log\big(e{j^*}\big)}{\mu_{j^*}}\right) \leq (1+o(1)) \max_{1 \leq j \leq p} \mu_j \, h^{-1}\! \left(\frac{\log(ej)}{\mu_j}\right).
\end{align*}
\end{proof}

The sharp asymptotic constant results in the Poisson model are collected in the theorem below.
   \begin{theorem}\label{thm:poisson_sharp_constant}
        Suppose \(\mu_1 \geq ... \geq \mu_p \geq 1\) and \(\frac{\log j^*}{(\log \alpha_p)(\log\log j^*)} \to \infty\).
        \begin{enumerate}[label=(\roman*)]
            \item Suppose \(\xi > 1\). If \(\frac{\log j^*}{\mu_{j^*}} \to 0\) or \(\frac{\log j^*}{\mu_{j^*}} \to \infty\), then \(\mathcal{R}_{\mathcal{P}}(\epsilon, \mu) \to 0\). 
               
            \item Suppose \(\xi < 1\). If \(\frac{\log j^*}{\mu_{j^*}} \to 0\) or \(\frac{\log j^*}{\mu_{j^*}} \to \infty\), then \(\mathcal{R}_{\mathcal{P}}(\epsilon, \mu) \to 1\).
        \end{enumerate} 
    \end{theorem}
    As one might expect from the Poisson tail, there are essentially two asymptotic regimes, a Gaussian regime and a Poisson regime. The Gaussian regime is in force when the rates are large, that is, when \(\frac{\log j^*}{\mu_{j^*}} \to 0\). In this regime, \(\epsilon \sim \xi \sqrt{2 \mu_{j^*} \log j^*}\), and the constant \(\sqrt{2}\) is natural when we recall the sharp constant in the classical result which asserts the maximum of \(n\) i.i.d. \(N(0, \sigma^2)\) random variables is asymptotically equivalent to \(\sqrt{2\sigma^2\log n}\). Likewise, the Poisson regime is in force when the rates are small, that is, when \(\frac{\log j^*}{\mu_{j^*}} \to \infty\). Here, we have \(\epsilon \sim \xi \frac{\log j^*}{\log(\frac{\log j^*}{\mu_{j^*}})}\), which is intuitive since the maximum of \(n\) i.i.d. \(\Poisson(1)\) random variables is asymptotically equivalent to \(\frac{\log n}{\log\log n}\). Theorem \ref{thm:poisson_sharp_constant} points out that \(\log j^*\) is the boundary between the two regimes. This type of boundary has been noted in a cruder form in \cite{arias-castro_sparse_2015,donoho_higher_2022}. These papers essentially identify a ``high counts'' regime in which \(\min_{1 \leq j \leq p} \mu_j = \omega(\log p)\) and a ``low counts'' regime in which \(\max_{1 \leq j \leq p} \mu_j = o(\log p)\). The boundary \(\log j^*\) provides a finer understanding of the asymptotic regimes since it is entirely a function of the rates and does not exhibit an explicit dimension dependence. 

    The asymptotic condition \(\frac{\log j^*}{(\log \alpha_p)(\log\log j^*)} \to \infty\) is mild and essentially amounts to a growth condition on \(\alpha_p\) ensuring it does not grow too quickly. Furthermore, the condition \(\frac{\log j^*}{(\log \alpha_p)(\log\log j^*)} \to \infty\) automatically implies \(u_{p, j^*} \to 0\) when \(j^* \to \infty\) as \(p \to \infty\).

    \section{Minimax testing rates in the multinomial model}\label{sec:main_results_multinomial}
    In this section, we study the problem (\ref{problem:multinomial0})-(\ref{problem:multinomial1}) in the model (\ref{model:multinomial}). Let \(h\) and \(h^{-1}\) denote the functions defined in Section \ref{section:poisson}. Recall we assume without loss of generality \(q_0(1) \geq q_0(2) \geq ... \geq q_0(p)\). Recall also we denote \(q_0^{\max} := \max_{1 \leq j \leq p} q_0(j) = q_0(1)\) and \(q_0^{-\max} := (q_0(2),...,q_0(p)) \in [0, 1]^{p-1}\). 


    \subsection{Upper bound}
    
    The minimax upper bound relies on a combination of two tests to detect two types of signals. 
    For \(\varepsilon > 0\), define the spaces 
    \begin{align}
        \Pi_1(q_0, \varepsilon) &:= \left\{ q \in \Delta_p : |q(1) - q_0(1)| \geq \varepsilon\right\}, \label{space:P1}\\
        \Pi_2(q_0, \varepsilon) &:= \left\{ q \in \Delta_p : \max_{2 \leq j \leq p} |q(j) - q_0(j)| \geq \varepsilon \right\}. \label{space:P2} 
    \end{align}

    \begin{lemma}\label{lemma:P1P2}
        If \(\varepsilon > 0\), then \(\Pi(q_0, \varepsilon) \subset \Pi_1(q_0, \varepsilon_1) \cup \Pi_2(q_0, \varepsilon_2)\) for any \(\varepsilon_1, \varepsilon_2 \geq 0\) such that \(\varepsilon_1 + \varepsilon_2 \leq \varepsilon\).  
    \end{lemma}
    \begin{proof}
        Fix any \(\varepsilon_1, \varepsilon_2 \geq 0\) such that \(\varepsilon_1 + \varepsilon_2 \leq \varepsilon\). It is immediate \(\Pi(q_0, \varepsilon) \subset \Pi(q_0, \varepsilon_1 + \varepsilon_2)\). Let \(q \in \Pi(q_0, \varepsilon_1 + \varepsilon_2)\) and note \(\varepsilon_1 + \varepsilon_2 \leq ||q-q_0||_\infty \leq |q(1) - q_0(1)| + \max_{2 \leq j \leq p} |q(j) - q_0(j)|\). Therefore, we must have either \(|q(1) - q_0(1)| \geq \varepsilon_1\), in which case \(q \in \Pi_1(q_0, \varepsilon_1)\), or \(\max_{2 \leq j \leq p} |q(j) - q_0(j)| \geq \varepsilon_2\), in which case \(q \in \Pi_2(q_0, \varepsilon_2)\). Thus, \(q \in \Pi_1(q_0, \varepsilon_1) \cup \Pi_2(q_1, \varepsilon_2)\) as claimed. 
    \end{proof}

    To detect signals in \(\Pi_1\), we will use \(X_1\) as the test statistic. Define the test
    \begin{equation}\label{test:multinomial_1}
        \varphi_1 = \mathbbm{1}\left\{|X_1 - nq_0(1)| \geq K_1 \left(1 + \sqrt{nq_0(1)(1-q_0(1))}\right)\right\},
    \end{equation}
    where \(K_1 > 0\) is a constant tuned to achieve a desired error level. 
    
    \begin{proposition}\label{prop:multinomial1}
        If \(\eta \in (0, 1)\), then there exists \(C_\eta > 0\) depending only on \(\eta\) such that 
        \begin{equation*}
            P_{q_0}\left\{\varphi_1 = 1\right\} + \sup_{q \in \Pi_1(q_0, C_\eta \varepsilon_1)} P_{q}\left\{\varphi_1 = 0\right\} \leq \frac{\eta}{2},
        \end{equation*} 
        where \(\varphi_1\) is given by (\ref{test:multinomial_1}) with \(K_1 = \left(\frac{\eta}{4}\right)^{-1/2}\) and \(\varepsilon_1 = \frac{1}{n} + \sqrt{\frac{q_0^{\max}(1-q_0^{\max})}{n}}\).
    \end{proposition}
    
    To detect signals in \(\Pi_2\), we ignore \(X_1\) and directly apply the maximum-type like that from Section \ref{section:poisson_upper_bound} to \(\left\{X_j\right\}_{j=2}^{p}\). Define the test
    \begin{equation}\label{test:multinomial_2}
        \varphi_2 = \mathbbm{1}\left\{\max_{2 \leq j \leq p}|X_j - nq_0(j)| > \max_{2 \leq j \leq p} nq_0(j)(1-q_0(j))h^{-1}\left(\frac{\log(K_2(j-1)^2)}{nq_0(j)(1-q_0(j))}\right)\right\}
    \end{equation}
    where \(K_2 \geq e\) is a constant tuned to achieve a desired error level. 

    \begin{proposition}\label{prop:multinomial2}
        If \(\eta \in (0, 1)\), then there exist \(K_2 \geq e\) and \(C_\eta > 0\) depending only on \(\eta\) such that
        \begin{equation*}
            P_{q_0}\{\varphi_2 = 1\} + \sup_{q \in \Pi_2(q_0, C_\eta \varepsilon_2)} P_{q} \left\{\varphi_2 = 0\right\} \leq \frac{\eta}{2},
        \end{equation*}
        where \(\varphi_2\) is given by (\ref{test:multinomial_2}) and \(\varepsilon_2 = \frac{1}{n} + \max_j q_0^{-\max}(j)(1-q_0^{-\max}(j))h^{-1}\left(\frac{\log(ej)}{nq_0^{-\max}(j)(1-q_0^{-\max}(j))}\right)\). 
    \end{proposition}
    The tests \(\varphi_1\) and \(\varphi_2\) are aggregated to produce a test for detecting signals in \(\Pi\). Define 
    \begin{equation}\label{test:full}
        \varphi = \varphi_1 \vee \varphi_2. 
    \end{equation}
    The following theorem, which is stated without proof, is an immediate consequence of Lemma \ref{lemma:P1P2} along with Propositions \ref{prop:multinomial1} and \ref{prop:multinomial2}.
    \begin{theorem}\label{thm:multinomial_upper_bound}
        If \(\eta \in (0, 1)\), then there exist \(K_1, C_\eta > 0\) and \(K_2 \geq e\) depending only on \(\eta\) such that 
        \begin{equation*}
            P_{q_0}\{\varphi = 1\} + \sup_{q \in \Pi(q_0, C_\eta\varepsilon)} P_q\left\{ \varphi = 0\right\} \leq \eta, 
        \end{equation*}
        where \(\varphi\) is given by (\ref{test:full}) and 
        \begin{equation*}
            \varepsilon = \frac{1}{n} + \sqrt{\frac{q_0^{\max}(1-q_0^{\max})}{n}} + \max_j q_0^{-\max}(j)(1-q_0^{-\max}(j))h^{-1}\left(\frac{\log(ej)}{nq_0^{-\max}(j)(1-q_0^{-\max}(j))}\right). 
        \end{equation*} 
    \end{theorem}
    Note since it has been assumed without loss of generality \(q_0(1) \geq q_0(2) \geq ... \geq q_0(p)\), that it must be the case \(\frac{1}{2} \geq \max_j q_0^{-\max}(j)\). Therefore, \(q_0^{-\max}(j)(1-q_0^{-\max}(j)) \asymp q_0^{-\max}(j)\) for all \(2 \leq j \leq p\). Hence, Theorem \ref{thm:multinomial_upper_bound} indeed asserts \(\varphi\) achieves the rate (\ref{rate:multinomial}). 
   
    \subsection{Lower bound}\label{section:multinomial_lowerbound}

    We now prove the lower bound on $\varepsilon_{\mathcal{M}}^*(q_0,n,\eta)$. 
    To do so, we will work under a Poissonized model where the data \(X \in \mathbb{Z}^p\) are given by Poisson sampling.
    For a probability distribution \(q \in \Delta_{p}\) on \(p\) categories, consider 
    \begin{align}
    \begin{split}\label{model:poissonized}
        N &\sim \Poisson(n), \\
        X \,|\, N &\sim \Multinomial(N, q). 
    \end{split}
    \end{align}
    Consequently, the marginal distribution of the data \(X\) is 
    \begin{equation}\label{eqn:poissonized_multinomial_dgp}
        X \sim \bigotimes_{j=1}^{p} \Poisson(nq(j)). 
    \end{equation}
    The probability distribution under the model~\eqref{eqn:poissonized_multinomial_dgp} will be denoted as $\mathbf P_q$. The minimax testing risk for problem (\ref{problem:multinomial0})-(\ref{problem:multinomial1}) in the model~\eqref{eqn:poissonized_multinomial_dgp} is defined as
    \begin{equation}\label{def:poissonized_multinomial_testing_risk}
        \mathcal{R}_{\mathcal{PM}}(\varepsilon,n, q_0) = \inf_{\varphi}\left\{ \mathbf{P}_{q_0}\left\{\varphi = 1\right\} + \sup_{q \in \Pi(q_0, \varepsilon)} \mathbf{P}_{q}\left\{\varphi = 0\right\} \right\},
    \end{equation}
    and the corresponding minimax separation rate is 
        \begin{align}
        \varepsilon_{\mathcal{PM}}^*(q_0,n,\eta) = \inf \left\{\varepsilon>0: \mathcal{R}_{\mathcal{PM}}(\varepsilon, n, q_0) \leq \eta\right\}.\label{def_poissonized_multinomial_rate}
    \end{align}
    The subscript $\mathcal{PM}$ stands for ``Poissonized multinomial''. The lemma below shows that the Poissonized rate~\eqref{def_poissonized_multinomial_rate} can be used to obtain a lower bound on the quantity of interest $\varepsilon_{\mathcal{M}}^*(q_0,n,\eta)$ provided $n$ is larger than a suitable constant depending on $\eta$. 

    \begin{lemma}\label{lem:relation_multinomial_poisson}
        If \(\varepsilon > 0\), then for any $c>0$, it holds that \(\mathcal{R}_{\mathcal{M}}(\varepsilon, n, q_0) \geq \mathcal{R}_{\mathcal{PM}}(\varepsilon, (1+c)n, q_0) - \frac{2(1+c)}{c^2n}\). 
    \end{lemma}
    Indeed, applying Lemma~\ref{lem:relation_multinomial_poisson} with $c = 1$, we get that for any constant $\delta \in (0,1-\eta)$ and $ n \geq 4/\delta$, and for any $\varepsilon < \varepsilon_{\mathcal{PM}}^*(q_0,2n,\eta+\delta)$, 
    \begin{align*}
        \mathcal{R}_{\mathcal{M}}(\varepsilon, n, q_0) \geq \mathcal{R}_{\mathcal{PM}}(\varepsilon, 2n, q_0) - \frac{4}{n}> \eta,
    \end{align*}
    and so $\varepsilon < \varepsilon_{\mathcal{M}}^*(q_0,n,\eta)$. 
    Since this is true for any $\varepsilon < \varepsilon_{\mathcal{PM}}^*(q_0,2n,\eta+\delta)$, we obtain the lower bound $\varepsilon_{\mathcal{M}}^*(q_0,n,\eta) \geq \varepsilon_{\mathcal{PM}}^*(q_0,2n,\eta+\delta)$ provided $n \geq 4/\delta$. 
    Note that $(1+c)n$ need not be an integer in the Poissonized model~\eqref{model:poissonized}.
    This fact being established, we now proceed by bounding below the Poissonized testing rate $\varepsilon_{\mathcal{PM}}^*\left(q_0,n,\eta\right)$ to obtain a lower bound on $\varepsilon_{\mathcal{M}}^*(q_0,n,\eta)$. 
    We recall that the rate we are aiming for is
    \begin{align*}
        \varepsilon_{\mathcal{M}}^*(q_0, n) \gtrsim \frac{1}{n} + \sqrt{\frac{q_0^{\max}(1-q_0^{\max})}{n}} + \max_{j} q_0^{-\max}(j) \Gamma\left(\frac{\log(ej)}{nq_0^{-\max}(j)}\right). 
    \end{align*}
    This rate contains three parts that are analyzed separately. We note that the sum of the first and third parts is analogous to the Poisson rate~\eqref{rate:poisson} after rescaling by $\frac{1}{n}$.
    
    \subsubsection{\texorpdfstring{Prior construction for the $1/n$ term}{Prior construction for the 1/n rate}}
    The $\frac{1}{n}$ term in the above rate is proved by analyzing the two-point testing problem $H_0: X \sim \mathbf P_{q_0}$ versus $H_1: X \sim \mathbf P_{q_1}$ where
    \begin{align*}
        q_1(j) := \left(1 - \frac{2c_\eta}{n}\right)q_0(j) + \frac{2c_\eta}{n}\mathbbm{1}_{\{j = 2\}}, \quad \forall j \in [p].
    \end{align*}
    The probability vector $q_1$ is essentially analogous to the vector $\mu'$ used in the alternative hypothesis from the problem~\eqref{eq:poisson_trivial_H0}-\eqref{eq:poisson_trivial_H1}, rescaled by a suitable constant to lie within the simplex. 
    The proposition below provides a lower bound of order $\frac{1}{n}$ using this construction.

    \begin{proposition}\label{prop:1/n_rate}
        If \(\eta \in (0, 1)\), then there exists \(c_\eta > 0\) depending only on \(\eta\) such that \(\mathcal{R}_{\mathcal{M}}\left(\frac{c_\eta}{n}, n, q_0\right) \geq \eta\). 
    \end{proposition}
    \noindent Note that, here, we obtained the desired lower bound on the multinomial separation rate $\varepsilon_{\mathcal{M}}^*(n,q_0,\eta)$ directly.

    \subsubsection{Prior construction for the parametric rate}\label{section:prior_construction_parametric}

    Since a lower bound of order $\frac{1}{n}$ has been derived in Proposition~\ref{prop:1/n_rate}, we will assume from now on that the $\frac{1}{n}$ term does not dominate in the rate~\eqref{rate:multinomial}.
    To establish the term \(\sqrt{\frac{q_0^{\max}(1-q_0^{\max})}{n}}\) in the lower bound (\ref{rate:multinomial}), it suffices to establish a lower bound of order \(q_0^{\max} \wedge \sqrt{\frac{q_0^{\max}(1-q_0^{\max})}{n}}\) since, up to universal constants, we have 
    \begin{align*}
        q_0^{\max} \wedge \sqrt{\frac{q_0^{\max}(1-q_0^{\max})}{n}} + \frac{1}{n} \asymp \sqrt{\frac{q_0^{\max}(1-q_0^{\max})}{n}} + \frac{1}{n}.
    \end{align*}

    The parametric rate $q_0^{\max} \wedge \sqrt{\frac{q_0^{\max}(1-q_0^{\max})}{n}} $ is proved by analyzing the two-point testing problem $H_0: X \sim \mathbf P_{q_0}$ versus $H_1: X \sim \mathbf P_{q_1'}$ where
    \begin{equation*}
        q_1'(j) = 
        \begin{cases}
            q_0(1) - c_\eta\epsilon &\textit{if } j = 1, \\
            q_0(j) \left(1 + \frac{c_\eta\epsilon}{1-q_0(1)}\right) &\textit{if } j \geq 2. 
        \end{cases}
    \end{equation*}
and \(\epsilon = q_0(1) \wedge \sqrt{\frac{q_0(1)(1-q_0(1))}{n}}\) for some sufficiently small constant \(c_\eta \in [0, 1]\). 
Using this reduction, the parametric rate is obtained in the proposition below.
    
    \begin{proposition}\label{prop:multinomial_parametric}
        If \(\eta \in (0, 1)\), then there exists \(c_\eta > 0\) depending only on \(\eta\) such that 
        \begin{equation*}
            \mathcal{R}_{\mathcal{PM}}\left( c_\eta \left(q_0^{\max} \wedge \sqrt{\frac{q_0^{\max}(1-q_0^{\max})}{n}}\right), n, q_0\right) \geq \eta.
        \end{equation*}
    \end{proposition}

    \subsubsection{\texorpdfstring{Prior construction for the term $\max_{j} q_0^{-\max}(j) \Gamma\left(\frac{\log(ej)}{nq_0^{-\max}(j)}\right)$}{Prior construction for the local term}}\label{section:prior_construction_local}
   
    We now address the remaining term $\max_{j} q_0^{-\max}(j) \Gamma\left(\frac{\log(ej)}{nq_0^{-\max}(j)}\right)$. 
    Although this rate is comparable to the Poisson rate from Theorem~\ref{thm:lower_bound}, after rescaling by $\frac{1}{n}$, the proof is more involved because our prior must be supported in the simplex. 
    In particular, any perturbation added to a coordinate must be offset by removing a corresponding mass amount from other coordinates. 
    To better appreciate why the multinomial prior construction does not directly follow from the Poisson prior construction~\eqref{eq_def_prior_poisson}, assume $q_0$ is the uniform distribution over $p$ coordinates for conceptual clarity. 
    The claimed minimax rate is given by
    \begin{align*}
        \frac{1}{n} + \sqrt{\frac{p^{-1}(1-p^{-1})}{n}} + \frac{1}{p} h^{-1}\left(\frac{p\log(p)}{n}\right)\asymp \begin{cases}
            \sqrt{\frac{\log(p)}{np}}& \text{ if } n \geq p\log(p),\\
            \frac{\log(p)/n}{\log\left(\frac{p\log(p)}{n}\right)} & \text{ if } n < p\log(p). 
        \end{cases}
    \end{align*}
    In the subgaussian regime where $n \geq p\log(p)$, the rate satisfies $\sqrt{\frac{\log(p)}{np}} \leq \frac{1}{p}$. 
    This implies that one can select two coordinates uniformly at random and apply a perturbation of order $\sqrt{\frac{\log(p)}{np}}$ to the first one while reducing the second one by the same amount, without causing any coordinate to become negative.  
    Notably, this straightforward construction is no longer achievable in the subpoissonian regime where $n \ll p\log(p)$: to enforce the simplex constraint, a perturbation of order $\frac{\log(p)/n}{\log\left(\frac{p\log(p)}{n}\right)} \gg \frac{1}{p}$ added to a coordinate must be compensated for by reducing $m\asymp \frac{p\log(p)/n}{\log\left(\frac{p\log(p)}{n}\right)} \gg 1$ coordinates.   
    Consequently, the main technical challenge is to ensure that such a perturbation of multiple coordinates remains indistinguishable from the null hypothesis (Lemmas~\ref{lemma:multinomial_conditional} and \ref{lemma:mgf_bound}) while simultaneously achieving a rate analogous to that in Theorem~\ref{thm:lower_bound}. 

    Formally, our prior distribution \(\pi\) is defined as follows. 
    Let 
        \begin{align}
            &j^* = \argmax_{j} nq_0^{-\max}(j) h^{-1}\left(\frac{\log(ej)}{nq_0^{-\max}(j)}\right),\label{def_j^*_mult}\\
            &\psi = \max_{j} nq_0^{-\max}(j) h^{-1}\left(\frac{\log(ej)}{nq_0^{-\max}(j)}\right),\\
            &m = \left\lceil h^{-1}\left(\frac{\log(ej^*)}{nq_0^{-\max}(j^*)}\right)\right\rceil \wedge (j^* - 1). \label{def_m_mult}
        \end{align}

    The integer $m$ represents the number of coordinates to be decreased.  
    A draw \(q \sim \pi\) is obtained by first drawing \(J \sim \Uniform(\{2,...,j^*+1\})\), then drawing uniformly at random a size-\(m\) subset \(\mathcal{I} \subset \{2,...,j^*+1\} \setminus \left\{J\right\}\), and finally setting 
    \begin{equation}
        q(j) = 
        \begin{cases}
            q_0(j) + c\frac{\psi}{n}   &\textit{if } j = J,\\
            q_0(j) - c\frac{\psi}{nm} &\textit{if } j \in \mathcal{I}, \\
            q_0(j) &\textit{otherwise},
        \end{cases}
        \label{eq_non-param_prior_mult}
    \end{equation}
    for \(1 \leq j \leq p\). 
    If $m=0$, then $\mathcal{I}$ is empty and we have $q = q_0$. Note that the first coordinate is never perturbed, i.e. \(q(1) = q_0(1)\). 
    The desired lower bound based on the prior construction~\eqref{eq_non-param_prior_mult} is established by combining Theorem~\ref{thm:poissonized_multinomial_lower_bound} and Lemma~\ref{lem:m=0} below.
    
    \begin{theorem}\label{thm:poissonized_multinomial_lower_bound}
        Let $j^*$ and $m$ be defined as in~\eqref{def_j^*_mult} and \eqref{def_m_mult}, respectively. There exists a large universal constant \(C_* > 0\) such that the following holds. If 
        \begin{equation}\label{eqn:multinomial_psi_large}
            \max_{j} q_0^{-\max}(j) h^{-1}\left(\frac{\log(ej)}{nq_0^{-\max}(j)}\right) \geq \frac{C_*}{n} 
        \end{equation}
        and \(\eta \in (0, 1)\), then there exists \(\tilde{C}_\eta \geq e\) and \(c_\eta > 0\) depending only on \(\eta\) such that for \(0 < c < c_\eta\) we have 
        \begin{equation*}
            \mathcal{R}_{\mathcal{PM}}\left(\frac{c\psi}{n}, n, q_0\right) \geq \eta,
        \end{equation*}
        where  
        \begin{equation*}
            \psi = nq_0^{-\max}(j^*) h^{-1}\left(\frac{\log(\tilde{C}_\eta j^*)}{nq_0^{-\max}(j^*)}\right)\mathbbm{1}_{\{m \geq 1\}}. 
        \end{equation*}
    \end{theorem}

    The following lemma shows that if \(\psi = 0\) due to \(m = 0\) (which can only happen when \(j^* = 1\)) in Theorem \ref{thm:poissonized_multinomial_lower_bound}, then we must be in the regime such that the other terms dominate in the rate, i.e. \(\varepsilon^*_{\mathcal{PM}}(q_0, n) \asymp \frac{1}{n} + \sqrt{\frac{q_0^{\max}(1-q_0^{\max})}{n}}\). 

    \begin{lemma}\label{lem:m=0}
        If \(j^* = 1\), then \(q_0^{-\max}(j^*) \Gamma\left(\frac{\log(ej^*)}{nq_0^{-\max}(j^*)}\right) \lesssim \sqrt{\frac{q_0^{\max}(1-q_0^{\max})}{n}} + \frac{1}{n}\). 
    \end{lemma}
    \begin{proof}
        If \(j^* = 1\), then 
        \begin{equation*}
            q_0^{-\max}(j^*) \Gamma\left(\frac{\log(ej^*)}{nq_0^{-\max}(j^*)}\right) = q_0(2) \Gamma\left(\frac{1}{nq_0(2)}\right) = 
            \begin{cases}
                \sqrt{\frac{q_0(2)}{n}} &\textit{if } n q_0(2) \geq 1, \\
                \frac{1}{n} \frac{1}{\log\left(\frac{e}{nq_0(2)}\right)} &\textit{if } nq_0(2) < 1.  
            \end{cases}
        \end{equation*}
        Consider \(q_0(2) \leq q_0(1)\) by our ordering assumption (which is made without loss of generality). Further consider \(q_0(2) \leq 1-q_0(1)\). Therefore, \(q_0(2) \leq q_0(1) \wedge (1-q_0(1)) \asymp q_0^{\max}(1-q_0^{\max})\). The desired result follows immediately. 
    \end{proof}

        The prior \(\pi\) involves removing probability mass from \(m\)-many coordinates of \(nq_0\) to compensate for the mass added to the coordinate \(J\). From the lower bound perspective, it must be argued that not only is the addition of mass undetectable, but also the removal of mass \(\frac{c\psi}{m}\) from \(m\)-many coordinates is undetectable. This latter point is related to lower bound arguments in sparse signal detection with sparsity level \(m = |\mathcal{I}|\). 

        To briefly review known results, consider \(Y \sim N(\theta, \sigma^2 I_d)\) and the testing problem \(H_0 : \theta = 0\) versus \(H_1 : ||\theta|| \geq \rho, ||\theta||_0 \leq m\). The minimax separation rate was shown in \cite{collier_minimax_2017} to be \((\rho^*)^2 \asymp \sigma^2 m \log\left(1 + \frac{d}{m^2}\right)\). The minimax lower bound involves the prior \(\nu\) where a draw \(\theta \sim \nu\) is obtained by drawing a uniformly at random size \(m\) subset \(S\) and setting \(\theta_j^2 \asymp \sigma^2\log\left(1 + \frac{d}{m^2}\right)\mathbbm{1}_{\{j \in S\}}\). This level of perturbation on \(m\) (uniformly) random coordinates was shown to be undetectable. 
        
         Returning to our Poissonized multinomial setting, consider that the noise level of the flattened, homoskedastic null is \(nq_0^{-\max}(j^*)\). Relying on intuition from the Gaussian model, the mass removal defined by \(\pi\) should be intuitively undetectable if \(\frac{\psi^2}{m^2} \lesssim nq_0^{-\max}(j^*) \log\left(1 + \frac{j^*}{m^2}\right)\). Our choice of \(\psi\) and \(m\) basically implies the condition essentially boils down to 
        \begin{equation*}
            nq_0^{-\max}(j^*) \lesssim \log\left(1 + \frac{j^*}{m^2}\right). 
        \end{equation*}
        It is not \textit{a priori} clear that our definition of \(j^*\) guarantees this condition is satisfied. To see why it turns out the condition is satisfied, there are essentially two regimes to understand: the subgaussian regime \(nq_0^{-\max}(j^*) \gtrsim \log(ej^*)\) and the subpoissonian regime \(nq_0^{-\max}(j^*) \lesssim \log(ej^*)\). There is nothing to argue if \(m = 0\), so suppose \(m \geq 1\). Lemma \ref{lemma:m_size} asserts \(m \lesssim (j^*)^{1/4}\), and so \(\log\left(1 + \frac{j^*}{m^2}\right) \asymp \log(ej^*)\); the undetectability condition is thus satisfied in the subpoissonian regime. In the subgaussian regime, it follows from \(h^{-1}(x) \asymp \sqrt{x}\) for \(x \lesssim 1\) that \(1 \leq m^2 \lesssim \frac{\log(ej^*)}{nq_0^{-\max}(j^*)}\). In other words, we actually have \(nq_0^{-\max}(j^*) \asymp \log(ej^*)\) and so the undetectability condition is also satisfied. 

    All of the lower bounds proved thus far can be combined directly. The following corollary formally states the desired minimax lower bound. 
    \begin{corollary}
        If \(\eta \in (0, 1)\), then there exists a constant \(c_\eta > 0\) depending only on \(\eta\) such that
        \begin{equation*}
            \mathcal{R}_{\mathcal{M}}\left(c_\eta\left( \frac{1}{n} + \sqrt{\frac{q_0^{\max}(1-q_0^{\max})}{n}} + \frac{\psi}{n}\right), n, q_0\right) \geq \eta. 
        \end{equation*}
    \end{corollary}
    
    \begin{proof}
        Fix \(\eta \in (0, 1)\) and take \(n_0\) such that \(\frac{8}{n_0} \leq 1-\eta\). 
        If \(n \leq n_0\), then Proposition 5 delivers \(\varepsilon_{\mathcal{M}}^*(n,q_0,\eta) \geq \frac{c_\eta}{n} \geq c'_\eta\), where $c'_\eta$ depends only on $\eta$. 
        Moreover, it is easy to see that 
        $$\sqrt{\frac{q_0^{\max}(1-q_0^{\max})}{n}} + \frac{\psi}{n} \leq \log(\tilde C_\eta)$$
        by noting that $x \mapsto x h^{-1}(\log(\tilde C_\eta j^*)/x)$ is an increasing function for $x >0$, and that $ q_0^{-\max}(j^*)\leq \frac{1}{j^*}\sum_{j\leq j^*} q^{-\max}(j)  \leq 1/j^*$.
        Therefore, it holds that $\varepsilon \asymp \frac{1}{n}$ in this regime.
        
        Assume now that \(n \geq n_0\). 
        Then by Lemma~\ref{lem:relation_multinomial_poisson}, we have \(\mathcal{R}_{\mathcal{M}}(c_\eta\varepsilon, n, q_0) \geq \mathcal{R}_{\mathcal{PM}}(c_\eta\varepsilon, 2n, q_0) - \frac{1-\eta}{2}\). 
        By Propositions~\ref{prop:1/n_rate}, \ref{prop:multinomial_parametric}, Lemma~\ref{lem:m=0} and Theorem~\ref{thm:poissonized_multinomial_lower_bound}, we can choose the constants $C_*$ and $\tilde C_\eta$ large enough that 
        $$ \mathcal{R}_{\mathcal{PM}}\left( c_\eta\left(\frac{1}{n}+\sqrt{\frac{q_0^{\max}(1-q_0^{\max})}{n}} + \frac{\psi}{n}\right), n, q_0\right) \geq \frac{1+\eta}{2},$$
        which yields $\mathcal{R}_{\mathcal{M}}(c_\eta\varepsilon, n, q_0) \geq \eta$.
    \end{proof}

    \subsection{Asymptotic constants}\label{section:multinomial_sharp_constant}
    Similar to the Poisson setting, we are able to establish the sharp constant in a certain asymptotic setup. Throughout this subsection, we consider \(p\) as a function of \(n\), and we consider asymptotics as \(n \to \infty\). Concretely, we consider sequences of testing problems where the null \(q_0\) changes with \(n\). Furthermore, throughout this section we will assume \(q_0(1) \geq ... \geq q_0(p) \geq \frac{1}{n}\) for all \(n\). 
    
    For any sequence \(\{\alpha_p\}_{p=1}^{\infty}\) with \(\alpha_p \to \infty\) as \(n \to \infty\) and for $n' = (1+c_n)n$ where $c_n = n^{-1/3}$, denote 
    \begin{equation}
        j^* = \argmax_{j} q_0^{-\max}(j)(1-q_0^{-\max}(j)) h^{-1}\left( \frac{\log(ej\alpha_p \log^2(ej))}{n'q_0^{-\max}(j)(1-q_0^{-\max}(j))} \right). \label{eq_def_jstar_mult_sharp_constant}
    \end{equation}
    Define 
    \begin{equation}
        \epsilon = \xi \cdot \max_{j} q_0^{-\max}(j)(1-q_0^{-\max}) h^{-1}\left(\frac{\log(ej)}{n'q_0^{-\max}(j)(1-q_0^{-\max}(j))} \right),\label{eq_def_eps_mult_sharp_constant}
    \end{equation}
    where \(\xi > 0\) is a fixed constant which does not change with \(n\). 
    \begin{theorem}\label{thm:multinomial_sharp_constant}
        Suppose \(q_0(1) \geq ... \geq q_0(p) \geq \frac{1}{n}\), \(\frac{\log j^*}{(\log \alpha_p)(\log\log j^*)} \to \infty\), $n \to \infty$, and \(\frac{\epsilon}{\sqrt{\frac{q_0^{\max}(1-q_0^{\max})}{n}}} \to \infty\).
        
        \begin{enumerate}[label=(\roman*)]
            \item Suppose \(\xi > 1\). If \(\frac{\log j^*}{nq_0^{-\max}(j^*)} \to 0\) or \(\frac{\log j^*}{nq_0^{-\max}(j^*)} \to \infty\), then \(\mathcal{R}_{\mathcal{M}}(\epsilon, n, q_0) \to 0\).
            
            \item Suppose \(\xi < 1\). If \(\frac{\log j^*}{nq_0^{-\max}(j^*)} \to 0\) or \(\frac{\log j^*}{nq_0^{-\max}(j^*)} \to \infty\), then \(\mathcal{R}_{\mathcal{M}}(\epsilon, n, q_0) \to 1\).
        \end{enumerate}
    \end{theorem}
   The condition \(\frac{\epsilon}{\sqrt{\frac{q_0^{\max}(1-q_0^{\max})}{n}}} \to \infty\) is, in some sense, necessary for a phase transition phenomenon to occur. When the parametric rate \(\sqrt{\frac{q_0^{\max}(1-q_0^{\max})}{n}} + \frac{1}{n}\) dominates in the rate, the asymptotic testing risk behaves classically. To elaborate, if the alternative hypothesis has separation \(\epsilon = \xi\left(\sqrt{\frac{q_0^{\max}(1-q_0^{\max})}{n}} + \frac{1}{n}\right)\), then there does not exist a detection boundary \(\xi^*\) such that the testing risk goes to \(0\) for \(\xi > \xi^*\) and goes to \(1\) for \(\xi < \xi^*\). Rather, the testing risk tends to some nontrivial quantity \(\beta(\xi) \in (0, 1)\). The absence of a phase transition is typical for hypothesis testing in parametric models (e.g. Gaussian location model), and is thus said to be classical.

    \section{Discussion}
    
    Testing signals with \(\ell_\infty\) separation is well known to be closely connected to testing very sparse signals in other separation metrics. Consequently, the related sparse signal detection literature is worth discussing.

    \subsection{Sparse Poisson mixture detection}
    Sparse testing in a Poisson model has been studied by Arias-Castro and Wang \cite{arias-castro_sparse_2015} who consider data \(X_1,...,X_p\) and the testing problem  
    \begin{align*}
        H_0 &: X_j \overset{ind}{\sim} \Poisson(\mu_j), \\
        H_1 &: X_j \overset{ind}{\sim} (1-\epsilon) \Poisson(\mu_j) + \frac{\epsilon}{2} \Poisson(\mu_j') + \frac{\epsilon}{2} \Poisson(\mu_j''). 
    \end{align*}
    It is assumed \(\min_{1 \leq j \leq p} \mu_j \geq 1\) for all \(p\). Adopting an asymptotic setup with \(p \to \infty\), they are interested in fundamental limits in the sparse setting in which the proportion of signals asymptotically vanishes but there are a nontrivial total number of signals. With these two desiderata, they assume \(\epsilon \to 0\) and \(p\epsilon \to \infty\).

    For a variety of models, the sparse mixture detection literature has studied fundamental limits in this setting. A long line of work has delivered sharp constants and subtle phase transitions \cite{ingster_problems_1997,donoho_higher_2004,donoho_higher_2015,donoho_higher_2022,arias-castro_sparse_2015,cai_optimal_2014,cai_optimal_2011,ditzhaus_signal_2019,kotekal_statistical_2022}. Following in this tradition, Arias-Castro and Wang parametrize \(\epsilon = p^{-\beta}\) for \(\beta \in \left(\frac{1}{2}, 1\right)\). Further, they consider two separate asymptotic regimes. In the regime \(\min_{1 \leq j \leq p} \mu_j = \omega(\log p)\), they parametrize \(\mu_j', \mu_j'' = \mu_j \pm \sqrt{2r\mu_j \log p}\) with \(r \in (0, 1)\). They derive the constant-sharp detection boundary
    \begin{equation*}
        \rho(\beta) = 
        \begin{cases}
            \beta - \frac{1}{2} &\textit{if } \frac{1}{2} < \beta \leq \frac{3}{4}, \\
            (1-\sqrt{1-\beta})^2 &\textit{if } \frac{3}{4} < \beta < 1. 
        \end{cases}
    \end{equation*}
    In other words, \(H_0\) and \(H_1\) separate asymptotically if \(r > \rho(\beta)\) and merge asymptotically if \(r < \rho(\beta)\). The detection boundary \(\rho(\beta)\) is the same detection boundary appearing in the analogous Gaussian version of the sparse mixture detection problem \cite{ingster_problems_1997,donoho_higher_2004}. Indeed, in this regime the distributions \(\Poisson(\mu_j)\) can be essentially approximated by the distributions \(N(\mu_j, \mu_j)\). 
    The problem of \(\ell_\infty\) testing we consider essentially corresponds to the case \(p\epsilon = O(1)\), which is not handled in \cite{arias-castro_sparse_2015} (nor in \cite{donoho_higher_2022} which investigates a two-sample version of this model). In this regime \(\min_{1 \leq j \leq p} \mu_j = \omega(\log p)\), our result Theorem \ref{thm:poisson_sharp_constant} asserts the sharp detection boundary is given by \(\xi = 1\), which exactly agrees with \(\sqrt{\rho(1)}\). Though the formal technical conditions of \cite{arias-castro_sparse_2015} do not cover \(p\epsilon = O(1)\), the detection boundary nevertheless agrees with their theoretical prediction.
    
    Arias-Castro and Wang \cite{arias-castro_sparse_2015} also study the regime \(\max_{1 \leq j \leq p} \mu_j = o(\log p)\) and parametrize \(\mu_j' = \mu_j^{1-\gamma} (\log p)^\gamma\) where \(\gamma \in (0, 1)\) and \(\mu_j''= 0\). They establish that \(H_0\) and \(H_1\) asymptotically separate if \(\gamma > \beta\) and asymptotically merge if \(\gamma < \beta\). In this regime, a Gaussian approximation is poor and thus the testing limits are different. 
    By adopting this narrow parametrization, it is difficult to interpret their result as a constant-sharp statement about a detection boundary. Furthermore, their result is too coarse to capture the subtle logarithmic effects found in Theorem \ref{thm:poisson_sharp_constant}. Specifically, the iterated log behavior \(\frac{\log j^*}{\log\left(\log j^*/\mu_{j^*}\right)}\) of the optimal separation when \(\frac{\log j^*}{\mu_{j^*}} \to \infty\) is missed.

    Furthermore, the division into the two asymptotic regimes considered in \cite{arias-castro_sparse_2015,donoho_higher_2022} is somewhat artificial. The boundary \(\log p\) between low counts \(\max_{1 \leq j \leq p} \mu_j = o(\log p)\) and high counts \(\min_{1 \leq j \leq p} \mu_j = \omega(\log p)\) is not so appealing as it depends explicitly on the ambient dimension \(p\). This choice does not allow for wide heterogeneity across categories in the null hypothesis. It is more appealing to aim at a, so-called, dimension-free analysis in which the asymptotic regimes under study and the obtained results are determined by the null hypothesis. Moreover, we derive corresponding asymptotic constants in the multinomial model, which was not covered in~\cite{arias-castro_sparse_2015,donoho_higher_2022}. 

    \subsection{Sparse uniformity testing}
    Bhattacharya and Mukherjee \cite{bhattacharya_sparse_2024} consider the problem of testing the uniformity hypothesis against sparse alternatives separated in total variation distance. They consider data \(Y_1,...,Y_n \overset{iid}{\sim} \pi\) where \(\pi \in \Delta_p\), and study the problem 
    \begin{align*}
        H_0 &: \pi = u, \\
        H_1 &: \dTV(\pi, u)\geq \rho \text{ and } ||\pi - u||_0 \leq s, 
    \end{align*}
    where \(u = \left(p^{-1},...,p^{-1}\right) \in \Delta_p\) is the uniform distribution. Bhattacharya and Mukherjee \cite{bhattacharya_sparse_2024} consider asymptotics as \(p \to \infty\) and parametrize \(s = p^{1-\beta}\) for \(\beta \in (0, 1)\), following in a long line of work in sparse signal detection \cite{ingster_problems_1997,donoho_higher_2004,donoho_higher_2015,donoho_higher_2022,arias-castro_sparse_2015,cai_optimal_2014,cai_optimal_2011,ditzhaus_signal_2019}. Among other results, they obtain the sharp constant in the minimax separation radius for the so-called sparse regime \(\frac{1}{2} < \beta < 1\). If \(\frac{n}{p \log^3 p} \to \infty\), then the null and alternative hypotheses asymptotically separate if \(\liminf_{p \to \infty} \frac{\rho}{s \sqrt{\frac{2 \log p}{np}}} > C(\beta)\) and asymptotically merge if \(\liminf_{p \to \infty} \frac{\rho}{s \sqrt{\frac{2 \log p}{np}}} < C(\beta)\) where \(C(\beta)\) is a constant (depending only on \(\beta\)) which they explicitly obtain. Furthermore, they show that no sequence of tests can separate \(H_0\) and \(H_1\) if \(\frac{n}{p \log p} \to 0\). As the authors note, the shape constraint \(\pi \in \Delta_p\) causes the triviality in this regime. Bhattacharya and Mukherjee \cite{bhattacharya_sparse_2024} can remove the logarithmic gap and reduce \(\log^3 p\) to \(\log p\) if one asks only for rate optimality (i.e. up to constants).

    Notably, their result concerns only \(\beta < 1\), i.e. it is not applicable to the case \(\beta = 1\) which corresponds to \(s \asymp 1\). For \(s \asymp 1\) and \(\sum_{j=1}^{p} \mathbbm{1}_{\{\pi_j \neq u_j\}}\leq s\), we have \(\dTV(\pi, u) \asymp ||\pi - u||_1 \asymp \left|\left| \pi - u\right|\right|_\infty\). Though the prior we use in the lower bound construction of Section \ref{section:multinomial_lowerbound} is not supported the set of \(s\)-sparse perturbations, it is interesting to see what rate is predicted by (\ref{rate:multinomial}) though it is not formally valid. Taking \(q_0 = \pi\) in (\ref{rate:multinomial}), we have \(\rho^* \asymp \sqrt{\frac{\log p}{np}} \mathbbm{1}_{\{n \geq p \log p\}} + \left(\sqrt{\frac{1}{np}} + \frac{\log p}{n \log\left(\frac{p\log p}{n}\right)}\right) \mathbbm{1}_{\{n < p \log p\}}\). The rate \(\sqrt{\frac{\log p}{np}}\) in the regime \(n \gtrsim p \log p\) asserted by \cite{bhattacharya_sparse_2024} is recovered. However, some care is needed in further interpretation. Strictly speaking, the regime \(n \lesssim p \log p\) is not meaningful when \(s \asymp 1\). To see this, consider \(\rho^* \asymp \frac{1}{p} \cdot \frac{x}{\log x}\) where \(x = \frac{p \log p}{n}\). Since \(\frac{p\log p}{n} \gtrsim 1\) implies \(\frac{x}{\log x} \gtrsim 1\), we have \(\rho^* \gtrsim \frac{1}{p}\), and it is immediately clear that no \(\pi \in \Delta_p\) exists with \(||\pi - u||_0 \lesssim 1\) and \(||\pi - u||_\infty \gtrsim \rho^*\).  Nevertheless, it would be interesting to see whether the rate \(\sqrt{\frac{1}{np}} + \frac{\log p}{n \log\left(\frac{p\log p}{n}\right)}\) predicted by (\ref{rate:multinomial}) in the regime \(n \lesssim p \log p\) actually holds when considering notions of ``soft'', rather than ``hard'', sparsity (e.g. formulated in terms of \(\ell_q\) norms for \(0 < q < 1\) rather than \(\ell_0\)).

    \section{Proofs for results in the Poisson model}
    
    Proofs of the main results in the Poisson model (\ref{model:poisson}) are presented in this section. The minimax upper bound in the Poisson setting (Theorem \ref{thm:upper_bound}) is proved in Section \ref{section:poisson_upper_bound_proof}, and the lower bound is proved in Section \ref{section:poisson_lower_bound_proof}. The proofs of the sharp asymptotic constants stated in Section \ref{section:poisson_sharp_constant} are deferred to Appendix \ref{appendix:poisson_sharp_constant}. 

    \subsection{Upper bound}\label{section:poisson_upper_bound_proof}
    As noted in Section \ref{section:poisson_upper_bound}, the proof of Theorem \ref{thm:upper_bound} is relatively straightforward; the result essentially follows by union bound and Bennett's inequality (Lemma \ref{lemma:Bennett}).
    
    \begin{proof}[Proof of Theorem \ref{thm:upper_bound}]
        Fix \(\eta \in (0, 1)\) and let \(C_\eta > 0\) be a quantity to be set later. Examining the Type I error and letting \(u_j = \mu_j h^{-1}\left(\frac{\log(C' j^2)}{\mu_j}\right)\), consider by union bound and Lemma \ref{lemma:Bennett},
        \begin{align*}
            P_{\mu}\left\{\varphi = 1\right\} \leq P_{\mu}\left( \bigcup_{j = 1}^{p} \left\{|X_j - \mu_j| > u_j \right\}\right) \leq \sum_{j=1}^{p} 2e^{-\mu_j h\left(\frac{u_j}{\mu_j}\right)} \leq \sum_{j=1}^{p}2e^{-\log(C'j^2)} \leq \sum_{j = 1}^{p} \frac{2}{C'j^2}.
        \end{align*}
        Since the series \(\sum_{j=1}^{\infty} j^{-2}\) converges, we can take \(C'\) sufficiently large to guarantee \(P_{\mu}\left\{\varphi = 1\right\} \leq \frac{\eta}{2}\). Let us now examine the Type II error. Fix \(\lambda \in \Lambda(\mu, C_\eta \psi)\). There exists \(1 \leq j^* \leq p\) such that \(|\lambda_{j^*} - \mu_{j^*}| \geq C_\eta \psi\). Let \(u^* = \max_{1 \leq j \leq p} \mu_{j} h^{-1}\left(\frac{\log(C'j^2)}{\mu_{j}}\right)\) and note \(u^* \leq \frac{C_\eta}{2}\psi\) as we can select \(C_\eta\) sufficiently large. Therefore, by Chebyshev's inequality, we have 
        \begin{align*}
            P_{\lambda}\left\{\varphi = 0\right\} \leq P_{\lambda}\left\{ |X_{j^*} - \mu_{j^*}| \leq u^* \right\} \leq \frac{\lambda_{j^*}}{\left(|\lambda_{j^*} - \mu_{j^*}| - \frac{C_\eta}{2}\psi\right)^2} \leq \frac{|\lambda_{j^*} - \mu_{j^*}|}{\frac{1}{4}|\lambda_{j^*} - \mu_{j^*}|^2} + \frac{\mu_{j^*}}{\frac{1}{4}C_\eta^2\psi^2} \leq \frac{4}{C_\eta\psi} + \frac{4\mu_{j^*}}{C_\eta^2\psi^2}.
        \end{align*}
        Since \(\psi \geq 1\), we can select \(C_\eta\) sufficiently large so that \(\frac{4}{C_\eta \psi} \leq \frac{\eta}{4}\). Likewise, consider \(h^{-1}(x) \gtrsim \sqrt{x}\) for \(x > 0\), and so \(\psi^2 \gtrsim \mu_{j^*}\). Therefore, taking \(C_\eta\) sufficiently large yields \(\frac{4\mu_{j^*}}{C_\eta^2\psi^2} \leq \frac{\eta}{4}\), and so the Type II error is bounded by \(\frac{\eta}{2}\) uniformly over \(\lambda \in \Lambda(\mu, C_\eta\psi)\). Hence, we have shown the testing risk is bounded by \(\eta\), as desired. 
    \end{proof}
    
    \subsection{Lower bound}\label{section:poisson_lower_bound_proof}
    Proposition \ref{prop:constant_lower_bound}, which asserts the constant order term of (\ref{rate:poisson}) in the lower bound is proved a simple two-point construction as noted in Section \ref{section:poisson_lower_bound}.
    \begin{proof}[Proof of Proposition \ref{prop:constant_lower_bound}]
        Fix \(\eta \in (0, 1)\) and take \(c_\eta = (1-\eta)^2\). Define \(\mu' = (\mu_1 + c_\eta, \mu_2,...,\mu_p)\), and consider by the Neyman-Pearson lemma and Lemma \ref{lemma:poisson_TV},
        \begin{align*}
            \mathcal{R}_{\mathcal{P}}(c_\eta, \mu) \geq 1 - \dTV(P_\mu, P_{\mu'}) \geq 1 - \dTV(\Poisson(\mu_1), \Poisson(\mu_1 + c_\eta)) \geq 1 - \sqrt{c_\eta} = \eta,
        \end{align*}
        as desired.
    \end{proof}
    
    The proof of Theorem \ref{thm:lower_bound} proceeds by first establishing the prior \(\pi\) defined in Section \ref{section:poisson_lower_bound} is supported on the proper parameter space containing rates \(\lambda\) that exhibit the desired separation. 

    \begin{lemma}\label{lemma:poisson_prior}
        If \(C \geq e\) and \(c \geq 0\), then \(\pi\) is supported on \(\Lambda\left(\mu, c \max_{1 \leq j \leq p} \mu_j h^{-1}\left(\frac{\log(ej)}{\mu_j}\right)\right)\). 
    \end{lemma}
    \begin{proof}
        If \(\lambda \sim \pi\), it is immediate \(||\lambda - \mu||_\infty = c\psi\). Note since \(C \geq e\) and \(h^{-1}\) is an increasing function, it follows \(\psi \geq \mu_{j^*}h^{-1}\left(\frac{\log(ej^*)}{\mu_{j^*}}\right) = \max_{1 \leq j \leq p} \mu_j h^{-1}\left(\frac{\log(ej)}{\mu_j}\right)\) since \(j^*\) is given by (\ref{def:jstar}). Therefore, \(\lambda \in \Lambda\left(\mu, c \max_{1 \leq j \leq p} \mu_j h^{-1}\left(\frac{\log(ej)}{\mu_j}\right)\right)\), completing the proof. 
    \end{proof}

    We now argue that the Bayes testing problem 
    \begin{align*}
        H_0 &: \lambda = \mu, \\
        H_1 &: \lambda \sim \pi,
    \end{align*}
    is connected to a Bayes testing problem with the auxiliary homoskedastic null (\ref{problem:homo0}). Proposition \ref{prop:poisson_flattening} establishes that we can consider a related Bayes testing problem which is applicable for furnishing a lower bound for the problem (\ref{problem:homo0})-(\ref{problem:homo1}). Thus, from the perspective of the lower bound, the heteroskedastic problem has been essentially reduced to a homoskedastic problem. 

    \begin{proposition}\label{prop:poisson_flattening}
        If \(C \geq e\) and \(c \geq 0\), then 
        \begin{align*}
            &\dTV\left(P_\mu, P_\pi\right) \leq \dTV\left(\Poisson(\mu_{j^*})^{\otimes j^*}, \frac{1}{j^*}\sum_{J=1}^{j^*} \bigotimes_{j=1}^{j^*} \Poisson(\mu_{j^*} + c\psi \mathbbm{1}_{\{j = J\}}) \right),
        \end{align*}
        where \(P_\pi = \int P_\lambda \, d\pi\) is the mixture induced by \(\pi\). 
    \end{proposition}
    Proposition \ref{prop:poisson_flattening} is a direct consequence of the following, general ``flattening" result which describes the relationship between a heteroskedastic null and an auxiliary homoskedastic version.
    
    \begin{proposition}[Flattening]\label{prop:general_flattening}
        Suppose \(\omega_1 \geq ... \geq \omega_p \geq 0\) and \(\gamma\) is a probability distribution on \([0, \infty)^p\). If \(k \in \{1,...,p\}, \underline{\omega} \leq \omega_k,\) and \(\gamma\) are such that for \(\xi \sim \gamma\), 
        \begin{enumerate}[label=(\roman*)]
            \item \(\min_{1 \leq j \leq k} \xi_j - \omega_j + \underline{\omega} \geq 0\), 
            \item \((\xi_1,...,\xi_k)\) and \((\xi_{k+1},...,\xi_p)\) are independent,
        \end{enumerate}
        then
        \begin{align}
        \begin{split}\label{eqn:flattening_split}
            \dTV\left(\bigotimes_{j=1}^{p} \Poisson(\omega_j), \int \bigotimes_{j=1}^{p} \Poisson(\xi_j) \, d\gamma(\xi) \right) &\leq \dTV\left(\Poisson(\underline{\omega})^{\otimes k}, \int \bigotimes_{j \leq k} \Poisson(\xi_j - \omega_j + \underline{\omega})\,d\gamma(\xi) \right) \\
            &\;\;\; + \dTV\left(\bigotimes_{j > k} \Poisson(\omega_j), \int \bigotimes_{j > k} \Poisson(\xi_j) \, d\gamma(\xi) \right). 
        \end{split}
        \end{align}
    \end{proposition}
    \begin{proof}
        Item \((ii)\) implies \(\int \bigotimes_{j=1}^{p} \Poisson(\xi_j) \, d\gamma(\xi) = \left(\int \bigotimes_{j \leq k} \Poisson(\xi_j) \, d\gamma(\xi) \right) \otimes \left( \int \bigotimes_{j > k} \Poisson(\xi_j) \, d\gamma(\xi) \right)\). Therefore,  
        \begin{align*}
            \dTV\left(\bigotimes_{j=1}^{p} \Poisson(\omega_j), \int \bigotimes_{j=1}^{p} \Poisson(\xi_j) \, d\gamma(\xi) \right) &\leq \dTV\left(\bigotimes_{j \leq k} \Poisson(\omega_j), \int \bigotimes_{j \leq k} \Poisson(\xi_j) \, d\gamma(\xi)\right) \\
            &\;\;\; + \dTV\left(\bigotimes_{j > k} \Poisson(\omega_j), \int \bigotimes_{j > k} \Poisson(\xi_j) \, d\gamma(\xi) \right),
        \end{align*}
        and so it suffices to examine just the first term on the right hand side. By the infinite divisibility of the Poisson distribution, we have \(\bigotimes_{j \leq k} \Poisson(\omega_j) = \Poisson(\underline{\omega})^{\otimes k} * \bigotimes_{j \leq k} \Poisson(\omega_j - \underline{\omega})\) since \(\omega_j \geq \omega_k \geq \underline{\omega}\) for \(j \leq k\). Likewise, by item \((i)\) we have \(\int \bigotimes_{j \leq k} \Poisson(\xi_j) \, d\gamma(\xi) = \left(\int \bigotimes_{j \leq k} \Poisson(\xi_j - (\omega_j - \underline{\omega})) \, d\gamma(\xi)\right) * \bigotimes_{j \leq k} \Poisson(\omega_j - \underline{\omega})\) for \(j \leq k\) and \(\xi \sim \gamma\). It thus follows by the data-processing inequality (Lemma \ref{lemma:dpi}) that 
        \begin{equation*}
            \dTV\left(\bigotimes_{j \leq k} \Poisson(\omega_j), \int \bigotimes_{j \leq k} \Poisson(\xi_j) \, d\gamma(\xi)\right) \leq \dTV\left(\Poisson(\underline{\omega})^{\otimes k}, \int \bigotimes_{j \leq k} \Poisson(\xi_j - \omega_j + \underline{\omega})\,d\gamma(\xi) \right), 
        \end{equation*}
        which yields the claimed result. 
    \end{proof}

    \begin{proof}[Proof of Proposition \ref{prop:poisson_flattening}]
        The result will follow by an application of Proposition \ref{prop:general_flattening}. With the choice \(\gamma = \pi\), \(k = j^*\), \(\omega = \mu\), and \(\underline{\omega} = \mu_{j^*}\), it is clear items \((i)\) and \((ii)\) are satisfied. Since \(\lambda \sim \pi\) implies \(\lambda_j = \mu_j\) for all \(j > j^*\), the second term in (\ref{eqn:flattening_split}) is zero. 
    \end{proof}   
    
    The lower bound argument proceeds by implementing the conditional second-moment method with the auxiliary null \(P^*_\mu = \Poisson(\mu_{j^*})^{\otimes j^*}\) and auxiliary mixture \(P^*_{\pi} = \frac{1}{j^*}\sum_{J=1}^{j^*} \bigotimes_{j=1}^{j^*} \Poisson(\mu_{j^*} + c\psi\mathbbm{1}_{\{j = J\}})\). For notational clarity, denote the data coming from either \(P^*_\mu\) and \(P^*_{\pi}\) as \(V\). We will condition on the event 
    \begin{equation}\label{def:truncation_event}
        E:= \left\{\max_{1 \leq j \leq j^*} V_j - \mu_{j^*} \leq \psi\right\}. 
    \end{equation}
    Denote \(\tilde{P}_\pi\) and \(\tilde{P}_\mu\) to be the conditional distributions \(P^*_\pi (\cdot \,|\, E)\) and \(P^*_{\mu}(\cdot \,|\, E)\) respectively, that is to say, for any event \(A\) we have \(\tilde{P}_{\mu}(A) = \frac{P^*_\mu(A \cap E)}{P^*_\mu(E)}\) and \(\tilde{P}_{\pi}(A) = \frac{P^*_{\pi}(A \cap E)}{P^*_{\pi}(E)}\). The following lemma asserts that it suffices to bound the \(\chi^2\)-divergence of the conditional distributions, i.e. the conditional second moment can be carried out. 

    \begin{lemma}\label{lemma:truncation_whp}
        Suppose (\ref{eqn:psi_large}) holds. If \(\alpha > 0\), then there exist \(C^* \geq e\) sufficiently large and \(c^* > 0\) sufficiently small both depending only on \(\alpha\) such that \(\dTV\left(P^*_{\mu}, P^*_{\pi}\right) \leq \frac{1}{2}\sqrt{\chi^2(\tilde{P}_\pi\,||\, \tilde{P}_\mu)} + \frac{2\alpha}{3}\) for all \(C \geq C^*\) and \(c \leq c^*\). 
    \end{lemma}
    \begin{proof}[Proof of Lemma \ref{lemma:truncation_whp}]
        By Corollary \ref{corollary:constant_rate_max}, \(C^* \geq e\) can be selected sufficiently large depending only \(\alpha\) so that \(C \geq C^*\) implies \(P^*_{\mu}(E^c) \leq \frac{\alpha}{12}\). 
        Similarly, consider \(P^*_\pi(E^c) \leq \frac{\alpha}{12} + P\left\{\Poisson(\mu_{j^*} + c\psi) - \mu_{j^*} > \psi\right\}\). 
        Observe by Chebyshev's inequality \(P\big\{\Poisson(\mu_{j^*} + c\psi) - \mu_{j^*} > \psi\big\} \leq \frac{\mu_{j^*}}{(1-c)^2\psi^2} + \frac{c}{(1-c)^2\psi}\). Since (\ref{eqn:psi_large}) holds (which is to say \(\psi \geq 1\)), we have \(\frac{c}{(1-c)^2\psi} \leq \frac{\alpha}{24}\) since \(c \leq c^*\) and we can select \(c^*\) sufficiently small. 
        Furthermore, consider \(\psi^2 \gtrsim \mu_{j^*}\log(Cj^*)\) since \(h^{-1}(x) \gtrsim \sqrt{x}\) by Lemma \ref{lemma:h_inverse}. 
        Therefore, taking \(C^*\) sufficiently large ensures \(\frac{\mu_{j^*}}{(1-c)^2\psi^2}  \leq \frac{\alpha}{24}\), and so \(P_{\pi}^*(E^c) \leq \frac{\alpha}{6}\). 
        An application of the triangle inequality and Lemma \ref{lemma:conditional_TV} then yields 
        \[\dTV(P^*_\mu, P^*_\pi) \leq \dTV(\tilde{P}_\mu, \tilde{P}_\pi) + 2\dTV(\tilde{P}_\mu, P^*_\mu) + 2\dTV(\tilde{P}_\pi, P^*_\pi) \leq \frac{1}{2} \sqrt{\chi^2(\tilde{P}_\pi\,||\,\tilde{P}_\mu)} + \frac{2\alpha}{3}.\] 
    \end{proof}

    Implementing the conditional second moment method boils down to bounding a certain moment generating function, as the following lemma establishes. 
    \begin{lemma}\label{lemma:conditional_second_moment}
        Suppose (\ref{eqn:psi_large}) holds. If \(\alpha > 0\) and \(C^*\geq e\), \(c^* > 0\) are given by Lemma \ref{lemma:truncation_whp}, then 
        \begin{equation*}
            \chi^2(\tilde{P}_\pi\,||\, \tilde{P}_\mu) + 1 \leq \frac{1}{\left(1-\frac{\alpha}{6}\right)^2} E\left(\exp\left(\frac{c^2\psi^2}{\mu_{j^*}}\mathbbm{1}_{\{J = J'\}}\right) P\left\{ \left. \max_{1 \leq j \leq j^*} W_j - \mu_{j^*} \leq \psi \,\right|\, \rho, \rho'\right\}\right)
        \end{equation*}
        for all \(C \geq C^*\) and \(c \leq c^*\). Here, \(\rho, \rho' \overset{iid}{\sim} \pi\) and \(J, J'\) are the corresponding random indices, and \(W_j \,|\, \rho, \rho' \overset{ind}{\sim} \Poisson\left(\frac{\rho_j\rho_j'}{\mu_{j^*}}\right)\). 
    \end{lemma}
    \noindent The proof is deferred to Appendix \ref{appendix:poisson_lower_bound}.

    \begin{proposition}\label{prop:chisquare_bound}
        Suppose (\ref{eqn:psi_large}) holds. If \(\alpha > 0\) and \(C^* \geq e\), \(c^* > 0\) are given by Lemma \ref{lemma:truncation_whp}, then there exists \(c^{**} \in (0, c^*]\) sufficiently small depending only on \(\alpha\) such that if \(C = C^*\) and \(c \leq c^{**}\), then 
        \begin{equation*}
            E\left(\exp\left(\frac{c^2\psi^2}{\mu_{j^*}} \mathbbm{1}_{\{J = J'\}}\right)P\left\{ \left. \max_{1 \leq j \leq j^*} W_j - \mu_{j^*} \leq \psi \,\right|\, \rho, \rho'\right\}\right) \leq 1 + \alpha. 
        \end{equation*}
    \end{proposition}
    
    The proof is deferred to Appendix \ref{appendix:poisson_lower_bound}. It is essential to use Bennett's inequality (\ref{lemma:Bennett}) in bounding the probability of the conditioned event to obtain some cancellation with the exponential term. With Proposition \ref{prop:chisquare_bound} providing a bound on the \(\chi^2\) divergence of the conditional distributions, Theorem \ref{thm:lower_bound} can now be proved. 
    \begin{proof}[Proof of Theorem \ref{thm:lower_bound}]
        Fix \(\eta \in (0, 1)\) and let \(\alpha \in (0, 1)\) be such that \(\frac{2\alpha}{3} + \frac{1}{2}\sqrt{\frac{1}{\left(1-\frac{\alpha}{6}\right)^2} \left(1 + \alpha\right) - 1} \leq 1-\eta\). 
        Let \(C^* \geq e\) and \(c^{**} > 0\) be given as in the statement of Proposition \ref{prop:chisquare_bound}. Take \(c_\eta = c^{**}\) and let \(C = C^*\). By Lemma \ref{lemma:poisson_prior}, \(\pi\) is supported on \(\Lambda\left(\mu, c \max_{1 \leq j \leq p} \mu_j h^{-1}\left(\frac{\log(ej)}{\mu_j}\right)\right)\). Hence, it follows by the Neyman-Pearson lemma that
        \begin{align*}
            \mathcal{R}_{\mathcal{P}}\left(c_\eta \max_{1 \leq j \leq p} \mu_j h^{-1}\left(\frac{\log(ej)}{\mu_j}\right), \mu\right) &\geq \inf_{\varphi}\left\{P_{\mu}\left\{ \varphi = 1 \right\} + P_{\pi}\left\{\varphi = 0 \right\}\right\} \\
            &= 1 - \dTV(P_\mu, P_{\pi}) \\
            &\geq 1 - \dTV(P_\mu^*, P_{\pi}^*) \\
            &\geq 1 - \frac{2\alpha}{3} - \frac{1}{2}\sqrt{\chi^2(\tilde{P}_\pi\,||\,\tilde{P}_\mu)}.
        \end{align*}
        The penultimate line follows from Proposition \ref{prop:poisson_flattening}, and the final line follows from Lemma \ref{lemma:truncation_whp}. Since (\ref{eqn:psi_large}) holds by hypothesis, it follows by Proposition \ref{prop:chisquare_bound}, Lemma \ref{lemma:conditional_second_moment}, and our choice of \(\alpha\) that \(1 - \frac{2\alpha}{3} - \frac{1}{2}\sqrt{\chi^2(\tilde{P}_\pi\,||\,\tilde{P}_\mu)} \geq 1 - \frac{2\alpha}{3} - \frac{1}{2}\sqrt{\frac{1}{\left(1-\frac{\alpha}{6}\right)^2}\left(1 + \alpha\right) - 1} \geq \eta\). The proof is complete. 
    \end{proof}

    \section{Proofs for results in the multinomial model}
    Proofs of the main results in the multinomial model (\ref{model:multinomial}) are presented in this section. The upper bound is addressed in Section \ref{section:multinomial_upper_bound_proofs} and the lower bound is addressed in Section \ref{section:multinomial_lower_bound_proofs}. 

    \subsection{Upper bound}\label{section:multinomial_upper_bound_proofs}
    The proof of Proposition \ref{prop:multinomial1} is straightforward. 
    \begin{proof}[Proof of Proposition \ref{prop:multinomial1}]
        Fix \(\eta \in (0, 1)\) and set \(C_\eta = 2K_1 \vee \frac{32}{\eta} \vee \frac{16}{\sqrt{\eta}}\). Examining the Type I error, consider 
        \begin{equation*}
            P_{q_0}\left\{\varphi_1 = 1\right\} \leq P_{q_0}\left\{|X_1-nq_0(1)| \geq K_1 \left(1 + \sqrt{nq_0(1)(1-q_0(1))}\right)\right\} \leq \frac{\eta}{4},
        \end{equation*}
        where we have used \(K_1 \geq (\eta/4)^{-1/2}\) and Chebyshev's inequality along with \(E_{q_0}(X_1) = nq_0(1)\) and \(\Var_{q_0}(X_1) = nq_0(1)(1-q_0(1))\). Examining the Type II error, consider that for \(q \in \Pi_1(q_0, C_\eta \varepsilon_1)\) we have \(|nq(1) - nq_0(1)| \geq C_\eta n\varepsilon_1 \geq 2K_1\left(1 + \sqrt{nq_0(1)(1-q_0(1))}\right)\) since \(C_\eta \geq 2K_1\). Therefore, by Chebyshev's inequality, we have 
        \begin{align*}
            &\sup_{q \in \Pi_1(q_0, C_\eta\varepsilon_1)} P_{q}\left\{|X_1 - nq_0(1)| < K_1\left(1 + \sqrt{nq_0(1)(1-q_0(1))}\right)\right\} \\
            &\leq \sup_{q \in \Pi_1(q_0, C_\eta\varepsilon_1)} P_q\left\{|X_1 - nq(1)| > |nq(1) - nq_0(1)| - \frac{C_\eta n \varepsilon_1}{2}\right\} \\
            &\leq \sup_{q \in \Pi_1(q_0 C_\eta\varepsilon_1)}\frac{nq(1)(1-q(1))}{\left(|nq(1) - nq_0(1)| - \frac{C_\eta n \varepsilon_1}{2}\right)^2} \\
            &\leq \sup_{q \in \Pi_1(q_0, C_\eta\varepsilon_1)} \frac{n|q(1) - q_0(1)|}{\frac{1}{4}|nq(1) - nq_0(1)|^2} + \frac{nq_0(1)(1-q_0(1))}{\frac{1}{4}C_\eta^2n^2\varepsilon_1^2} \\
            &\leq \frac{4}{C_\eta\left(1 + \sqrt{nq_0^{\max}(1-q_0^{\max})}\right)} + \frac{4}{C_\eta^2} \\
            &\leq \frac{\eta}{4}.
        \end{align*}
        To obtain the third line, we have used the inequality \(|x(1-x) - y(1-y)| \leq |x-y|\) for \(x, y \in [0, 1]\). The final line follows from \(\frac{4}{C_\eta^2} \leq \frac{\eta}{8}\) and \(\frac{4}{C_\eta} \leq \frac{\eta}{8}\). Hence, the sum of the Type I and Type II errors is bounded by \(\frac{\eta}{2}\), as desired. 
    \end{proof}

    The proof of Proposition \ref{prop:multinomial2} broadly follows the same reasoning of the proof of Theorem \ref{thm:upper_bound}. The only modification is to use the version of Bennett's inequality for binomial random variables (Corollary \ref{corollary:bennett_binomial}) rather than that for Poisson random variables (Lemma \ref{lemma:Bennett}). 

    \begin{proof}[Proof of Proposition \ref{prop:multinomial2}]
        Fix \(\eta \in (0, 1)\) and let \(C_\eta > 0\) be a quantity to be set; we will point out in the course of the proof where \(C_\eta\) is to be chosen sufficiently large. Let us examine the Type I error first. Note that under the null hypothesis, we have \(X_j \sim \Binomial(n, q_0(j))\). Consequently, by union bound and Corollary \ref{corollary:bennett_binomial} we have 
        \begin{align*}
            P_{q_0}\left\{\varphi_2 = 1\right\} \leq \sum_{j=2}^{p} P_{q_0}\left\{|X_j - nq_0(j)| > nq_0(j)(1-q_0(j))h^{-1}\left(\frac{\log(K_2(j-1)^2)}{nq_0(j)(1-q_0(j))}\right)\right\} \leq \sum_{j=2}^{p} \frac{2}{K_2(j-1)^2}.
        \end{align*}
        Since \(\sum_{j=2}^{\infty} \frac{1}{(j-1)^2} < \infty\), we can take \(K_2\) sufficiently large depending only on \(\eta\) to ensure \(P_{q_0}\left\{\varphi_2 = 1\right\} \leq \frac{\eta}{4}\). 
        
        With the Type I error handled, we turn to the Type II error. Fix \(q \in \Pi_2(q_0, C_\eta \varepsilon_2)\). Then there exists \(2 \leq j^* \leq p\) such that \(|nq(j^*) - nq_0(j^*)| \geq C_\eta \varepsilon_2\). Define \(u^* = \max_{2 \leq j \leq p} nq_0(j)(1-q_0(j))h^{-1}\left(\frac{\log(K_2(j-1)^2)}{nq_0(j)(1-q_0(j))}\right)\). Note \(u^* \leq \frac{C_\eta}{2} n\varepsilon_2\) since \(C_\eta\) can be chosen sufficiently large. Since \(|nq(j^*) - nq_0(j^*)| \geq C_\eta n\varepsilon_2\), we can apply Chebyshev's inequality as follows,  
        \begin{align*}
            P_{q}\left\{\varphi_2 = 0\right\} \leq P_{q}\left\{ |X_{j^*} - nq_0(j^*)| \leq u^* \right\} &\leq \frac{|nq(j^*) - nq_0(j^*)|}{\frac{1}{4}|nq(j^*) - nq_0(j^*)|^2} + \frac{nq_0(j^*)(1-q_0(j^*))}{\frac{1}{4}C_\eta^2n^2\varepsilon_2^2} \\
            &\leq \frac{4}{C_\eta n\varepsilon_2} + \frac{4nq_0(j^*)(1-q_0(j^*))}{C_\eta^2n^2\varepsilon_2^2}.
        \end{align*}
        To obtain the third line, we have used the inequality \(|x(1-x) - y(1-y)| \leq |x-y|\) for \(x, y \in [0, 1]\). Note \(n\varepsilon_2 \geq 1\), and so taking \(C_\eta\) sufficiently large depending only on \(\eta\) implies \(\frac{4}{C_\eta n\varepsilon_2} \leq \frac{\eta}{8}\). Furthermore, observe \(h^{-1}(x) \gtrsim \sqrt{x}\) for all \(x \geq 0\), and so \(n^2\varepsilon_2^2 \gtrsim nq_0(j^*)(1-q_0(j^*))\). Therefore, taking \(C_\eta\) sufficiently large also guarantees \(\frac{4nq_0(j^*)(1-q_0(j^*))}{C_\eta^2n^2\varepsilon_2^2} \leq \frac{\eta}{8}\). Thus, the Type II error is bounded by \(\frac{\eta}{4}\) uniformly over \(q \in \Pi_2(q_0, C_\eta\varepsilon_2)\), and so we have shown the testing risk is bounded by \(\frac{\eta}{2}\) as claimed. 
    \end{proof}

    \subsection{Lower bound}\label{section:multinomial_lower_bound_proofs}

        We present here the proof structure of the lower bound in the multinomial model. 
        We start by proving Proposition~\ref{prop:1/n_rate} which asserts the $\frac{1}{n}$ lower bound.
        \begin{proof}[Proof of Proposition~\ref{prop:1/n_rate}]
        Fix \(\eta \in (0, 1)\) and take \(c_\eta = \frac{1-\eta}{2}\). Define \(q_1\) with \(q_1(j) := \left(1 - \frac{2c_\eta}{n}\right)q_0(j) + \frac{2c_\eta}{n}\mathbbm{1}_{\{j = 2\}}\). It is immediate to verify \(q_1 \in \Delta_p\) since \(2c_\eta < 1\), and, moreover, we have \(||q - q_0||_\infty \geq |q_1(2) - q_0(2)| = \left|\frac{2c_\eta}{n} - \frac{2c_\eta}{n}q_0(2)\right| \geq \frac{c_\eta}{n}\) since \(q_0(2) \leq \frac{q_0(1)+q_0(2)}{2} \leq \frac{1}{2}\). Therefore, \(q_1 \in \Pi(q_0, \frac{c_\eta}{n})\) and so 
    $$\mathcal{R}_{\mathcal{M}}\left(\frac{c_\eta}{n}, n, q_0\right) \geq 1 - \dTV(P_{q_0},P_{q_1}) \geq 1 - n\dTV(q_0, q_1).$$ 
        Moreover, \(\dTV(q_0, q_1) = \frac{1}{2} \sum_{j=1}^{p} |q_0(j) - q_1(j)| = \frac{1}{2} \sum_{j \neq 2} \frac{2c_\eta}{n} q_0(j) + \frac{|q_1(2) - q_0(2)|}{2} \leq \frac{1}{2} \sum_{j = 1}^{p} \frac{2c_\eta}{n} q_0(j) + \frac{1}{2} \frac{2c_\eta}{n} \leq \frac{2c_\eta}{n} \leq \frac{1-\eta}{n}\), completing the proof.   
    \end{proof}

    Next, we prove Lemma~\ref{lem:relation_multinomial_poisson} which connects the minimax risks in models~\eqref{model:multinomial} and~\eqref{eqn:poissonized_multinomial_dgp}. The proof relies on the fact that a $\operatorname{Poisson}((1+c)n)$ random variable exceeds the value $n$ with probability at least $1-\frac{1+c}{c^2n}$ by Chebyshev's inequality.
    Therefore, the model $X \sim \bigotimes_{j=1}^p \operatorname{Poisson}((1+c)n q(j))$ typically contains more information than the model $X \sim \Multinomial(n,q)$, making detection easier.
    
    \begin{proof}[Proof of Lemma~\ref{lem:relation_multinomial_poisson}]
        Recall in the model (\ref{model:poissonized}), at sample size \((1+c)n\) we have \(N \sim \Poisson((1+c)n)\). Conditionally on \(N\), the random variable \(X\) follows a multinomial distribution with parameters \(N\) and \(q\). Therefore, one can construct i.i.d.~random variables \(Y_1,...,Y_N \,|\, N \overset{iid}{\sim} \Multinomial(1, q)\) such that \(X\) is the histogram of \(Y_1,...,Y_n\) conditionally on \(N\). For any test \(\tilde{\varphi}\) applicable to data from \(\Multinomial(n, q)\) (equivalently, \(Y_1,...,Y_n\)), observe 
        \begin{align*}
            &P_{q_0}\left\{\tilde{\varphi}(Y_1,...,Y_n) = 1\right\} + \sup_{q \in \Pi(q_0, \varepsilon)} P_{q}\left\{\tilde{\varphi}(Y_1,...,Y_n) = 0\right\} \\
            &= E\left(P_{q_0}\left\{\tilde{\varphi}(Y_1,...,Y_n) = 1\,|\,N\right\}\right) + \sup_{q \in \Pi(q_0, \varepsilon)} E\left(P_{q}\left\{\tilde{\varphi}(Y_1,...,Y_n) = 0 \,|\,N\right\}\right) \\
            &\geq \inf_{\varphi}\left\{ E\left(\mathbf{P}_{q_0}\left\{\varphi(X) = 1\,|\,N\right\}\mathbbm{1}_{\{N \geq n\}}\right) + \sup_{q \in \Pi(q_0, \varepsilon)} E\left(\mathbf{P}_{q}\left\{\varphi(X) = 0 \,|\,N\right\}\mathbbm{1}_{\{N \geq n\}}\right)\right\} \\
            &= \inf_{\varphi}\left\{ \mathbf{P}_{q_0}\left\{\varphi(X) = 1\right\} + \sup_{q \in \Pi(q_0, \varepsilon)} \mathbf{P}_{q}\left\{\varphi(X) = 0 \right\}\right\} - 2P\left\{N < n\right\} \\
            &\geq \mathcal{R}_{\mathcal{PM}}(\varepsilon, (1+c)n, q_0) - \frac{2(1+c)}{c^2 n}. 
        \end{align*}
        Here, we have used Chebyshev's inequality to argue \(P\{N < n\} \leq \frac{(1+c)n}{c^2 n^2} = \frac{2(1+c)}{c^2 n}\). Taking infimum over \(\tilde{\varphi}\) yields the desired result. 
    \end{proof}

    The next piece is the parametric term in the rate which is provided by Proposition~\ref{prop:multinomial_parametric}. As discussed in Section \ref{section:prior_construction_parametric}, the lower bound argument involves a two-point construction. The proof of this proposition is deferred to Appendix~\ref{section:multinomial_lower_bound_proof}.
    
    Finally, we turn to the remaining term \(\max_j q_0^{-\max}(j) \Gamma\left(\frac{\log(ej)}{nq_0^{-\max}(j)}\right)\) discussed in Section \ref{section:prior_construction_local}. It is convenient to have an estimate on the size of \(m\) for later use in proofs, which the following lemma provides. Note that the exponent \(1/4\) in Lemma \ref{lemma:m_size} is not critical and could be replaced with any constant \(\beta \in (0, 1)\), however it is useful later on to choose some \(\beta < \frac{1}{2}\). 
    \begin{lemma}\label{lemma:m_size}
        There exists a sufficiently large universal constant \(C_*\) such that if (\ref{eqn:multinomial_psi_large}) holds, then \(m \lesssim (j^*)^{1/4}\).
    \end{lemma}

        The proof is deferred to Appendix~\ref{section:multinomial_lower_bound_proof}.
            It now remains to prove a lower bound of the order of $\frac{\psi}{n}:=\max_{j} q_0^{-\max} h^{-1}\left(\frac{\log(ej)}{nq_0^{-\max}(j)}\right)$, as stated in Theorem~\ref{thm:poissonized_multinomial_lower_bound}.     
    The following proposition establishes \(\pi\) is indeed supported on \(\Pi(q_0, \frac{c\psi}{n})\).
    \begin{lemma}\label{lem:support_prior}
        There exists a sufficiently large universal constant \(C_* > 0\) and sufficiently small \(c_\eta\) depending only on \(\tilde{C}_\eta\) such that if (\ref{eqn:multinomial_psi_large}) holds and \(0 < c < c_\eta\), then \(\pi\) is supported on \(\Pi(q_0, \frac{c\psi}{n})\). 
    \end{lemma}
    \begin{proof}[Proof of Lemma~\ref    {lem:support_prior}]
        If \(m = 0\), then a draw \(q \sim \pi\) is deterministic and satisfies \(q = q_0\) and so there is nothing to prove. Suppose \(m \geq 1\). For a draw \(q \sim \pi\), it is clear \(||q - q_0||_\infty = \frac{c\psi}{n}\). To show \(q \in \Delta_p\), consider \(\sum_{j=1}^{p} q(j) = c\frac{\psi}{n} - m \cdot c\frac{\psi}{nm} + \sum_{j=1}^{p} q_0(j) = 1\). By Lemma \ref{lemma:m_size}, we have \(m \lesssim (j^*)^{1/4}\), from which it follows \(m \geq \kappa^{-1} \left\lceil h^{-1}\left(\frac{\log(ej^*)}{nq_0^{-\max}(j^*)}\right) \right\rceil\), by definition of \(m\), for some universal constant \(\kappa > 0\). Therefore, we have 
        \begin{equation*}
            c\frac{\psi}{nm} \leq c \kappa \frac{h^{-1}\left(\frac{\log(\tilde{C}_\eta j^*)}{nq_0^{-\max}(j^*)}\right)}{\left\lceil h^{-1}\left(\frac{\log(ej^*)}{nq_0^{-\max}(j^*)}\right) \right\rceil} \cdot q_0^{-\max}(j^*) \leq q_0^{-\max}(j^*) = q_0(j^*+1)
        \end{equation*}
        where the second inequality follows from taking \(c_\eta\) sufficiently small depending on \(\tilde{C}_\eta\) and using \(c < c_\eta\). Since it has been assumed without loss of generality that \(q_0(j)\) is decreasing in \(j\), it thus follows \(q_0(j) - \frac{c\psi}{nm} \geq q_0(j) - q_0(j^*+1) \geq 0\) for all \(j \in \mathcal{I}\). Hence \(q \in \Delta_p\), and thus we have shown \(q \in \Pi(q_0, \frac{c\psi}{n})\). The proof is complete.
    \end{proof}

    The following proposition (Proposition \ref{prop:flattening}) shows that the hypothesis testing problem 
    \begin{align*}
        H_0 &: X \sim \bigotimes_{j=1}^{p} \Poisson(n q_0(j)), \\
        H_1 &: q \sim \pi \text{ and } X|q \sim \bigotimes_{j=1}^{p} \Poisson(nq(j))
    \end{align*}
    is no easier than the testing problem 
    \begin{align}
        H_0 &: Y \sim \Poisson\left(n q_0^{-\max}(j^*)\right)^{\otimes j^*}, \label{problem:multinomial_flat0}\\
        H_1 &: q \sim \pi \text{ and } Y|q \sim \bigotimes_{j=1}^{j^*} \Poisson(n\tilde{q}(j)) \label{problem:multinomial_flat1}
    \end{align}
    where \(\tilde{q} \in \R^{j^*}\) is given by \(\tilde{q}(j-1) = q(j) - q_0(j) + q_0^{-\max}(j^*)\) for \(2 \leq j \leq j^*+1\). That is to say, we have
    \begin{equation*}
        \tilde{q}(j-1) = 
        \begin{cases}
            q_0^{-\max}(j^*) + c\frac{\psi}{n}   &\textit{if } j = J,\\
            q_0^{-\max}(j^*) - c\frac{\psi}{nm} &\textit{if } j \in \mathcal{I}, \\
            q_0^{-\max}(j^*) &\textit{otherwise},
        \end{cases}
    \end{equation*}
    where \(\mathcal{I}\) and \(J\) are given in the definition of \(\pi\). This auxiliary testing problem can be understood as the result of flattening the original heteroskedastic null distribution into a homoskedastic null distribution which is more amenable to analysis. 
       
    The following proposition relates the initial testing problem to the flattened one~\eqref{problem:multinomial_flat0}-\eqref{problem:multinomial_flat1}.
    
    \begin{proposition}\label{prop:flattening}
        If \(c_\eta > 0\) is sufficiently small depending only on \(\tilde{C}_\eta\), then 
        \begin{align*}
            &\dTV\left(\bigotimes_{j=1}^{p} \Poisson(nq_0(j)), \int \bigotimes_{j=1}^{p} \Poisson(nq(j)) \, \pi(dq) \right) \\
            &\leq \dTV\left(\Poisson(nq_0^{-\max}(j^*))^{\otimes j^*}, \int \bigotimes_{j=1}^{j^*} \Poisson(n\tilde{q}(j)) \, \pi(dq)\right)
        \end{align*}
        provided \(c < c_\eta\). 
    \end{proposition}
    \begin{proof}
        
        The result will follow from an application of Proposition \ref{prop:general_flattening}. Let \(\gamma = \pi\), \(k = j^*+1\), \(\omega = nq_0\), and \(\underline{\omega} = nq_0^{-\max}(j^*)\). It is clear item \((ii)\) of the statement of Proposition \ref{prop:general_flattening} is satisfied. Note in the notation of Proposition \ref{prop:general_flattening}, we have \(\xi = nq\) and \(\xi_j - \omega_j + \underline{\omega} = nq_0(j)\mathbbm{1}_{\{j = 1\} \cup \{j > j^*\}} + n\tilde{q}(j-1)\mathbbm{1}_{\{2 \leq j \leq j^*+1\}}\). Since \(c_\eta\) is sufficiently small, it follows by Lemma \ref{lemma:qtilde} that item \((i)\) is satisfied. The result then follows from Proposition \ref{prop:general_flattening} since the second term in (\ref{eqn:flattening_split}) is zero.
    \end{proof}

        To move forward with the proof of Theorem~\ref{thm:poissonized_multinomial_lower_bound}, it must first be verified \(\tilde{q}(j) \geq 0\) for all \(j \leq j^*\) so that the definition of the alternative hypothesis (\ref{problem:multinomial_flat1}) is coherent.
    \begin{lemma}\label{lemma:qtilde}
        If \(c_\eta > 0\) is sufficiently small depending only on \(\tilde{C}_\eta\), then \(\tilde{q}(j) \geq 0\) for all \(j \leq j^*\). 
    \end{lemma}
    \begin{proof}
        If \(m = 0\), then \(\tilde{q} = q_0\) and so there is nothing to prove. If \(m \geq 1\), then
        \begin{align*}
            \tilde{q}(j) &\geq q_0^{-\max}(j^*) - \frac{c\psi}{nm} \\
            &= q_0^{-\max}(j^*)\left(1 - c \cdot \frac{h^{-1}\left(\frac{\log(\tilde{C}_\eta j^*)}{nq_0^{-\max}(j^*)} \right)}{\left\lceil h^{-1}\left(\frac{\log(ej^*)}{nq_0^{-\max}(j^*)} \right)\right\rceil }\right) \\
            &\geq 0
        \end{align*}
        since \(c < c_\eta\) and \(c_\eta\) is chosen sufficiently small depending on \(\tilde{C}_\eta\).
    \end{proof}

    The following lemma bounds the total variation distance associated to the flattened problem in Proposition \ref{prop:flattening} via the conditional second-moment method. For notational ease, let \(P_0 = \bigotimes_{j=1}^{j^*} \Poisson(n q_0^{-\max}(j^*))\) denote the null hypothesis (\ref{problem:multinomial_flat0}) and denote the alternative hypothesis (\ref{problem:multinomial_flat1}) by \(P_\pi = \int \bigotimes_{j=1}^{j^*} \Poisson(n\tilde{q}(j)) \, \pi(dq)\). 
    Denote \(\mu_j = nq_0^{-\max}(j)\). We will condition on the event
    \begin{equation}\label{def:conditional_event}
        E := \left\{ \max_{1 \leq j \leq j^*} Y_j - \mu_{j^*} \leq \psi \right\}.
    \end{equation}
    Let \(\tilde{P}_0\) and \(\tilde{P}_{\pi}\) denote the measures \(P_0\) and \(P_\pi\) conditioned on the event \(E\), that is to say, for any event \(A\) we have \(\tilde{P}_0(A) = \frac{P_0(A \cap E)}{P_0(E)}\) and \(\tilde{P}_\pi(A) = \frac{P_\pi(A \cap E)}{P_\pi(E)}\).

    \begin{lemma}\label{lem:TV_chi2}
        If \(\alpha > 0\), then there exists \(\tilde{C}_\eta\) sufficiently large and \(c_\eta\) sufficiently small depending only on \(\alpha\) such that 
        \begin{equation}\label{eqn:multinomial_master_lower_bound}
            \dTV\left(P_0, P_{\pi}\right) \leq \frac{1}{2}\sqrt{\chi^2\left(\left.\left.\tilde{P}_\pi \,\right|\right|\, \tilde{P}_0\right)} + \frac{2\alpha}{3}
        \end{equation}
        provided \(0 < c < c_\eta\). 
    \end{lemma}
    \begin{proof}
        By Corollary \ref{corollary:constant_rate_max}, \(\tilde{C}_\eta\) can be selected large enough to ensure \(P_{0}(E^c) \leq \frac{\alpha}{6}\). Furthermore, consider that \(P_{\pi}(E^c) \leq \frac{\alpha}{6} + P\left\{ \Poisson(\mu_{j^*} + c\psi) - \mu_{j^*} > \psi \right\} \leq \frac{\alpha}{3}\) since \(c < c_\eta\) and \(c_\eta\) can be taken sufficiently small. Triangle inequality delivers
        \begin{equation*}
            \dTV(P_{0}, P_{\pi}) \leq \dTV(\tilde{P}_0, \tilde{P}_{\pi}) + \frac{2\alpha}{3} \leq \frac{1}{2} \sqrt{\chi^2\left(\left.\left.\tilde{P}_\pi \,\right|\right|\, \tilde{P}_0\right)} + \frac{2\alpha}{3}
        \end{equation*}
        as desired.
    \end{proof}

    It remains to bound the \(\chi^2\) divergence in (\ref{eqn:multinomial_master_lower_bound}). The following lemma reduces the task to bounding two specific probabilistic quantities. 

    \begin{lemma}\label{lemma:multinomial_conditional}
        If \(\alpha > 0\), then there exists \(\tilde{C}_\eta\) sufficiently large and \(c_\eta\) sufficiently small depending only on \(\alpha\) such that
        \begin{align}
            \begin{split}\label{eqn:multinomial_chisquare_bound}
            &\chi^2\left(\left.\left.\tilde{P}_\pi \,\right|\right|\, \tilde{P}_0\right) + 1 \\
            &\leq \frac{1}{1-\frac{\alpha}{3}} E\left(e^{\frac{c^2\psi^2}{m^2\mu_{j^*}}|\mathcal{I} \cap \mathcal{I}'|} \right) \left[1 + \frac{1}{j^*\!-\!m} \bigg( e^{\frac{c^2\psi^2}{\mu_{j^*}}}P\left\{ \Poisson\!\left(\frac{(\mu_{j^*} \! + c\psi)^2}{\mu_{j^*}}\right) \leq \mu_{j^*}\! + \psi \right\} - 1 \bigg)_{\!\!+} \right]
            \end{split}
        \end{align}
        provided \(0 < c < c_\eta\). Here, \(\mathcal{I}\) and \(\mathcal{I}'\) are i.i.d. copies and we adopt the convention $\frac{c^2\psi^2}{m^2\mu_{j^*}}|\mathcal{I} \cap \mathcal{I}'| = 0$ if $\psi = 0$ and $\mathcal{I} = \mathcal{I}' = \emptyset$ due to $m = 0$.  
    \end{lemma}

    The proof is deferred to Appendix~\ref{section:multinomial_lower_bound_proof}. The following lemma furnishes a bound for the moment generating function appearing in (\ref{eqn:multinomial_chisquare_bound}). 

    \begin{lemma}\label{lemma:mgf_bound}
        There exists a sufficiently large universal constant \(C_* > 0\) such that the following holds. If (\ref{eqn:multinomial_psi_large}) holds, then there exists \(C_\eta^\dagger\) and \(c_\eta\) depending only on \(\tilde{C}_\eta\) such that 
        \begin{equation*}
            E\left(e^{\frac{c^2\psi^2}{m^2\mu_{j^*}}|\mathcal{I} \cap \mathcal{I}'|}\right) \leq e^{c^2 C_\eta^\dagger}
        \end{equation*}
        provided \(c < c_\eta\). 
    \end{lemma}
    The proof is deferred to Appendix~\ref{section:multinomial_lower_bound_proof}.         
    Having bounded the moment generating function, we now turn to bounding the other term in (\ref{eqn:multinomial_chisquare_bound}). We carefully employ Bennett's inequality (\ref{lemma:Bennett}) to obtain some cancellation between the exponential term and the lower tail probability. This cancellation is crucial to obtaining the sharp rate, otherwise only the subgaussian portion of the rate would be established. 
    \begin{lemma}\label{lem:control_chi2_multinomial}
        There exists a sufficiently large universal constant \(C_* \geq 1\) such that the following holds. If the condition (\ref{eqn:multinomial_psi_large}) is satisfied, then there exists \(C_{\eta}^{\dagger\dagger}\) and \(c_\eta\) depending only on \(\tilde{C}_\eta\) such that   
        \begin{equation*}
            1 + \frac{1}{j^*-m} \left( e^{\frac{c^2\psi^2}{\mu_{j^*}}}P\left\{ \Poisson\left(\frac{(\mu_{j^*} + c\psi)^2}{\mu_{j^*}}\right) \leq \mu_{j^*} + \psi \right\} - 1 \right)_{\!+}  \leq 1 + \frac{1}{\tilde{C}_\eta} + e^{c^2 C_{\eta}^{\dagger\dagger}}
        \end{equation*}
        provided \(c < c_\eta\). 
    \end{lemma}

    The proof is deferred to Appendix~\ref{section:multinomial_lower_bound_proof}.       
        To conclude the proof of Theorem~\ref{thm:poissonized_multinomial_lower_bound}, it remains to combine Lemmas~\ref{lem:TV_chi2}, \ref{lemma:multinomial_conditional}, \ref{lemma:mgf_bound} and \ref{lem:control_chi2_multinomial} to obtain
    \begin{align*}
        \dTV\left(P_0, P_{\pi}\right) \leq 1-\eta
    \end{align*}
    provided $\alpha,\tilde{c}_\eta,c$ are chosen sufficiently small and $\tilde{C}_\eta, C_*$ are chosen sufficiently large.
            
    \bibliographystyle{skotekal.bst}
    \bibliography{sup_test.bib}     

    \appendix

    \section{Deferred proofs in the Poisson model }\label{appendix:poisson_lower_bound}

    \subsection{Deferred proofs for the lower bound}
    In this section, we present the proofs of Lemma \ref{lemma:conditional_second_moment} and Proposition \ref{prop:chisquare_bound}, both of which were used to prove Theorem \ref{thm:lower_bound}.

    \begin{proof}[Proof of Lemma \ref{lemma:conditional_second_moment}]
        Note from the proof of Lemma \ref{lemma:truncation_whp} that \(P^*_\pi(E^c) \vee P^*_\mu(E^c) \leq \frac{\alpha}{6}\). The likelihood ratio is \(\frac{d\tilde{P}_\pi}{d\tilde{P}_\mu} = \frac{P_\mu^*(E)}{P_\pi^*(E)} \cdot \frac{dP_\pi}{dP_\mu} \mathbbm{1}_E\). To bound the \(\chi^2\)-divergence, let us write \(\delta = \rho - \mu_{j^*}\mathbf{1}_{j^*}\) and \(\delta' = \rho' - \mu_{j^*}\mathbf{1}_{j^*}\) for \(\rho, \rho' \overset{iid}{\sim} \pi\). Here, \(\mathbf{1}_{j^*} \in \R^{j^*}\) denotes the vector with all entries equal to one. Let us also write \(J, J'\) to denote the associated random indices. Consider 
        \begin{align*}
            &\chi^2\left(\tilde{P}_\pi \,||\, \tilde{P}_\mu \right) + 1 \\
            &= \frac{P_\mu(E)}{P_\pi(E)^2}\sum_{x \in (\mathbb{N}\cup\{0\})^p} \frac{dP_\pi^2(x)}{dP_\mu(x)} \cdot \mathbbm{1}_{\left\{x \in E\right\}} \\
            &\leq \frac{1}{\left(1-\frac{\alpha}{6}\right)^2}\iint \sum_{x \in E} \prod_{j=1}^{p} \frac{1}{x_{j}!} \frac{e^{-\rho_{j}}\rho_{j}^{x_{j}} e^{-\rho_{j}'}(\rho_{j}')^{x_{j}}}{e^{-\mu_{j^*}}\mu_{j^*}^{x_{j}}} \, \pi(d\rho)\pi(d\rho') \\
            &= \frac{1}{\left(1-\frac{\alpha}{6}\right)^2}\iint \exp\left( \sum_{j=1}^{p} -\rho_j - \rho_j' + \mu_{j^*} + \frac{\rho_{j}\rho'_{j}}{\mu_{j^*}}\right) P\left\{ \left. \max_{1 \leq j \leq j^*} W_j - \mu_{j^*} \leq \psi \right| \rho, \rho' \right\} \, \pi(d\rho)\pi(d\rho') \\
            &= \frac{1}{\left(1-\frac{\alpha}{6}\right)^2}\iint \exp\left(\frac{\langle \delta, \delta'\rangle}{\mu_{j^*}}\right) P\left\{ \left. \max_{1 \leq j \leq j^*} W_j - \mu_{j^*} \leq \psi \right| \rho, \rho' \right\} \, \pi(d\rho)\pi(d\rho')
        \end{align*}
        where \(W_j \overset{ind}{\sim} \Poisson\left(\frac{\rho_j \rho_j'}{\mu_{j^*}}\right)\). Since \(\frac{\langle \delta, \delta'\rangle}{\mu_{j^*}} = \frac{c^2\psi^2}{\mu_{j^*}} \mathbbm{1}_{\left\{ J = J' \right\}} \), the claimed result follows. 
    \end{proof}
    
    \begin{proof}[Proof of Proposition \ref{prop:chisquare_bound}]
        First, note 
        \begin{align}
            \begin{split}
            &E\left(\exp\left(\frac{c^2\psi^2}{\mu_{j^*}}\mathbbm{1}_{\{J = J'\}}\right) P\left\{ \left. \max_{1 \leq j \leq j^*} W_j - \mu_{j^*} \leq \psi \,\right|\, \rho, \rho'\right\}\right)\\
            &\leq \left(1 - \frac{1}{j^*}\right) + \iint_{\{J = J'\}} \exp\left(\frac{c^2\psi^2}{\mu_{j^*}}\right) P\left\{ \left. \max_{1 \leq j \leq j^*} W_j - \mu_{j^*} \leq \psi \right| \rho, \rho' \right\} \, \pi(d\rho)\pi(d\rho'). \label{eqn:chisquare_bound}
            \end{split}
        \end{align}
        Let \(c \leq c^{**}\) where \(c^{**}\) will be selected later. There are two cases to consider. Let \(\tilde{c}\) denote a sufficiently small constant depending only on \(\alpha\). \newline 

        \noindent \textbf{Case 1:} Suppose \(\mu_{j^*} > \tilde{c} \log\left(Cj^*\right)\). It follows by Lemma \ref{lemma:h_inverse} that \(e^{\frac{c^2\psi^2}{\mu_{j^*}}} \leq \exp\left(\frac{c^2}{\mu_{j^*}} \cdot C'\mu_{j^*}^2 \frac{\log\left(Cj^*\right)}{\mu_{j^*}}\right) \leq \exp\left( c^2 C'' \log(1+j^*)\right)\) where \(C', C''\) are constants depending on \(\alpha\). Thus, from (\ref{eqn:chisquare_bound}) we have the bound 
        \begin{align*}
            &E\left(\exp\left(\frac{c^2\psi^2}{\mu_{j^*}}\mathbbm{1}_{\{J = J'\}}\right) P\left\{ \left. \max_{1 \leq j \leq j^*} W_j - \mu_{j^*} \leq \psi \,\right|\, \rho, \rho'\right\}\right)\nonumber \\
            &\leq 1 + \frac{1}{j^*} \left(\exp\left(c^2 C'' \log(1 + j^*)\right) - 1\right)\nonumber \\
            &\leq 1 + c^2C'' \nonumber \\
            &\leq 1 + \alpha.
        \end{align*}
        Here, we have taken \(c\) sufficiently small and have used the bound \((1+y)^\delta \leq 1+\delta y\) for \(y \geq 0\) and \(\delta \in (0, 1)\). The analysis for this case is complete. \newline 

        \noindent \textbf{Case 2:} Suppose \(\mu_{j^*} \leq \tilde{c} \log\left(Cj^*\right)\). We now consider two further subcases. \newline 
        
        \textbf{Case 2.1:} Suppose \(\psi \leq c^{-2}\mu_{j^*}\). Then by Lemma \ref{lemma:h_inverse}
        \begin{align*}
            e^{\frac{c^2\psi^2}{\mu_{j^*}}} &\leq e^\psi \\
            &= \exp\left(\mu_{j^*} h^{-1}\left(\frac{\log(C j^*)}{\mu_{j^*}}\right)\right) \\
            &\leq \exp\left( C_1 \frac{\log(Cj^* )}{\log\left(\frac{\log(Cj^*)}{\mu_{j^*}}\right)} \right) \\
            &\leq \exp\left(\frac{C_1 \log(C)}{\log(\tilde{c}^{-1})} \cdot \log\left(1+j^*\right) \right).
        \end{align*}
        where \(C_1 > 0\) is a universal constant whose value can change from instance to instance. We have used \(\mu_{j^*} \leq \tilde{c}\log(Cj^*)\) to obtain the final line. By taking \(\tilde{c}\) sufficiently small and arguing as in Case 1, we have 
        \begin{equation*}
            E\left(\exp\left(\frac{c^2\psi^2}{\mu_{j^*}}\mathbbm{1}_{\{J = J'\}}\right) P\left\{ \left. \max_{1 \leq j \leq j^*} W_j - \mu_{j^*} \leq \psi \,\right|\, \rho, \rho'\right\}\right) \leq 1 + \alpha.
        \end{equation*}
        The analysis for this case is complete. \newline 

        \textbf{Case 2.2:} Suppose \(\psi > c^{-2} \mu_{j^*}\). By Lemma \ref{lemma:bennett_cancellation}, we have 
        \begin{align*}
            e^{\frac{c^2\psi^2}{\mu_{j^*}}} P\left\{ \left. \max_{1 \leq j \leq j^*} W_j - \mu_{j^*} \leq \psi \right| \rho, \rho' \right\} &\leq e^{\frac{c^2\psi^2}{\mu_{j^*}}} P\left\{W_J \leq \mu_{j^*} + \psi \,|\, J\right\} \\
            &= e^{\frac{c^2\psi^2}{\mu_{j^*}}} P\left\{\Poisson\left(\frac{(\mu_{j^*} + c\psi)^2}{\mu_{j^*}}\right) \leq \psi + \mu_{j^*} \right\} \\
            &\leq \exp\left(-\mu_{j^*}h\left(\frac{\psi}{\mu_{j^*}}\right) + 2\mu_{j^*}\left(1 + \frac{\psi}{\mu_{j^*}}\right)\log\left(1 + \frac{c\psi}{\mu_{j^*}}\right)\right).
        \end{align*}
        Examining (\ref{eqn:chisquare_bound}) and noting \(\mu_{j^*}h\left(\frac{\psi}{\mu_{j^*}}\right) = \log(Cj^*)\), we thus have  
        \begin{align*}
            &1 - \frac{1}{j^*} + \iint_{\{J = J'\}} \exp\left(\frac{c^2\psi^2}{\mu_{j^*}}\right) P\left\{ \left. \max_{1 \leq j \leq j^*} W_j - \mu_{j^*} \leq \psi \right| \rho, \rho' \right\} \, \pi(d\rho)\pi(d\rho') \\
            &\leq 1 + \frac{1}{j^*}  \exp\left(-\mu_{j^*}h\left(\frac{\psi}{\mu_{j^*}}\right) + 2\mu_{j^*}\left(1 + \frac{\psi}{\mu_{j^*}}\right)\log\left(1 + \frac{c\psi}{\mu_{j^*}}\right)\right) \\
            &= 1 + \exp\left( -2\mu_{j^*} h\left(\frac{\psi}{\mu_{j^*}}\right) + 2\mu_{j^*} \left(1 + \frac{\psi}{\mu_{j^*}}\right) \log\left(1 + \frac{c\psi}{\mu_{j^*}}\right) + \log(C)\right) \\
            &\leq 1 + \exp\left( \mu_{j^*} g\left(\frac{\psi}{\mu_{j^*}}\right) + \log(C)\right) 
        \end{align*}
        where \(g : [0, \infty) \to \R\) is the function \(g(x) = -2h(x) + 2(1+x)\log(1+c^{**}x)\). Note we have used \(c \leq c^{**}\) here. Since \(C = C^*\), consider we can take \(c^{**}\) sufficiently small such that for all \(x \geq \frac{1}{(c^{**})^2}\) we have \(g(x) \leq x\left(-\log(C) + \log\left(\alpha\right)\right)\). This is immediately seen by noting 
        \begin{align*}
            g(x) = -2(1+x)\log\left(\frac{1+x}{1+c^{**}x}\right) + 2x \leq x \left( -2 \log\left(\frac{1+x}{1+c^{**}x}\right) + 2\right),
        \end{align*}
        and so taking \(c^{**}\) sufficiently small clearly yields the desired property. Since \(\frac{\psi}{\mu_{j^*}} \geq \frac{1}{(c^{**})^2}\) and \(\psi \geq 1\) (due to the fact \(C \geq e\), \(h^{-1}\) is an increasing function, and (\ref{eqn:psi_large}) holds), it immediately follows that 
        \begin{align*}
            \exp\left( \mu_{j^*} g\left(\frac{\psi}{\mu_{j^*}}\right) + \log(C)\right) &\leq \exp\left(\mu_{j^*} \cdot \frac{\psi}{\mu_{j^*}}\left(-\log\left(C\right) + \log\left(\alpha\right)\right) + \log(C)\right) \leq \alpha. 
        \end{align*}
        Therefore, we have shown 
        \begin{equation*}
            E\left(\exp\left(\frac{c^2\psi^2}{\mu_{j^*}}\mathbbm{1}_{\{J = J'\}}\right) P\left\{ \left. \max_{1 \leq j \leq j^*} W_j - \mu_{j^*} \leq \psi \,\right|\, \rho, \rho'\right\}\right) \leq 1 + \alpha.   
        \end{equation*}
        The analysis for this case is complete. With all of the cases analyzed, the proof is complete. 
    \end{proof}

    \subsection{Interpretation of the subpoissonian regime: Proof of Lemma~\ref{lemma:intuition_subpoissonian_regime}}\label{subsec:intuition_subpoissonian_regime}

        \begin{proof}[Proof of Lemma~\ref{lemma:intuition_subpoissonian_regime}]
    We write $\psi = \mu_{j^*} h^{-1}\left(\frac{\log(ej^*)}{\mu_{j^*}}\right)$ and first note that
    \begin{align*}
        \log(j^*) \geq \log\left(\frac{e\log(j^*)}{\mu_{j^*}}\right) \geq \log(e/c) , 
    \end{align*}
    which ensures $j^*$ can be made arbitrarily large by taking $c$ small enough.
    Now, we have
    \begin{align}
        P_{\mu}\left\{\forall j \in [j^*]: X_j \geq 1\right\} = \prod_{j=1}^{j^*} (1-e^{-\mu_j}) \leq \exp\bigg(- \sum_{j=1}^{j^*} e^{-\mu_j}\bigg).\label{eq_proba_observed_once}
    \end{align}
    Let $c_0 = \frac{1}{2} h^{-1}(2)$. 
    We justify that the condition $\mu_{j^*} < c\log(ej^*)$ implies $\psi \leq \frac{c_0}{2} \log(ej^*)$ provided $c$ is small enough. 
    Writing $x = \frac{\mu_{j^*}}{\log(ej^*)} \leq c$, and noting the function $t \mapsto t\,h^{-1}(1/t)$ is increasing over $(0,\infty)$, we have
    \begin{align*}
        \psi = \log(ej^*) \frac{\mu_{j^*}}{\log(ej^*)} h^{-1}\left(\frac{\log(ej^*)}{\mu_{j^*}}\right) = \log(ej^*) \,x h^{-1}\!\left(\frac{1}{x}\right) \leq ch^{-1}\!\left(\frac{1}{c}\right) \log(ej^*) \leq \frac{c_0}{2} \log(ej^*)
    \end{align*}
    provided $c$ is small enough. 
    Moreover, for any $j \in [j^*]$ such that $c_0\log(ej) \geq \psi$, we have
    \begin{align*}
        c_0 \geq \frac{\psi}{\log(ej)} \geq \frac{\mu_j}{\log(ej)} h^{-1}\left(\frac{\log(ej)}{\mu_j}\right),
    \end{align*}
    which implies $\frac{\mu_j}{\log(ej)}\leq \frac{1}{2}$ by definition of $c_0$, or equivalently $\mu_{j} \leq \log\big((ej)^{1/2}\big)$. 
    Therefore
    \begin{align*}
        \sum_{j=1}^{j^*} e^{-\mu_j} \geq \sum_{\substack{j \leq j^*\\c_0\log(j)\geq \psi}} e^{-\log((ej)^{1/2})} \geq \sum_{\substack{j \leq j^*\\\log(j)\geq \frac{\log(ej^*)}{2}}} \frac{1}{(ej)^{1/2}} = \sum_{j= \left\lceil\sqrt{ej^*}\right\rceil}^{j^*} \frac{1}{(ej)^{1/2}},
    \end{align*}
    where we have used $\psi \leq \frac{c_0}{2} \log(ej^*)$ in the second inequality. 
    The right-hand side can now be made arbitrarily large by taking $j^*$ large enough, which can be achieved by taking $c$ small enough. 
    Combining with~\eqref{eq_proba_observed_once} yields the result.
    \end{proof}

    \section{Multinomial model}
    
    \subsection{Multinomial lower bound}\label{section:multinomial_lower_bound_proof}

    \begin{proof}[Proof of Proposition \ref{prop:multinomial_parametric}]

    Fix \(\eta \in (0, 1)\) and write, for ease of notation, \(\epsilon = q_0(1) \wedge \sqrt{\frac{q_0(1)(1-q_0(1))}{n}}\) where \(c_\eta \in [0, 1]\) is to be set. Define
    \begin{equation*}
        q_1(j) = 
        \begin{cases}
            q_0(1) - c_\eta\epsilon &\textit{if } j = 1, \\
            q_0(j) \left(1 + \frac{c_\eta\epsilon}{1-q_0(1)}\right) &\textit{if } j \geq 2. 
        \end{cases}
    \end{equation*}
    We will prove the lower bound of Proposition \ref{prop:multinomial_parametric} by considering the testing problem 
    \begin{align*}
        &H_0: q = q_0, \\
        &H_1: q = q_1. 
    \end{align*}
    In order for this construction to furnish a valid lower bound, it must be verified \(q_1\) is separated from \(q_0\) and is a probability vector. 

    \begin{lemma}\label{lem:prior_well_separated}
        If \(c_\eta \in [0, 1]\), then \(q_1 \in \Pi(q_0, c_\eta\epsilon)\).  
    \end{lemma}
    \begin{proof}
        It is clear \(q_1(j) \geq 0\) for all \(j \geq 1\) since \(\epsilon \leq q_0(1)\) and \(c_\eta \leq 1\). Furthermore, consider \(\sum_{j=1}^{p} q_1(j) = 1 - c_\eta\epsilon + \frac{c_\eta\epsilon}{1-q_0(1)}\sum_{j=2}^{p} q_0(j) = 1\). Hence, \(q_1 \in \Delta_p\). Furthermore, consider that \(||q_1 - q_0||_\infty = c_\eta\epsilon \left(1 \vee \left(\max_{j \geq 2} \frac{q_0(j)}{1-q_0(1)}\right)\right) = c_\eta\epsilon\) since \(1 \geq \frac{q_0(j)}{\sum_{i=2}^{p} q_0(i)} = \frac{q_0(j)}{1-q_0(1)}\) for all \(j \geq 2\). Thus, we have shown \(q_1 \in \Pi(q_0, c_\eta\epsilon)\) as desired. 
    \end{proof}

    \begin{proposition}\label{prop:chi2_small_parametric_rate}
        If \(c_\eta \in [0, 1]\), then
        \begin{equation*}
            \chi^2(\mathbf{P}_{q_1}|| \mathbf{P}_{q_0}) = \chi^2\left(\left.\left.\bigotimes_{j=1}^{p} \Poisson(nq_1(j)) \, \right|\right|\, \bigotimes_{j=1}^{p} \Poisson(nq_0(j)) \right) \leq e^{c_\eta^2} - 1. 
        \end{equation*}
    \end{proposition}
    \begin{proof}
        By direct calculation, 
        \begin{align*}
            &\chi^2\left(\left.\left. \bigotimes_{j=1}^{p} \Poisson(nq_1(j))\,\right|\right|\, \bigotimes_{j=1}^{p} \Poisson(nq_0(j)) \right) + 1 \\
            &= \sum_{x \in (\mathbb{N} \cup \{0\})^p} \prod_{j=1}^{p} \frac{1}{x_j!} \cdot \frac{e^{-2nq_1(j)}(nq_1(j))^{2x_j}}{e^{-nq_0(j)}(nq_0(j))^{x_j}} \\
            &= \exp\left(n\sum_{j=1}^{p} -2q_1(j) + q_0(j) + \frac{q_1^2(j)}{q_0(j)}\right) \sum_{x \in (\mathbb{N} \cup \{0\})^p} \prod_{j=1}^{p} \frac{e^{-n\frac{q_1(j)^2}{q_0(j)}}}{x_j!} \cdot \left(n\frac{q_1(j)^2}{q_0(j)}\right)^{x_j} \\
            &= \exp\left(n\sum_{j=1}^{p} -2q_1(j) + q_0(j) + \frac{q_1^2(j)}{q_0(j)}\right) \\
            &= \exp\left(n\left(-1 + \frac{(q_0(1) - c_\eta\epsilon)^2}{q_0(1)} + \sum_{j=2}^{p} q_0(j) \left(1 + \frac{c_\eta\epsilon}{1-q_0(1)}\right)^2 \right)\right) \\
            &= \exp\left(n\left(-1 + q_0(1) - 2c_\eta\epsilon + \frac{c_\eta^2\epsilon^2}{q_0(1)} + 1-q_0(1) + 2c_\eta\epsilon + \frac{c_\eta^2\epsilon^2}{1-q_0(1)}\right)\right) \\
            &\leq \exp\left(\frac{nc_\eta^2\epsilon^2}{q_0(1)(1-q_0(1))}\right) \\
            &\leq e^{c_\eta^2}
        \end{align*}
        since \(\epsilon^2 \leq \frac{q_0(1)(1-q_0(1))}{n}\). The proof is complete. 
    \end{proof}

    Proposition~\ref{prop:multinomial_parametric} follows by combining Lemma~\ref{lem:prior_well_separated} and Proposition~\ref{prop:chi2_small_parametric_rate} as follows
    \begin{align*}
        \mathcal{R}_{\mathcal{PM}}\left( c_\eta\left(q_0^{\max} \wedge \sqrt{\frac{q_0^{\max}(1-q_0^{\max})}{n}}\right), n, q_0\right) &\geq 1-\dTV(\mathbf P_{q_0}, \mathbf P_{q_1}) \\
        &\geq 1-\sqrt{\chi^2\bigg( \left.\left. \bigotimes_{j=1}^{p} \Poisson(nq_1(j)) \,\right|\right|\, \bigotimes_{j=1}^{p} \Poisson(nq_0(j)) \bigg)}\\
        &\geq 1-\sqrt{2c_\eta^2}\\
        &= \eta,
    \end{align*}
    where we have taken \(c_\eta = \frac{1-\eta}{\sqrt{2}}\).
\end{proof}

    \subsubsection{Proof of Theorem \ref{thm:poissonized_multinomial_lower_bound}}\label{subsubsec:Proof_thm_PM_LB}
    We now present the construction of the lower bound in Theorem \ref{thm:poissonized_multinomial_lower_bound}. Fix \(\eta \in (0, 1)\) and recall
    \begin{equation*}
        j^* = \argmax_{j} \, nq_0^{-\max}(j) h^{-1}\left(\frac{\log(ej)}{nq_0^{-\max}(j)}\right).
    \end{equation*}
    Let \(0 < c < c_\eta\) where \(c_\eta\) is a small constant depending only on \(\eta\) to be set. Recall
    \begin{equation*}
        m = \left\lceil h^{-1}\left(\frac{\log(ej^*)}{nq_0^{-\max}(j^*)}\right) \right\rceil \wedge (j^* - 1)
    \end{equation*}
    and
    \begin{equation*}
        \psi = nq_0^{-\max}(j^*) h^{-1}\left(\frac{\log(\tilde{C}_\eta j^*)}{nq_0^{-\max}(j^*)}\right) \mathbbm{1}_{\{m \geq 1\}},
    \end{equation*}
    where \(\tilde{C}_\eta \geq e\) is a large constant depending only on \(\eta\) to be set. 

        \begin{proof}[Proof of Lemma~\ref{lemma:m_size}]
        If \(m = 0\) the claim is trivially true, so suppose \(m \geq 1\). Suppose \(nq_0^{-\max}(j^*) \leq \log\left(ej^*\right)\). Recall (\ref{eqn:multinomial_psi_large}) gives \(nq_0^{-\max}(j^*) h^{-1}\left(\frac{\log(ej^*)}{nq_0^{-\max}(j^*)}\right) \geq C_*\). Consider that \(nq_0^{-\max}(j^*) \leq \log\left(ej^*\right)\) implies an application of Lemma \ref{lemma:h_inverse} yields \(nq_0^{-\max}(j^*) h^{-1}\left(\frac{\log(ej^*)}{nq_0^{-\max}(j^*)}\right) \asymp \frac{\log(ej^*)}{\log\left(\log(ej^*)/(nq_0^{-\max}(j^*))\right)}\), it then follows that   
        \begin{equation*}
            C_*^{-1} \log(ej^*) \geq \kappa \log\left(\frac{\log(ej^*)}{nq_0^{-\max}(j^*)}\right)
        \end{equation*}
        for some universal constant \(\kappa > 0\). Rearranging yields \(nq_0^{-\max}(j^*) \gtrsim \left(\frac{1}{j^*}\right)^{\frac{1}{\kappa C_*}}\). Now, consider \(m \lesssim h^{-1}\left(\frac{\log\left(ej^*\right)}{nq_0^{-\max}(j^*)}\right)\) since \(nq_0^{-\max}(j^*) \leq \log(ej^*)\) implies \(h^{-1}\left(\frac{\log\left(ej^*\right)}{nq_0^{-\max}(j^*)}\right) \gtrsim 1\). 
        Therefore, it follows \(m \lesssim \log(ej^*)/(nq_0^{-\max}(j^*)) \lesssim (j^*)^{\frac{1}{\kappa C_*}} \log(ej^*)\). In other words, we can take \(C_*\) sufficiently large such that to ensure \(m \lesssim (j^*)^{1/4}\). Now suppose \(nq_0^{-\max}(j^*) > \log(ej^*)\). It follows from Lemma \ref{lemma:h_inverse} that \(m \leq \left\lceil h^{-1}\left(\frac{\log(ej^*)}{nq_0^{-\max}(j^*)}\right) \right\rceil \lesssim 1\). Since \(m \geq 1\) implies \(j^* \geq 2\), it follows \(1 \lesssim (j^*)^{1/4}\), and so we have the desired result. 
    \end{proof}

We recall that the following prior distribution \(\pi\) for the alternative hypothesis will be used in the proof of Theorem \ref{thm:poissonized_multinomial_lower_bound}. A draw \(q \sim \pi\) is obtained by first drawing \(J \sim \Uniform(\{2,...,j^*+1\})\), then drawing uniformly at random a size \(m\) subset \(\mathcal{I} \subset \{2,...,j^*+1\} \setminus \left\{J\right\}\), and finally setting 
    \begin{equation*}
        q(j) = 
        \begin{cases}
            q_0(j) + c\frac{\psi}{n}   &\textit{if } j = J,\\
            q_0(j) - c\frac{\psi}{nm} &\textit{if } j \in \mathcal{I}, \\
            q_0(j) &\textit{otherwise},
        \end{cases}
    \end{equation*}
    for \(1 \leq j \leq p\). 
    If $m=0$, then $\mathcal{I}$ is empty and we have $q = q_0$. Note that the first coordinate is never perturbed, i.e. \(q(1) = q_0(1)\).

    \begin{proof}[Proof of Lemma~\ref{lemma:multinomial_conditional}]
        Consider that the likelihood ratio is 
        \begin{equation*}
            \frac{d\tilde{P}_{\pi}}{d\tilde{P}_0} = \frac{P_0(E)}{P_{\pi}(E)} \cdot \frac{dP_{\pi}}{dP_0} \mathbbm{1}_E.
        \end{equation*}
        To bound the \(\chi^2\)-divergence, let \(q, q' \overset{iid}{\sim} \pi\) and write \(\rho = n\tilde{q}\) and \(\rho' = n \tilde{q}'\). Let us also write \(\delta = \rho - \mu_{j^*}\mathbf{1}_{j^*}\) and \(\delta' = \rho' - \mu_{j^*}\mathbf{1}_{j^*}\) where $\mu_{j^*} = nq_0^{-\max}(j^*)$. Here, \(\mathbf{1}_{j^*} \in \R^{j^*}\) denotes the vector with all entries equal to one. By Corollary \ref{corollary:constant_rate_max}, we can take \(\tilde{C}_\eta\) sufficiently large depending on \(\alpha\) to ensure \(P_\pi(E^c) \leq \frac{\alpha}{6}\), so that $P_\pi(E)^2 \geq 1-\alpha/3$. With this in hand, consider 
        \begin{align}
            &\chi^2\left(\left.\left.\tilde{P}_\pi \,\right|\right|\, \tilde{P}_0\right) + 1 \nonumber\\
            &= \frac{P_0(E)}{P_{\pi}(E)^2}\sum_{x \in (\mathbb{N}\cup\{0\})^p} \frac{dP_{\pi}^2(x)}{dP_0(x)} \cdot \mathbbm{1}_{\left\{x \in E\right\}} \nonumber\\
            &\leq \frac{1}{1-\frac{\alpha}{3}}\iint \sum_{x \in E} \prod_{j=1}^{p} \frac{1}{x_{j}!} \frac{e^{-\rho_{j}}\rho_{j}^{x_{j}} e^{-\rho_{j}'}(\rho_{j}')^{x_{j}}}{e^{-\mu_{j^*}}\mu_{j^*}^{x_{j}}} \, {\pi}(dq){\pi}(dq') \nonumber\\
            &= \frac{1}{1-\frac{\alpha}{3}}\iint \exp\left( \sum_{j=1}^{p} -\rho_j - \rho_j' + \mu_{j^*} + \frac{\rho_{j}\rho'_{j}}{\mu_{j^*}}\right) P\left\{ \left. \max_{1 \leq j \leq j^*} W_j - \mu_{j^*} \leq \psi \right| \rho, \rho' \right\} \, {\pi}(dq){\pi}(dq') \nonumber\\
            &= \frac{1}{1-\frac{\alpha}{3}}\iint \exp\left(\frac{\langle \delta, \delta'\rangle}{\mu_{j^*}}\right) P\left\{ \left. \max_{1 \leq j \leq j^*} W_j - \mu_{j^*} \leq \psi \right| \rho, \rho' \right\} \, {\pi}(dq){\pi}(dq') \label{eq_chi2_multinomial}
        \end{align}
        where \(W_j \overset{ind}{\sim} \Poisson(\frac{\rho_j \rho_{j}'}{\mu_{j^*}})\). Let \(J, \mathcal{I}\) be the random objects associated with \(q\) and \(J', \mathcal{I}'\) be those associated with \(q'\). Consider we can write 
        \begin{align}
            \frac{\langle \delta, \delta'\rangle}{\mu_{j^*}} &= \frac{c^2\psi^2}{\mu_{j^*}} \mathbbm{1}_{\{J = J'\}}  - \frac{c^2\psi^2}{m\mu_{j^*}} \left(\mathbbm{1}_{\{J \in \mathcal{I}'\}} + \mathbbm{1}_{\{J' \in \mathcal{I}\}}\right) + \frac{c^2\psi^2}{m^2 \mu_{j^*}}|\mathcal{I} \cap \mathcal{I}'| \\
            &\leq \frac{c^2\psi^2}{\mu_{j^*}} \mathbbm{1}_{\{J = J'\}} + \frac{c^2\psi^2}{m^2\mu_{j^*}}|\mathcal{I} \cap \mathcal{I}'|.\label{eq_scalar_product}
        \end{align}
        Therefore, 
        \begin{align*}
            &\iint \exp\left(\frac{\langle \delta, \delta'\rangle}{\mu_{j^*}}\right) P\left\{ \left. \max_{1 \leq j \leq j^*} W_j - \mu_{j^*} \leq \psi \right| \rho, \rho' \right\} \, {\pi}(dq){\pi}(dq') \\
            &\leq E\left(\exp\left(\frac{c^2\psi^2}{\mu_{j^*}} \mathbbm{1}_{\{J = J'\}} + \frac{c^2\psi^2}{m^2\mu_{j^*}}|\mathcal{I} \cap \mathcal{I}'|\right) P\left\{ \left. \max_{1 \leq j \leq j^*} W_j - \mu_{j^*} \leq \psi \right| \rho, \rho' \right\} \right) \\
            &= E\left( \exp\left(\frac{c^2\psi^2}{m^2\mu_{j^*}}|\mathcal{I} \cap \mathcal{I}'|\right) E\left( \left. \exp\left(\frac{c^2\psi^2}{\mu_{j^*}} \mathbbm{1}_{\{J = J'\}}\right)P\left\{ \left. \max_{1 \leq j \leq j^*} W_j - \mu_{j^*} \leq \psi \right| \rho, \rho' \right\}\right| \mathcal{I}, \mathcal{I}'\right)\right). 
        \end{align*}
        It is clear \(J \,|\, \mathcal{I} \sim \Uniform(\mathcal{I}^c)\) and \(J'\,|\,\mathcal{I}' \sim \Uniform(\mathcal{I}'^c)\). Therefore, \(\mathbbm{1}_{\{J = J'\}} \,|\,\mathcal{I}, \mathcal{I}' \sim \Bernoulli\left(\frac{|\mathcal{I}^c \cap \mathcal{I}'^c|}{|\mathcal{I}^c| \cdot |\mathcal{I}'^c|}\right)\). With this in hand, consider 
        \begin{align*}
            &E\left( \left. \exp\left(\frac{c^2\psi^2}{\mu_{j^*}} \mathbbm{1}_{\{J = J'\}}\right)P\left\{ \left. \max_{1 \leq j \leq j^*} W_j - \mu_{j^*} \leq \psi \right| \rho, \rho' \right\}\right| \mathcal{I}, \mathcal{I}'\right) \\[5pt]
            &\leq \left(1 - \frac{|\mathcal{I}^c \cap \mathcal{I}'^c|}{|\mathcal{I}^c| \cdot |\mathcal{I}'^c|}\right) + E\left( \left. \exp\left(\frac{c^2\psi^2}{\mu_{j^*}} \right)P\left\{ \left. W_J - \mu_{j^*} \leq \psi \right| \rho, \rho' \right\} \cdot \mathbbm{1}_{\{J = J'\}}\right| \mathcal{I}, \mathcal{I}'\right) \\[5pt]
            &\leq \left(1 - \frac{|\mathcal{I}^c \cap \mathcal{I}'^c|}{|\mathcal{I}^c| \cdot |\mathcal{I}'^c|}\right) + e^{\frac{c^2\psi^2}{\mu_{j^*}}} E\left( \left. P\left\{ \Poisson\left(\frac{(\mu_{j^*} + c\psi)^2}{\mu_{j^*}}\right) \leq \mu_{j^*} + \psi \right\} \cdot \mathbbm{1}_{\{J = J'\}}\right| \mathcal{I}, \mathcal{I}'\right) \\[5pt]
            &\leq \left(1 - \frac{|\mathcal{I}^c \cap \mathcal{I}'^c|}{|\mathcal{I}^c| \cdot |\mathcal{I}'^c|}\right) + \frac{|\mathcal{I}^c \cap \mathcal{I}'^c|}{|\mathcal{I}^c| \cdot |\mathcal{I}'^c|} e^{\frac{c^2\psi^2}{\mu_{j^*}}} P\left\{ \Poisson\left(\frac{(\mu_{j^*} + c\psi)^2}{\mu_{j^*}}\right) \leq \mu_{j^*} + \psi \right\}  \\
            &\leq 1 + \frac{|\mathcal{I}^c \cap \mathcal{I}'^c|}{|\mathcal{I}^c| \cdot |\mathcal{I}'^c|} \left[ e^{\frac{c^2\psi^2}{\mu_{j^*}}}P\left\{ \Poisson\left(\frac{(\mu_{j^*} + c\psi)^2}{\mu_{j^*}}\right) \leq \mu_{j^*} + \psi \right\} - 1 \right]\\[5pt]
            &\leq 1 + \frac{1}{j^*-m} \left( e^{\frac{c^2\psi^2}{\mu_{j^*}}}P\left\{ \Poisson\left(\frac{(\mu_{j^*} + c\psi)^2}{\mu_{j^*}}\right) \leq \mu_{j^*} + \psi \right\} - 1 \right)_+. 
        \end{align*}
        Therefore, 
        \begin{align*}
            &\chi^2\left(\left.\left.\tilde{P}_\pi \,\right|\right|\, \tilde{P}_0\right) + 1 \\
            &\leq \frac{1}{1-\frac{\alpha}{3}} E\left(e^{\frac{c^2\psi^2}{m^2\mu_{j^*}}|\mathcal{I} \cap \mathcal{I}'|} \right) \left[1 + \frac{1}{j^*-m} \left( e^{\frac{c^2\psi^2}{\mu_{j^*}}}P\left\{ \Poisson\left(\frac{(\mu_{j^*} + c\psi)^2}{\mu_{j^*}}\right) \leq \mu_{j^*} + \psi \right\} - 1 \right)_+ \right]
        \end{align*}
        as claimed. 
    \end{proof}

    \begin{proof}[Proof of Lemma~\ref{lemma:mgf_bound}]
        If \(m = 0\), then the conclusion trivially holds for any \(C^\dagger_\eta > 0\) since we use the convention $\frac{c^2\psi^2}{m^2\mu_{j^*}}|\mathcal{I} \cap \mathcal{I}'| = 0$. Suppose \(m \geq 1\). Consider \(|\mathcal{I} \cap \mathcal{I}'|\) is a hypergeometric random variable. Consequently, by Proposition (20.6) from \cite{aldous_exchangeability_1985}, page 173, there exists a random variable \(B \sim \Binomial(\frac{m}{j^*-1}, m)\) and a \(\sigma\)-algebra \(\mathcal{F}\) such that \(|\mathcal{I} \cap \mathcal{I}'| = E(B \,|\, \mathcal{F})\). Therefore, we have by Jensen's inequality 
        \begin{equation}
            E\left( e^{\frac{c^2\psi^2}{m^2\mu_{j^*}}|\mathcal{I} \cap \mathcal{I}'|}\right) \leq \left(1 - \frac{m}{j^*-1} + \frac{m}{j^*-1} e^{\frac{c^2\psi^2}{m^2\mu_{j^*}}}\right)^{m} \leq \exp\left(\frac{m^2}{j^*-1}\left(e^{\frac{c^2\psi^2}{m^2\mu_{j^*}}} - 1\right)\right) \label{eqn:multinomial_chisquare_II}. 
        \end{equation}

        We now split the analysis into two cases. \newline 

        \noindent \textbf{Case 1:} Suppose \(nq_0^{-\max}(j^*) > \log\left(\tilde{C}_\eta j^*\right)\). It follows from Lemma \ref{lemma:h_inverse} that \(m \lesssim 1\). It further follows from Lemma \ref{lemma:h_inverse} that for some universal constant \(L > 0\) whose value can change from instance to instance, 
        \begin{align*}
            (\ref{eqn:multinomial_chisquare_II}) &\leq \exp\left(\frac{L}{j^*}\left(e^{\frac{c^2\psi^2}{m^2\mu_{j^*}}}-1\right)\right) \\
            &\leq \exp\left(\frac{L}{j^*}\left(e^{Lc^2 \log(\tilde{C}_\eta j^*)} - 1\right)\right) \\
            &\leq \exp\left(L^2c^2 \frac{\log(\tilde{C}_\eta j^*)}{j^*} \exp\left(Lc^2 \log(\tilde{C}_\eta j^*)\right)\right) \\
            &\leq \exp\left(c^2\tilde{C}_{\eta}'\right)
        \end{align*}
        for some \(\tilde{C}_\eta'\) depending only on \(\tilde{C}_\eta\). Here, we have used \(m \geq 1\) to obtain the first inequality and \(e^x \leq 1 + xe^x\) for \(x > 0\) to obtain the third inequality. Furthermore, we have taken \(c_\eta\) sufficiently small to ensure \(L c^2 \leq 1\) to obtain the final line. \newline

        \noindent \textbf{Case 2:} Suppose \(nq_0^{-\max}(j^*) \leq \log\left(\tilde{C}_\eta j^*\right)\). By Lemma \ref{lemma:m_size}, we can take \(C_*\) sufficiently large so that \(m \lesssim (j^*)^{1/4}\). Furthermore, since \(m \geq 1\), we can conclude 
        \begin{equation*}
            m \asymp \left\lceil h^{-1}\left(\frac{\log(ej^*)}{nq_0^{-\max}(j^*)}\right) \right\rceil.
        \end{equation*}
        With this in hand, consider
        \begin{align*}
            \frac{c^2\psi^2}{m^2 \mu_{j^*}} &\lesssim c^2 \cdot nq_0^{-\max}(j^*) \cdot \left(\frac{h^{-1}\left(\frac{\log(\tilde{C}_\eta j^*)}{nq_0^{-\max}(j^*)} \right)}{\left\lceil h^{-1}\left(\frac{\log(ej^*)}{nq_0^{-\max}(j^*)}\right) \right\rceil}\right)^2 \\
            &\leq \left(\frac{h^{-1}\left(\frac{\log(\tilde{C}_\eta j^*)}{nq_0^{-\max}(j^*)} \right)}{\left\lceil h^{-1}\left(\frac{\log(ej^*)}{nq_0^{-\max}(j^*)}\right) \right\rceil}\right)^2 \cdot c^2 \log\left(\tilde{C}_\eta j^*\right). 
        \end{align*}
        Since \(m \lesssim (j^*)^{1/4}\) implies \(\log j^* \lesssim \log\left(1 + \frac{j^*}{m^2}\right)\), we can conclude 
        \begin{equation*}
            \frac{c^2\psi^2}{m^2 \mu_{j^*}} \leq c^2 \tilde{C}_\eta''\log\left(1 + \frac{j^*}{m^2}\right)
        \end{equation*}
        for some \(\tilde{C}_\eta'' > 0\) depending only on \(\tilde{C}_\eta\). It follows by taking \(c_\eta\) sufficiently small depending only on \(\tilde{C}_\eta\) that 
        \begin{align*}
            (\ref{eqn:multinomial_chisquare_II}) &\leq \exp\left(\frac{m^2}{j^* - 1} \left(e^{c^2 \tilde{C}_\eta''\log\left(1 + \frac{j^*}{m^2}\right)} - 1\right)\right) \\
            &\leq \exp\left(\frac{m^2}{j^*-1} \left(\left(1 + \frac{c^2\tilde{C}_\eta'' j^*}{m^2}\right) - 1\right)\right) \\
            &\leq \exp\left(c^2\tilde{C}_\eta''\right).
        \end{align*} 
        Here, we have taking \(c_\eta\) sufficiently small to ensure \(c^2\tilde{C}_{\eta}'' < 1\) and we have used the inequality \((1+x)^\delta \leq 1+\delta x\) for \(\delta \in [0, 1]\) and \(x \geq 0\). The analysis for this case is complete. \newline

        \noindent The claimed result follows from putting the cases together. 
    \end{proof}

    \begin{proof}[Proof of Lemma~\ref{lem:control_chi2_multinomial}]
    The claim is clear if $ \exp\left(\frac{c^2\psi^2}{\mu_{j^*}}\right)P\left\{ \Poisson\left(\frac{(\mu_{j^*} + c\psi)^2}{\mu_{j^*}}\right) \leq \mu_{j^*} + \psi \right\} \leq 1$. 
    We therefore assume $\exp\left(\frac{c^2\psi^2}{\mu_{j^*}}\right) P\left\{ \Poisson\left(\frac{(\mu_{j^*} + c\psi)^2}{\mu_{j^*}}\right) \leq \mu_{j^*} + \psi \right\} > 1$. 
        The analysis is split into two cases. Let \(0 < \tilde{c}_\eta < 1\) denote a sufficiently small constant depending only on \(\tilde{C}_\eta\). Further, let us take \(c < c_\eta\) with \(c_\eta < \frac{1}{2}\) to be set. \newline 

        \noindent \textbf{Case 1:} Suppose \(nq_0^{-\max}(j^*) > \tilde{c}_\eta \log\left(\tilde{C}_\eta j^*\right)\). Consider by Lemma \ref{lemma:h_inverse} that $m \leq C(\eta)$ for some constant $C(\eta)>0$ and that 
        \begin{align*}
            \frac{c^2\psi^2}{\mu_{j^*}} = c^2 nq_0^{-\max}(j^*) h^{-1}\left(\frac{\log(\tilde{C}_\eta j^*)}{nq_0^{-\max}(j^*)}\right)^2 \leq c^2 \tilde{C}_\eta'\log(ej^*)     
        \end{align*}
        where \(\tilde{C}_\eta'\) depends only on \(\tilde{C}_\eta\). Therefore,  
        \begin{align*}
            &1 + \frac{1}{j^*-m} \left( e^{\frac{c^2\psi^2}{\mu_{j^*}}}P\left\{ \Poisson\left(\frac{(\mu_{j^*} + c\psi)^2}{\mu_{j^*}}\right) \leq \mu_{j^*} + \psi \right\} - 1 \right)  \\
            &\leq 1 + \frac{j^*}{j^*-m} \cdot \frac{1}{j^*} \left(\exp\left( c^2\tilde{C}_\eta' \log(ej^*)\right) -1\right)\nonumber \\
            & \leq 1+ \left(1+\frac{m}{j^* - m}\right) \cdot \frac{1}{j^*} \left(\exp\left( c^2\tilde{C}_\eta' \log(1+ej^*)\right) -1\right)\nonumber\\
            &\leq 1 + (1+C(\eta)) c^2 \tilde{C}_\eta''\nonumber \\
            &\leq e^{c^2\tilde{C}_\eta''}
        \end{align*}
        for some \(\tilde{C}_\eta'' > 0\) depending only on \(\tilde{C}_\eta\). We have also taken \(c_\eta\) sufficiently small such that \(c^2\tilde{C}_\eta' < 1\) and used the bound \((1+y)^\delta \leq 1+\delta y\) for \(y \geq 0\) and \(\delta \in (0, 1)\). \newline 

        \noindent \textbf{Case 2:} Suppose \(nq_0(j^*) \leq \tilde{c}_\eta \log\left(\tilde{C}_\eta j^*\right)\). Let us split into two further subcases. \newline 

        \textbf{Case 2.1:} Suppose \(\psi \leq c^{-2}\mu_{j^*}\). Then by Lemma \ref{lemma:h_inverse} and \(m \geq 1\), we have 
        \begin{align*}
            e^{\frac{c^2\psi^2}{\mu_{j^*}}} &\leq e^\psi \\
            &\leq \exp\left( L \frac{\log(\tilde{C}_\eta j^* )}{\log\left(\frac{e\log(\tilde{C}_\eta j^*)}{nq_0^{-\max}(j^*)}\right)} \right) \\
            &\leq \exp\left(\frac{L  \log(\tilde{C_\eta})}{\log(\tilde{c}_\eta^{-1})} \cdot \log\left(ej^*\right) \right)
        \end{align*}
        where \(L > 0\) is a universal constant whose value can change from line to line. By taking \(\tilde{c}_\eta\) sufficiently small depending only on \(\tilde{C}_\eta\) and arguing as in Case 1, we have 
        \begin{equation*}
            1 + \frac{1}{j^*-m} e^{\frac{c^2\psi^2}{\mu_{j^*}}}P\left\{ \Poisson\left(\frac{(\mu_{j^*} + c\psi)^2}{\mu_{j^*}}\right) \leq \mu_{j^*} + \psi \right\}  \leq \exp\left(c^2 \tilde{C}_{\eta}'''\right)
        \end{equation*}
        for some \(\tilde{C}_{\eta}'''\) depending only on \(\tilde{C}_\eta\). The analysis for this subcase is complete. \newline 

        \textbf{Case 2.2:} Suppose \(\psi > c^{-2} \mu_{j^*}\). Consider \(W_J \,|\, J \sim \Poisson\left(\frac{(\mu_{j^*} + c\psi)^2}{\mu_{j^*}}\right)\). Note that \(\frac{(\mu_{j^*} + c\psi)^2}{\mu_{j^*}} = \mu_{j^*} + 2c\psi + \frac{c^2\psi^2}{\mu_{j^*}} > \mu_{j^*} + \psi\) since \(\psi > c^{-2}\mu_{j^*}\). This is important as we are now able to apply Lemma~\ref{lemma:Bennett}. Doing so yields
        \begin{align*}
            e^{\frac{c^2\psi^2}{\mu_{j^*}}} P\left\{ \left. \max_{1 \leq j \leq j^*} W_j - \mu_{j^*} \leq \psi \right| \rho, \rho' \right\} &\leq e^{\frac{c^2\psi^2}{\mu_{j^*}}} P\left\{W_J \leq \mu_{j^*} + \psi \,|\, J\right\} \\
            &= e^{\frac{c^2\psi^2}{\mu_{j^*}}} P\left\{\Poisson\left(\frac{(\mu_{j^*} + c\psi)^2}{\mu_{j^*}}\right) \leq \frac{(\mu_{j^*} + c\psi)^2}{\mu_{j^*}}\left(\frac{\mu_{j^*}\psi + \mu_{j^*}^2}{(\mu_{j^*}+c\psi)^2} \right)\right\} \\
            &\leq \exp\left( \frac{c^2\psi^2}{\mu_{j^*}} - \frac{(\mu_{j^*} + c\psi)^2}{\mu_{j^*}} h\left(-1 + \frac{\mu_{j^*}\psi + \mu_{j^*}^2}{(\mu_{j^*} + c\psi)^2}  \right)\right).
        \end{align*}
        Consider that 
        \begin{align*}
            &\frac{(\mu_{j^*} + c\psi)^2}{\mu_{j^*}} h\left(-1 + \frac{\mu_{j^*}\psi + \mu_{j^*}^2}{(\mu_{j^*} + c\psi)^2}  \right) \\
            &= \frac{(\mu_{j^*} + c\psi)^2}{\mu_{j^*}} \left( \frac{\mu_{j^*}\psi + \mu_{j^*}^2}{(\mu_{j^*} + c\psi)^2} \log\left(\frac{\mu_{j^*}\psi + \mu_{j^*}^2}{(\mu_{j^*} + c\psi)^2}\right) + 1 - \frac{\mu_{j^*} \psi + \mu_{j^*}^2}{(\mu_{j^*} + c\psi)^2}  \right) \\
            &= \mu_{j^*}  \left(1 + \frac{\psi}{\mu_{j^*}}\right) \log\left(1 + \frac{\psi}{\mu_{j^*}}\right) - 2 \mu_{j^*}\left(1 + \frac{\psi}{\mu_{j^*}}\right)\log\left(1 + \frac{c\psi}{\mu_{j^*}}\right) + \frac{(\mu_{j^*} + c\psi)^2}{\mu_{j^*}} - \mu_{j^*}\left(1 + \frac{\psi}{\mu_{j^*}}\right) \\
            &= \mu_{j^*} h\left(\frac{\psi}{\mu_{j^*}}\right) - 2\mu_{j^*}\left(1 + \frac{\psi}{\mu_{j^*}}\right) \log\left(1 + \frac{c\psi}{\mu_{j^*}}\right) + 2c\psi + \frac{c^2\psi^2}{\mu_{j^*}}.
        \end{align*}
        Therefore, we have
        \begin{align*}
            &1 + \frac{1}{j^*-m} e^{\frac{c^2\psi^2}{\mu_{j^*}}}P\left\{ \Poisson\left(\frac{(\mu_{j^*} + c\psi)^2}{\mu_{j^*}}\right) \leq \mu_{j^*} + \psi \right\} \\
            &\leq 1 + \frac{1}{j^*-m} \exp\left(-\mu_{j^*} h\left(\frac{\psi}{\mu_{j^*}}\right) + 2\mu_{j^*} \left(1 + \frac{\psi}{\mu_{j^*}}\right) \log\left(1 + \frac{c\psi}{\mu_{j^*}}\right) - 2c\psi\right) \\
            &\leq 1 + \frac{1}{j^*-m} \exp\left(-\mu_{j^*} h\left(\frac{\psi}{\mu_{j^*}}\right) + 2\mu_{j^*} \left(1 + \frac{\psi}{\mu_{j^*}}\right) \log\left(1 + \frac{c\psi}{\mu_{j^*}} \right)\right) \\
            &\leq 1 + \frac{j^*}{j^*-m}\exp\left( \mu_{j^*} g\left(\frac{\psi}{\mu_{j^*}}\right) + \mu_{j^*} h\left(\frac{\psi}{\mu_{j^*}}\right) - \log j^*\right) \\
            &= 1 + \frac{j^*}{j^*-m} \exp\left(\mu_{j^*}  g\left(\frac{\psi}{\mu_{j^*}}\right) + \log(\tilde{C}_\eta)\right)
        \end{align*}
        where \(g : [0, \infty) \to \R\) is the function \(g(x) = -2h(x) + 2(1+x)\log(1+c x)\). We have used that \(\mu_{j^*} h\left(\frac{\psi}{\mu_{j^*}}\right) = \mu_{j^*} h\left(h^{-1}\left(\frac{\log(\tilde{C}_\eta j^*)}{nq_0^{-\max}(j^*)}\right)\right) = \log(\tilde{C}_\eta j^*)\) to obtain the final line. Consider there exists \(c_\eta\) sufficiently small depending only on \(\tilde{C}_\eta\) such that for all \(x \geq \frac{1}{c^2}\), we have \(g(x) \leq x\left(-2\log(\tilde{C}_\eta) - \log\left(\frac{j^*}{j^*-m}\right)\right)\). This is immediately seen by noting 
        \begin{align*}
            g(x) = -2(1+x)\log\left(\frac{1+x}{1+cx}\right) + 2x \leq x \left( -2 \log\left(\frac{1+x}{1+cx}\right) + 2\right).
        \end{align*}
        Taking \(c_\eta\) sufficiently small clearly yields the desired property since \(c < c_\eta\) and \(c_\eta\) need only depend on \(\tilde{C}_\eta\) since \(m \lesssim (j^*)^{1/4}\) by Lemma \ref{lemma:m_size}. Since \(\frac{\psi}{\mu_{j^*}} \geq \frac{1}{c^2}\) and \(\psi \geq 1\) (since \(C_*\) and \(\tilde{C}_\eta\) can be taken sufficiently large to ensure it by way of condition (\ref{eqn:multinomial_psi_large})), it immediately follows that 
        \begin{align*}
            &1 + \frac{j^*}{j^*-m}\exp\left( \mu_{j^*} g\left(\frac{\psi}{\mu_{j^*}}\right) + \log(\tilde{C}_\eta)\right) \\
            &\leq 1 + \frac{j^*}{j^*-m}\exp\left(\mu_{j^*} \cdot \frac{\psi}{\mu_{j^*}}\left(-2\log\left(\tilde{C}_\eta\right) -\log\left(\frac{j^*}{j^*-m}\right)\right) + \log(\tilde{C}_\eta)\right) \\
            &\leq 1 + \frac{1}{\tilde{C}_\eta}. 
        \end{align*}
        The analysis for this case is complete, and so the proof is complete.         
    \end{proof}

    \section{Asymptotic constant: Poisson}\label{appendix:poisson_sharp_constant}
    In this section, we prove Theorem \ref{thm:poisson_sharp_constant}. Section \ref{section:poisson_sharp_constant_upper_bound} proves the upper bound and Section \ref{section:poisson_sharp_constant_lower_bound} proves the lower bound. 

    \subsection{Proof of (i) in Theorem \ref{thm:poisson_sharp_constant}}\label{section:poisson_sharp_constant_upper_bound}
    \begin{lemma}\label{lemma:epsilon_diverges}
        Suppose \(\mu_1 \geq ... \geq \mu_p \geq 1\), and let $j^*$ and $\epsilon$ be defined as in~\eqref{eq_def_jstar_poisson_sharp_constant} and~\eqref{eq_def_eps_poisson_sharp_constant}, respectively. 
        Assume also \(\frac{\log j^*}{(\log \alpha_p)(\log\log j^*)} \to \infty\). If \(\frac{\log j^*}{\mu_{j^*}} \to 0\) or \(\frac{\log j^*}{\mu_{j^*}} \to \infty\), then \(\epsilon \to \infty\). 
    \end{lemma}
    \begin{proof}
        If \(\frac{\log j^*}{\mu_{j^*}} \to 0\), then consider by Lemma \ref{lemma:h_asymptotics} we have \(\epsilon \sim \xi \sqrt{2 \mu_{j^*} \log j^*}\). Since \(\mu_{j^*} \geq 1\) and \(\log j^* \to \infty\) (since \(\frac{\log j^*}{\log \alpha_p} \to \infty\) and \(\alpha_p \to \infty\)), it follows \(\epsilon \to \infty\). On the other hand, if \(\frac{\log j^*}{\mu_{j^*}} \to \infty\), it is clear that \(\epsilon \to \infty\) since \(\mu_{j^*} \geq 1\) and \(\lim_{x \to \infty} h^{-1}(x) = \infty\).
    \end{proof}
    
    \begin{proof}[Proof of (i) in Theorem \ref{thm:poisson_sharp_constant}]
        Fix \(\xi > 1\). For ease of notation, let \(\psi = \frac{\epsilon}{\xi}\), and consider the test 
        \begin{equation*}
            \varphi = \mathbbm{1}\left\{||X - \mu||_\infty \geq \psi \right\}.
        \end{equation*}
        To see that the Type I error vanishes, observe by union bound and Lemma \ref{lemma:Bennett}, 
        \begin{align*}
            P_\mu\left\{\varphi = 1\right\} \leq \sum_{j=1}^{p} 2e^{-\log(ej\alpha_p\log^2(ej))} = \frac{1}{\alpha_p} \sum_{j=1}^{p} \frac{2}{ej\log^2(ej)} \to 0 
        \end{align*}
        since \(\sum_{j=1}^{\infty} \frac{1}{j\log^2(ej)} < \infty\) and \(\alpha_p \to \infty\). Let us now turn to the Type II error. Suppose \(\lambda \in \Lambda(\mu, \epsilon)\). Then there exists \(j'\) such that \(|\lambda_{j'} - \mu_{j'}| \geq \epsilon\). Observe we can apply Chebyshev's inequality since \(\xi > 1\) to obtain, 
        \begin{equation*}
            P_\lambda\left\{\varphi = 0\right\} \leq P_\lambda\left\{|\lambda_{j'} - \mu_{j'}| - \psi \leq |X_{j'} - \lambda_{j'}| \right\} \leq \frac{\lambda_{j'}}{\left(|\lambda_{j'} - \mu_{j'}| - \psi\right)^2} \leq \frac{1}{|\lambda_{j'} - \mu_{j'}| \left(1 - \frac{1}{\xi}\right)^2} + \frac{\mu_{j'}}{\psi^2 (\xi - 1)^2}.
        \end{equation*}
        The first term vanishes uniformly over \(\lambda \in \Lambda(\mu, \epsilon)\) since \(\epsilon \to \infty\) by Lemma \ref{lemma:epsilon_diverges}. To show the second term vanishes, consider \(\psi^2 \geq \left(\mu_{j'} h^{-1}\left( \frac{\log(ej'\alpha_p\log^2(ej'))}{\mu_{j'}} \right)\right)^2 \gtrsim \mu_{j'} \log(ej'\alpha_p \log^2(ej'))\) since \(h(x) \gtrsim \sqrt{x}\) by Lemma \ref{lemma:h_inverse}. Since \(\alpha_p \to \infty\), it follows \(\frac{\mu_{j'}}{\psi^2 (\xi - 1)^2} \to 0\) uniformly over \(\lambda \in \Lambda(\mu, \epsilon)\). The proof is complete. 
    \end{proof}

    \subsection{Proof of (ii) in Theorem \ref{thm:poisson_sharp_constant}}\label{section:poisson_sharp_constant_lower_bound}
    In this section, we present the proof of the lower bound for the sharp asymptotic constant, namely item \textit{(ii)} in Theorem \ref{thm:poisson_sharp_constant}. Essentially the same argument of Theorem \ref{thm:lower_bound} can be used, with more care to track constants. We will use the prior \(\pi\) in which a draw \(\lambda \sim \pi\) is obtained by drawing \(J \sim \Uniform(\{1,...,j^*\})\) and setting \(\lambda_j = \mu_j + \epsilon \mathbbm{1}_{\{j = J\}}\). It is clear \(\pi\) is supported on \(\Lambda(\mu, \epsilon)\).

    As asserted by the following proposition, we reduce to the auxiliary homoskedastic version of the testing problem by applying Proposition \ref{prop:general_flattening}. The proof is omitted as the result can be established by the same proof of Proposition \ref{prop:poisson_flattening}. 
    \begin{proposition}\label{prop:asymptotic_flattening}
        We have 
        \begin{equation*}
            \dTV\left(P_\mu, P_\pi\right) \leq \dTV\left(\Poisson(\mu_{j^*})^{\otimes j^*}, \frac{1}{j^*}\sum_{J=1}^{j^*} \bigotimes_{j = 1}^{j^*} \Poisson(\mu_{j^*} + \epsilon \mathbbm{1}_{\{j=J\}}) \right),
        \end{equation*}
        where \(P_\pi = \int P_\lambda \, d\pi\) is the mixture induced by \(\pi\). 
    \end{proposition}

    As in the proof of Theorem \ref{thm:lower_bound}, we proceed by the conditional second moment method. For notational ease, define \(P_\mu^* = \Poisson(\mu_{j^*})^{\otimes j^*}\) and \(P_\pi^* = \int \otimes_{j=1}^{j^*} \Poisson(\mu_{j^*} + \epsilon \mathbbm{1}_{\{j = J\}}) \, d\pi\). For notational clarity, denote the data coming from either distribution as \(V\), and we will condition on the event 
    \begin{equation*}
        E := \left\{ \max_{1 \leq j \leq j^*} V_j - \mu_{j^*} \leq \psi\right\},
    \end{equation*}
    where \(\psi = \frac{\epsilon}{\xi}\). Denote \(\tilde{P}_\pi\) and \(\tilde{P}_\mu\) to be the conditional distributions \(P_{\pi}^*(\cdot\,|\, E)\) and \(P_\mu^*(\cdot\,|\, E)\) respectively. 

    \begin{lemma}\label{lemma:asymptotic_poisson_second_moment}
        If \(\xi < 1\) and either \(\frac{\log j^*}{\mu_j^*} \to 0\) or \(\frac{\log j^*}{\mu_{j^*}} \to \infty\), then 
        \begin{equation*}
            \dTV\left(P_\mu^*, P_\pi^*\right) \leq \frac{1}{2}\sqrt{\chi^2(\tilde{P}_\pi\,||\,\tilde{P}_\mu)} + o(1)
        \end{equation*}
        where \(o(1) \to 0\) as \(p \to \infty\). 
    \end{lemma}
    \begin{proof}
        By union bound and Lemma \ref{lemma:Bennett}, 
        \begin{equation*}
            P_\mu^*(E^c) \leq j^* \exp\left(-\mu_{j^*}h\left(\frac{\psi}{\mu_{j^*}}\right)\right) = \exp\left(-\log(ej^*\alpha_p\log^2(ej^*))+\log j^*\right) = o(1)
        \end{equation*} 
        since \(\alpha_p \to \infty\). Likewise, we can apply Lemma \ref{lemma:Bennett} since \(\xi < 1\), we have 
        \begin{align*}
            P_\pi^*(E^c) &\leq o(1) + P\left\{\Poisson\left(\mu_{j^*} + \epsilon \right) - \mu_{j^*} > \psi \right\} \\
            &= o(1) + P\left\{\Poisson\left(\mu_{j^*} + \epsilon \right) > \left(\mu_{j^*} + \epsilon \right) \left(1 + \frac{(1 - \xi)\psi}{\mu_{j^*} + \epsilon}\right) \right\} \\
            &\leq o(1) + \exp\left( - \left(\mu_{j^*} + \epsilon \right)h\left(\frac{(1-\xi)\psi}{\mu_{j^*} + \epsilon}\right) \right).
        \end{align*}
        If \(\liminf_{p \to \infty} \frac{(1-\xi)\psi}{\mu_{j^*} + \epsilon} > 0\), then it follows from \(\epsilon \to \infty\) given by Lemma \ref{lemma:epsilon_diverges} that \(\left(\mu_{j^*} + \epsilon \right)h\left(\frac{(1-\xi)\psi}{\mu_{j^*} + \epsilon}\right) \to \infty\). On the other hand, if \(\frac{(1-\xi)\psi}{\mu_{j^*} + \epsilon} \to 0\), then it follows from Lemma \ref{lemma:h_asymptotics} that \(h(x) \sim \frac{x^2}{2}\) as \(x \to 0\), and so \(\left(\mu_{j^*} + \epsilon \right)h\left(\frac{(1-\xi)\psi}{\mu_{j^*} + \epsilon}\right) \sim \frac{(1-\xi)^2\psi^2}{\mu_{j^*} + \xi \psi}\). If \(\frac{\log j^*}{\mu_{j^*}} \to 0\), then \(\psi^2 \sim 2\xi^2 \mu_{j^*}\log j^*\) by Lemma \ref{lemma:h_asymptotics}, and so \(\frac{(1-\xi)^2\psi^2}{\mu_{j^*} + \xi \psi} \to \infty\). If \(\frac{\log j^*}{\mu_{j^*}} \to \infty\), then \(\psi^2 \sim \left(\frac{\log j^*}{\log\left(\frac{\log j^*}{\mu_{j^*}}\right)}\right)^2\). Since \(\frac{\log j^*}{\mu_{j^*}} \to \infty\), it follows \(\frac{(1-\xi)^2\psi^2}{\mu_{j^*} + \xi \psi} \to \infty\). Hence, we have shown \(P_\pi(E^c) = o(1)\). Thus by triangle inequality and Lemma \ref{lemma:conditional_TV}, we have 
        \begin{equation*}
            \dTV\left(P_\mu^*, P_\pi^*\right) \leq \dTV\left(\tilde{P}_\mu, \tilde{P}_\pi\right) + 2\dTV\left(\tilde{P}_\mu, P^*_\mu\right) + 2\dTV\left(\tilde{P}_\pi, P_\pi^*\right) \leq \frac{1}{2}\sqrt{\chi^2(\tilde{P}_\pi \,||\, \tilde{P}_\mu)} + o(1). 
        \end{equation*}
    \end{proof}
    
    \begin{lemma}\label{lemma:asymptotic_poisson_conditional}
        If \(\xi < 1\) and either \(\frac{\log j^*}{\mu_j^*} \to 0\) or \(\frac{\log j^*}{\mu_{j^*}} \to \infty\), then 
        \begin{equation*}
            \chi^2(\tilde{P}_\pi \,||\, \tilde{P}_\mu) + 1\leq (1+o(1)) E\left(\exp\left(\frac{\epsilon^2}{\mu_{j^*}} \mathbbm{1}_{\{J = J'\}} P\left\{\left.\max_{1 \leq j \leq j^*} W_j - \mu_{j^*} \leq \psi \,\right|\, \rho, \rho'\right\}\right) \right)
        \end{equation*}
        where \(\rho, \rho' \overset{iid}{\sim} \pi\) and \(J, J'\) are corresponding random indices, and \(W_j \,|\, \rho, \rho' \overset{ind}{\sim} \Poisson\left(\frac{\rho_j \rho_j'}{\mu_{j^*}}\right)\).
    \end{lemma}
    \begin{proof}
        The result can be obtained by noting \(P_\mu^*(E^c), P_\pi^*(E^c) = o(1)\) and following the calculations in the proof of Lemma \ref{lemma:conditional_second_moment}.
    \end{proof}

    \begin{proposition}\label{prop:asymptotic_poisson_mgf}
        If \(\xi < 1\) and either \(\frac{\log j^*}{\mu_j^*} \to 0\) or \(\frac{\log j^*}{\mu_{j^*}} \to \infty\), then 
        \begin{equation*}
             E\left(\exp\left(\frac{\epsilon^2}{\mu_{j^*}} \mathbbm{1}_{\{J = J'\}}\right) P\left\{\left.\max_{1 \leq j \leq j^*} W_j - \mu_{j^*} \leq \psi \,\right|\, \rho, \rho'\right\}\right) = 1 + o(1). 
        \end{equation*}
    \end{proposition}
    \begin{proof}
        We break up the analysis into two cases. \newline 

        \noindent \textbf{Case 1:} Suppose \(\frac{\log j^*}{\mu_{j^*}} \to 0\). It follows by Lemma \ref{lemma:h_asymptotics} that \(\frac{\epsilon^2}{\mu_{j^*}} \sim 2\xi^2\log j^*\). If \(\xi < \frac{1}{\sqrt{2}}\), then we directly have 
        \begin{align*}
            E\left(\exp\left(\frac{\epsilon^2}{\mu_{j^*}} \mathbbm{1}_{\{J = J'\}}\right) P\left\{\left.\max_{1 \leq j \leq j^*} W_j - \mu_{j^*} \leq \psi \,\right|\, \rho, \rho'\right\}\right) &\leq 1 - \frac{1}{j^*} + \frac{1}{j^*} e^{2\xi^2(1+o(1))\log j^*} \\
            &\leq 1 + \exp\left((2\xi^2 - 1)(1+o(1))\log j^*\right) \\
            & = 1+o(1).
        \end{align*}
        Suppose \(\xi \geq \frac{1}{\sqrt{2}}\). Then \(\frac{(\mu_{j^*} + \epsilon)^2}{\mu_{j^*}} = \mu_{j^*} + 2\epsilon + \frac{\epsilon^2}{\mu_{j^*}} > \mu_{j^*} + \psi\), and so we can apply Lemma \ref{lemma:Bennett} to obtain 
        \begin{align*}
            &E\left(\exp\left(\frac{\epsilon^2}{\mu_{j^*}} \mathbbm{1}_{\{J = J'\}}\right) P\left\{\left.\max_{1 \leq j \leq j^*} W_j - \mu_{j^*} \leq \psi \,\right|\, \rho, \rho'\right\}\right) \\
            &\leq 1 - \frac{1}{j^*} + \frac{1}{j^*} e^{\frac{\epsilon^2}{\mu_{j^*}}} P\left\{\Poisson\left(\frac{(\mu_{j^*} + \epsilon)^2}{\mu_{j^*}}\right) \leq \mu_{j^*} + \psi \right\} \\
            &\leq 1 - \frac{1}{j^*} + \frac{1}{j^*} e^{\frac{\epsilon^2}{\mu_{j^*}}} P\left\{\Poisson\left(\mu_{j^*} + 2\epsilon\right) \leq \mu_{j^*} + \psi \right\} \\
            &\leq 1 - \frac{1}{j^*} + \frac{1}{j^*} e^{\frac{\epsilon^2}{\mu_{j^*}}} P\left\{\Poisson\left(\mu_{j^*} + 2\epsilon\right) \leq \left(\mu_{j^*} + 2\epsilon\right)\left(1 + \frac{(1-2\xi)\psi}{\mu_{j^*} + 2\epsilon}\right) \right\} \\
            &\leq 1 - \frac{1}{j^*} + \frac{1}{j^*} \exp\left(\frac{\epsilon^2}{\mu_{j^*}}\right) \exp\left(-(\mu_{j^*} + 2\epsilon)h\left(\frac{(1-2\xi)\psi}{\mu_{j^*} + 2\epsilon}\right)\right).
        \end{align*}
        Since \(\psi \sim \sqrt{2\mu_{j^*}\log j^*}\) and \(\frac{\log j^*}{\mu_{j^*}} \to 0\), it follows from \(h(x) \sim \frac{x^2}{2}\) as \(x \to 0\) that 
        \begin{equation*}
            \left(\mu_{j^*} + 2\epsilon\right)h\left(\frac{(1-2\xi)\psi}{\mu_{j^*} + 2\epsilon}\right) \sim \frac{\left(1 - 2\xi\right)^2}{2} \frac{\psi^2}{\mu_{j^*} + 2\epsilon} \sim \left(1 - 2\xi\right)^2 \log j^*. 
        \end{equation*}
        Since \(\frac{\epsilon^2}{\mu_{j^*}} \sim 2\xi^2 \log j^*\), we thus have 
        \begin{align*}
            &\frac{1}{j^*} \exp\left(\frac{\epsilon^2}{\mu_{j^*}}\right) \exp\left(-(\mu_{j^*} + 2\epsilon)h\left(\frac{(1-2\xi)\psi}{\mu_{j^*} + 2\epsilon}\right)\right) \\
            &= \exp\left((1+o(1))\left(-1 + 2\xi^2 - \left(1 - 2\xi\right)^2 \right)\log j^* \right) \\
            &\to 0
        \end{align*}
        as \(p \to \infty\) since \(\xi < 1\). Hence, \(E\left(\exp\left(\frac{\epsilon^2}{\mu_{j^*}} \mathbbm{1}_{\{J = J'\}}\right) P\left\{\left.\max_{1 \leq j \leq j^*} W_j - \mu_{j^*} \leq \psi \,\right|\, \rho, \rho'\right\}\right) = o(1)\). 
        \newline
        
        \noindent \textbf{Case 2:} Suppose \(\frac{\log j^*}{\mu_{j^*}} \to \infty\). Consider \(\frac{(\mu_{j^*} + \epsilon)^2}{\mu_{j^*}} > \mu_{j^*} + \psi\) for all \(p\) sufficiently large since \(\frac{\psi}{\mu_{j^*}} \sim \frac{(\log j^*)/\mu_{j^*}}{\log((\log j^*)/\mu_{j^*})} \to \infty\) since \(\frac{\log j^*}{\mu_{j^*}} \to \infty\). Therefore, we can apply Lemma \ref{lemma:Bennett}. Following the calculations of Lemma \ref{lemma:bennett_cancellation}, we obtain for all \(p\) sufficiently large, 
        \begin{align*}
            &E\left(\exp\left(\frac{\epsilon^2}{\mu_{j^*}} \mathbbm{1}_{\{J = J'\}} P\left\{\left.\max_{1 \leq j \leq j^*} W_j - \mu_{j^*} \leq \psi \,\right|\, \rho, \rho'\right\}\right) \right) \\
            &\leq 1 + \frac{1}{j^*} e^{\frac{\epsilon^2}{\mu_{j^*}}} P\left\{\Poisson\left(\frac{(\mu_{j^*} + \epsilon)^2}{\mu_{j^*}}\right)\leq \mu_{j^*} + \psi\right\} \\
            &\leq 1 + \frac{1}{j^*}\exp\left(-\mu_{j^*}h\left(\frac{\psi}{\mu_{j^*}}\right) + 2\mu_{j^*}\left(1 + \frac{\psi}{\mu_{j^*}}\right)\log\left(1 + \frac{\xi \psi}{\mu_{j^*}}\right) \right) \\
            &= 1 + \exp\left(\mu_{j^*}g\left(\frac{\psi}{\mu_{j^*}}\right) + \log(e\alpha_p\log^2(ej^*))\right)
        \end{align*}
        where \(g(x) := -2h(x) + 2(1+x)\log(1 + \xi x) - 2\xi x\). We have used \(\mu_{j^*}h\left(\frac{\psi}{\mu_{j^*}}\right) = \log(ej^*\alpha_p\log^2(ej^*)) = \log j^* + \log(e\alpha_p\log^2(ej^*))\). Consider 
        \begin{align*}
            g(x) = -2(1+x)\log\left(\frac{1+x}{1+\xi x}\right) + 2x - 2\xi x \sim 2x\left(-\log\left(\frac{1}{\xi}\right) - \xi + 1\right)
        \end{align*}
        as \(x \to \infty\). Consider that \(\log(1/t) + t > 1\) for all \(t \in (0, 1)\). Since \(\xi < 1\), we have \(-\log\left(\frac{1}{\xi}\right) - \xi + 1 < 0\), and so \(\lim_{x \to \infty} g(x) = -\infty\). Since \(\frac{\psi}{\mu_{j^*}} \to \infty\), it immediately follows that 
        \begin{align*}
            &\exp\left(\mu_{j^*} g\left(\frac{\psi}{\mu_{j^*}}\right) + \log(e\alpha_p\log^2(ej^*))\right) \\
            &= \exp\left(2\left(-\log\left(\frac{1}{\xi}\right) - \xi + 1\right)(1+o(1))\psi + \log(e\alpha_p\log^2(ej^*)) \right) \\
            &= o(1).
        \end{align*}
        Here, we have used \(\psi \sim \log j^* / \log\left(\frac{\log j^*}{\mu_{j^*}}\right)\) grows faster than \(\log(e\alpha_p \log^2(ej^*))\) because we have assumed \(\mu_{j^*} \geq 1\) and \(\log j^* /((\log \alpha_p)(\log\log j^*)) \to \infty\). The proof is complete. 
    \end{proof}

    \begin{proof}[Proof of (ii) in Theorem \ref{thm:poisson_sharp_constant}]
        Fix \(\xi < 1\). Then by Proposition \ref{prop:asymptotic_flattening}, Lemma \ref{lemma:asymptotic_poisson_second_moment}, Lemma \ref{lemma:asymptotic_poisson_conditional}, and Proposition \ref{prop:asymptotic_poisson_mgf},
        \begin{align*}
            \lim_{p \to \infty} \mathcal{R}_{\mathcal{P}}(\epsilon, \mu) &\geq \lim_{p \to \infty} \left(1 - \dTV\left(P_\mu, P_\pi\right)\right) \geq \lim_{p \to \infty}\left(1 - \frac{1}{2}\sqrt{\chi^2\left(\tilde{P}_{\pi} \,||\, \tilde{P}_\mu\right)}\right) = 1,
        \end{align*}
        as claimed. 
    \end{proof}

    \section{Asymptotic constant: multinomial}\label{appendix:multinomial_sharp_constant}
    In this section, we prove Theorem \ref{thm:multinomial_sharp_constant}. Section \ref{section:multinomial_sharp_constant_upper_bound} proves the upper bound and Section \ref{section:multinomial_sharp_constant_lower_bound} proves the lower bound. 

    \subsection{Proof of (i) in Theorem \ref{thm:multinomial_sharp_constant}}\label{section:multinomial_sharp_constant_upper_bound}

    \begin{lemma}\label{lemma:multinomial_epsilon_diverges}
        Suppose \(q_0(1) \geq ... \geq q_0(p) \geq \frac{1}{n}\) and \(\frac{\log j^*}{(\log \alpha_p)(\log\log j^*)} \to \infty\), and let $j^*$ and $\epsilon$ be defined as in~\eqref{eq_def_jstar_mult_sharp_constant} and~\eqref{eq_def_eps_mult_sharp_constant}, respectively. If \(\frac{\log j^*}{\mu_{j^*}} \to 0\) or \(\frac{\log j^*}{\mu_{j^*}} \to \infty\), then \(\epsilon \to \infty\). 
    \end{lemma}
    \begin{proof}
        The result follows from Lemma \ref{lemma:epsilon_diverges} by taking \(\mu_j = nq_0^{-\max}(j)(1-q_0^{-\max}(j))\). 
    \end{proof}
    
    \begin{proof}[Proof of (i) in Theorem \ref{thm:multinomial_sharp_constant}]
        Fix \(\xi > 1\). For ease of notation, let \(\psi = \frac{\epsilon}{\xi}\), and consider the test \(\varphi = \mathbbm{1}\{||X - nq_0||_\infty \geq n'\psi\}\). By union bound and Corollary \ref{corollary:bennett_binomial}, the Type I error can be shown to vanish as follows, 
        \begin{align*}
            P_{q_0}\left\{\varphi = 1\right\} &\leq P_{q_0}\left\{|X_1 - nq_0(1)| \geq n'\psi\right\} + \sum_{j=2}^{p} P\left\{|X_j - nq_0(j)| \geq n'\psi\right\} \\
            &\leq \frac{q_0(1)(1-q_0(1))}{n'\psi^2} + \sum_{j=2}^{p} 2e^{-\log\left(ej\alpha_p\log^2(ej)\right)} \\
            &= \xi^2\frac{q_0(1)(1-q_0(1))}{n' \epsilon^2} + \frac{1}{\alpha_p} \sum_{j=2}^{p} \frac{2}{ej \log^2(ej)} \\
            &= o(1)
        \end{align*}
        since \(\frac{\epsilon}{\sqrt{\frac{q_0^{\max}(1-q_0^{\max})}{n}}} \to \infty\), \(\alpha_p \to \infty\), and \(\sum_{j=1}^{\infty} \frac{1}{j\log^2(ej)} < \infty\). To show the Type II error vanishes, fix \(q \in \Pi(q_0, \epsilon)\). Then there exists \(j'\) such that \(|q(j') - q_0(j')| \geq \epsilon\). If \(j' \geq 2\), then the proof of (i) in Theorem \ref{thm:poisson_sharp_constant} can be essentially repeated to show \(P_{q}\{\varphi = 1\} = o(1)\) uniformly over all such \(q\) (i.e. \(q \in \Pi(q_0, \epsilon)\) such that \(\max_{j \geq 2} |q_0(j) - q(j)| \geq \epsilon\)). If \(j' = 1\), then observe by Chebyshev's inequality (which can be applied for sufficiently large \(n\) in what follows since \(\xi > 1\)), 
        \begin{align*}
            P_{q}\left\{ \varphi = 1 \right\} &\leq P_{q}\left\{|X_1 - nq_0(1)| \leq n'\psi \right\} \\
            &\leq P_q\left\{ |nq(1) - nq_0(1)| - n'\psi \leq |X_1 - nq(1)| \right\} \\
            &\leq \frac{nq(1)(1-q(1))}{\left(|nq(1)-nq_0(1)| - n'\psi\right)^2} \\
            &\leq \frac{n\left|q(1)(1-q(1)) - q_0(1)(1-q_0(1))\right|}{\left(|nq(1)-nq_0(1)| - n'\psi\right)^2} + \frac{nq_0(1)(1-q_0(1))}{\left(|nq(1)-nq_0(1)| - n'\psi\right)^2} \\
            &\leq \frac{n|q(1) - q_0(1)|}{n^2 |q(1) - q_0(1)|^2 \left(1 - \frac{1}{\xi}\right)^2(1+o(1))} + \frac{nq_0(1)(1-q_0(1))}{n^2\epsilon^2\left(1 - \frac{1}{\xi}\right)^2  (1+o(1))} \\
            &= o(1)
        \end{align*}
        uniformly over all such \(q\) (i.e. \(q \in \Pi(q_0, \epsilon)\) such that \(|q_0(1) - q(1)| \geq \epsilon\)). Here, we have used \(n' = n(1+o(1))\) and the inequality \(|x(1-x) - y(1-y)| \leq |x-y|\) for \(x, y \in [0, 1]\) and \(\epsilon \to \infty\) by Lemma \ref{lemma:multinomial_epsilon_diverges}. The proof is complete.
    \end{proof}

    \subsection{Proof of (ii) in Theorem \ref{thm:multinomial_sharp_constant}}\label{section:multinomial_sharp_constant_lower_bound}

In this section, we present the proof of the lower bound for the sharp asymptotic constant, namely item \textit{(ii)} in Theorem \ref{thm:multinomial_sharp_constant}. The same argument of Theorem \ref{thm:poissonized_multinomial_lower_bound} can be used, with more care to track constants and a slight adjustment for $m$.

    Since $n\to \infty$, we can choose the constant $c$ appearing in Lemma~\ref{lem:relation_multinomial_poisson} as $c = c_n = n^{-1/3}$ so that $c = o(1)$ and $c^2 n \to \infty$. 
    Therefore, we obtain $\mathcal{R}_{\mathcal{M}}(\varepsilon, n, q_0) \geq \mathcal{R}_{\mathcal{PM}}\big(\varepsilon, (1+c_n)n, q_0\big) + o(1)$. 
    We now need to analyze the testing risk $\mathcal{R}_{\mathcal{PM}}\big(\varepsilon, n', q_0\big)$ where $n' = (1+c_n)n$. 
    We recall the definition of $j^*$ and $\epsilon$ from~\eqref{eq_def_jstar_mult_sharp_constant} and~\eqref{eq_def_eps_mult_sharp_constant}, respectively, and define $\psi = n' \epsilon$ as well as
    \begin{align}
            &m = \left(2\lor\left\lceil h^{-1}\left(\frac{\log(ej^*)}{n'q_0^{-\max}(j^*)}\right)\right\rceil \right) \wedge (j^* - 1).\label{def_m_sharp_constant}
        \end{align}         
        
    A draw \(q \sim \pi\) is obtained by first drawing \(J \sim \Uniform(\{2,...,j^*+1\})\), then drawing uniformly at random a size-\(m\) subset \(\mathcal{I} \subset \{2,...,j^*+1\} \setminus \left\{J\right\}\), and finally setting 
    \begin{equation}
        q(j) = 
        \begin{cases}
            q_0(j) + \frac{\psi}{n'}   &\textit{if } j = J,\\
            q_0(j) - \frac{\psi}{n'm} &\textit{if } j \in \mathcal{I}, \\
            q_0(j) &\textit{otherwise},
        \end{cases}\label{eq_prior_sharp_constant_mult}
    \end{equation}
    for \(1 \leq j \leq p\). 
    We recall that $m \lesssim {j^*}^{1/4}$ by proceeding as in Lemma~\ref{lemma:m_size} and that the prior is indeed supported on \(\Pi(q_0, \frac{\psi}{n'})\) by proceeding as in Lemma~\ref{lem:support_prior}.
    The testing risk $\mathcal{R}_{\mathcal{PM}}\big(\varepsilon, n', q_0\big)$ is associated with the testing problem
    \begin{align}
        &H_0: Y \sim \bigotimes_{j=1}^p \Poisson(n'q_0(j)) \label{H0_at_n'}\\
        &H_1: q \sim \pi \text{ and } Y|q \sim \bigotimes_{j=1}^p \Poisson(n'q(j)).\label{H1_at_n'}
    \end{align}
    We aim to prove a lower bound equal to $\xi \cdot q_0^{-\max}(j^*) \,h^{-1}\!\left(\frac{\log(ej^*)}{n'q_0^{-\max}(j^*)}\right)$, which will yield a lower bound of the desired order on the initial problem since we have $n = (1+o(1))n'$, so by Lemma~\ref{lem:removing_logs}
    $$\xi \max_{1\leq j \leq p} q_0^{-\max}(j) \,h^{-1}\!\left(\frac{\log(ej)}{n'q_0^{-\max}(j)}\right) = (1+o(1)) \cdot \xi \max_{1\leq j \leq p} q_0^{-\max}(j) \,h^{-1}\!\left(\frac{\log(ej)}{nq_0^{-\max}(j)}\right).$$
    Similarly as in the proof of Theorem~\ref{thm:poissonized_multinomial_lower_bound}, we will consider the flattened version at sample size $n'$

    \begin{align}
        H_0 &: Y \sim \Poisson\left(n' q_0^{-\max}(j^*)\right)^{\otimes j^*}, \label{H0_flattened_sharp_cst}\\
        H_1 &: q \sim \pi \text{ and } Y|q \sim \bigotimes_{j=1}^{j^*} \Poisson(n'\tilde{q}(j)) \label{H1_flattened_sharp_cst}
    \end{align}
    where \(\tilde{q} \in \R^{j^*}\) is given by \(\tilde{q}(j-1) = q(j) - q_0(j) + q_0^{-\max}(j^*)\) for \(2 \leq j \leq j^*+1\), that is to say, 
    \begin{equation*}
        \tilde{q}(j-1) = 
        \begin{cases}
            q_0^{-\max}(j^*) + c\frac{\psi}{n'}   &\textit{if } j = J,\\
            q_0^{-\max}(j^*) - c\frac{\psi}{n'm} &\textit{if } j \in \mathcal{I}, \\
            q_0^{-\max}(j^*) &\textit{otherwise}.
        \end{cases}
    \end{equation*}

    The following proposition relates the initial testing problem to the flattened one~\eqref{H0_flattened_sharp_cst}-\eqref{H1_flattened_sharp_cst}.

    \begin{proposition}\label{prop:flattening_sharp_constant}
        We have
        \begin{align*}
            &\dTV\left(\bigotimes_{j=1}^{p} \Poisson(n'q_0(j)), \int \bigotimes_{j=1}^{p} \Poisson(n'q(j)) \, \pi(dq) \right) \\
            &\leq \dTV\left(\Poisson(n' q_0^{-\max}(j^*))^{\otimes j^*}, \int \bigotimes_{j=1}^{j^*} \Poisson(n'\tilde q(j)) \, \pi(dq)\right)
        \end{align*}
        provided \(c < c_\eta\). 
    \end{proposition}
    \begin{proof}
        The result will follow from an application of Proposition \ref{prop:general_flattening}. Let \(\gamma = \pi\), \(k = j^*+1\), \(\omega = n'q_0\), and \(\underline{\omega} = n'q_0^{-\max}(j^*)\). It is clear item \((ii)\) of the statement of Proposition \ref{prop:general_flattening} is satisfied. Note in the notation of Proposition \ref{prop:general_flattening}, we have \(\xi = n'q\) and \(\xi_j - \omega_j + \underline{\omega} = n'q_0(j)\mathbbm{1}_{\{j = 1\} \cup \{j > j^*\}} + n'\tilde{q}(j-1)\mathbbm{1}_{\{2 \leq j \leq j^*+1\}}\). 
                If \(m = 0\), then \(\tilde{q} = q_0\) and so there is nothing to prove. If \(m \geq 1\), then
        \begin{align*}
            \tilde{q}(j) &\geq q_0^{-\max}(j^*) - \frac{c\psi}{n'm} \\
            &= q_0^{-\max}(j^*)\left(1 - c \cdot \frac{h^{-1}\left(\frac{\log(e j^*)}{n'q_0^{-\max}(j^*)} \right)}{\left\lceil h^{-1}\left(\frac{\log(ej^*)}{n'q_0^{-\max}(j^*)} \right)\right\rceil }\right) \\
            &\geq 0.
        \end{align*}
        The result then follows from Proposition \ref{prop:general_flattening} since the second term in (\ref{eqn:flattening_split}) is zero.
    \end{proof}

    We now proceed with the conditional second-moment method. For notational ease, let us denote \(P_0 = \Poisson(n' q_0^{-\max}(j^*))^{\otimes j^*}\) and \(P_\pi = \int \bigotimes_{j=1}^{j^*} \Poisson(n'\tilde{q}(j)) \, \pi(dq)\). Denote \(\mu_j = n'q_0^{-\max}(j)\). We will condition on the event
    \begin{equation}\label{def:conditional_event_sharp_cst}
        E := \left\{ \max_{1 \leq j \leq j^*} Y_j - \mu_{j^*} \leq \psi \right\}.
    \end{equation}
    Let \(\tilde{P}_0\) and \(\tilde{P}_{\pi}\) denote the measures \(P_0\) and \(P_\pi\) conditioned on the event \(E\), that is to say, for any event \(A\) we have \(\tilde{P}_0(A) = \frac{P_0(A \cap E)}{P_0(E)}\) and \(\tilde{P}_\pi(A) = \frac{P_\pi(A \cap E)}{P_\pi(E)}\).
    Proceeding as in Lemma~\ref{lemma:asymptotic_poisson_second_moment}, we obtain
    \begin{align}
        \dTV\left(P_0, P_\pi\right) \leq \frac{1}{2}\sqrt{\chi^2(\tilde{P}_\pi\,||\,\tilde{P}_\mu)} + o(1).\label{eq_TV_chi2_mult_sharp_constant}
    \end{align}

        It remains to bound the \(\chi^2\) divergence in the display above. 
        The following lemma reduces the task to bounding two specific probabilistic quantities. 

    \begin{lemma}\label{lemma:multinomial_conditional_sharp_cst}
        If $\xi < 1$ and either \(\frac{\log j^*}{\mu_j^*} \to 0\) or \(\frac{\log j^*}{\mu_{j^*}} \to \infty\), then
        \begin{align}
            \begin{split}\label{eqn:multinomial_chisquare_bound_sharp_cst}
            &\chi^2\left(\left.\left.\tilde{P}_\pi \,\right|\right|\, \tilde{P}_0\right) + 1 \\
            &\leq (1+o(1)) E\left(e^{\frac{\psi^2}{m^2\mu_{j^*}}|\mathcal{I} \cap \mathcal{I}'|} \right) \left[1 + \frac{1}{j^*\!-\!m} \bigg( e^{\frac{\psi^2}{\mu_{j^*}}}P\left\{ \Poisson\!\left(\frac{(\mu_{j^*} \! + \psi)^2}{\mu_{j^*}}\right) \leq \mu_{j^*}\! + \psi \right\} - 1 \bigg)_{\!\!+} \right]
            \end{split}
        \end{align}
        provided \(0 < c < c_\eta\). Here, \(\mathcal{I}\) and \(\mathcal{I}'\) are i.i.d. copies and we adopt the convention $\frac{\psi^2}{m^2\mu_{j^*}}|\mathcal{I} \cap \mathcal{I}'| = 0$ if $\psi = 0$ and $\mathcal{I} = \mathcal{I}' = \emptyset$ due to $m = 0$.  
    \end{lemma}
    \begin{proof}[Proof of Lemma~\ref{lemma:multinomial_conditional_sharp_cst}]
    The result can be obtained by noting \(P_0(E^c), P_\pi(E^c) = o(1)\) as in the proof of Lemma~\ref{lemma:asymptotic_poisson_second_moment} and following the calculations in the proof of Lemma~\ref{lemma:multinomial_conditional}.
        
    \end{proof}

\begin{lemma}\label{lem:sharp_constant_prior_mult}
    Assume \eqref{eqn:multinomial_psi_large} holds for some sufficiently large $C_*$ for the index $j^*$ defined in~\eqref{eq_def_jstar_mult_sharp_constant}.
    If $\xi<1$ and either \(\frac{\log j^*}{\mu_j^*} \to 0\) or \(\frac{\log j^*}{\mu_{j^*}} \to \infty\), then
    \begin{align*}
        E\bigg( e^{\frac{\psi^2}{m^2\mu_{j^*}}|\mathcal{I} \cap \mathcal{I}'|}\bigg) 
        & = \exp(o(1)).
    \end{align*}
\end{lemma}
    
    \begin{proof}[Proof of Lemma~\ref{lem:sharp_constant_prior_mult}]

    \textbf{Case 1.} Assume that $\frac{\log(j^*)}{\mu_{j^*}}\to 0$, which implies $\psi \sim \xi\sqrt{2\mu_{j^*} \log(ej^*)}$ by Lemma \ref{lemma:h_asymptotics}. 
    Proceeding as in~\eqref{eqn:multinomial_chisquare_II}, we obtain, for some constant $C$ whose value can change from line to line 
    \begin{align*}
        E\left( e^{\frac{\psi^2}{m^2\mu_{j^*}}|\mathcal{I} \cap \mathcal{I}'|}\right) &\leq  \exp\left(\frac{m^2}{(j^*-1)}\left(e^{\frac{\psi^2}{m^2\mu_{j^*}}} - 1\right)\right)\\
        & \leq \exp\left(\frac{C {j^*}^{1/2}}{(j^*-1)}\left(\exp\left({\frac{2\xi^2\log(j^*)(1+o(1))}{4}}\right) - 1\right)\right)\\
        & \leq \exp\left(\frac{C }{\sqrt{j^*}}\left({(j^*)}^{\frac{\xi^2}{2}(1+o(1))} - 1\right)\right)\\
        & = \exp(o(1)) \qquad \text{ since } j^* \to \infty \text{ and } \xi<1.
    \end{align*}
    
    \textbf{Case 2.} Assume now that $\frac{\log(j^*)}{\mu_{j^*}}\to \infty$, which implies $\psi \sim \xi\frac{\log(j^*)}{\log(\log(j^*)/\mu_{j^*})}$. 
    We have 
    \begin{align*}
        m \sim h^{-1} \left(\frac{\log(j^*)}{\mu_{j^*}}\right) \sim \frac{\log(j^*)/\mu_{j^*}}{\log(\log(j^*)/\mu_{j^*})}.
    \end{align*}
We obtain
\begin{align*}
        E\left( e^{\frac{\psi^2}{m^2\mu_{j^*}}|\mathcal{I} \cap \mathcal{I}'|}\right) &\leq  \exp\left(\frac{m^2}{(j^*-1)}\bigg(e^{\frac{\psi^2}{m^2\mu_{j^*}}} - 1\bigg)\right)\\
        & \leq \exp\left(\frac{C {j^*}^{1/2}}{(j^*-1)}\left(e^{\xi^2 \mu_{j^*}(1+o(1))} - 1\right)\right)\\
        & \leq \exp\left(\frac{C {j^*}^{1/2}}{(j^*-1)}\left(e^{o(\log(j^*))} - 1\right)\right)\\
        & = \exp(o(1))
    \end{align*}
    for $c$ small enough and $j^* \to \infty$, since the assumption $\frac{\log(j^*)}{\mu_{j^*}}\to \infty$ implies $\xi^2\mu_{j^*} \leq \frac{1+o(1)}{4} \log (j^*)$.
    \end{proof}

    \begin{lemma}\label{lem:control_chi2_multinomial_sharp_cst}
        There exists a sufficiently large universal constant \(C_* \geq 1\) such that the following holds. If the condition (\ref{eqn:multinomial_psi_large}) is satisfied for the index $j^*$ defined in~\eqref{eq_def_jstar_mult_sharp_constant} and either \(\frac{\log j^*}{\mu_j^*} \to 0\) or \(\frac{\log j^*}{\mu_{j^*}} \to \infty\), then   
        \begin{equation*}
            1 + \frac{1}{j^*-m} \left( e^{\frac{\psi^2}{\mu_{j^*}}}P\left\{ \Poisson\left(\frac{(\mu_{j^*} + \psi)^2}{\mu_{j^*}}\right) \leq \mu_{j^*} + \psi \right\} - 1 \right)_{\!+}  \leq 1 + o(1).
        \end{equation*}
    \end{lemma}
    \begin{proof}
        The proof proceeds by noting that, since $m \lesssim {j^*}^{1/4}$ and $j^* \to \infty$, we have 
        $\frac{1}{j^*-m} \sim \frac{1}{j^*}$, so
        \begin{align*}
            1 + \frac{1}{j^*-m} \left( e^{\frac{\psi^2}{\mu_{j^*}}}P\left\{ \Poisson\left(\frac{(\mu_{j^*} + \psi)^2}{\mu_{j^*}}\right) \leq \mu_{j^*} + \psi \right\} - 1 \right)_{\!+} \\\leq (1+o(1))E\left(\exp\left(\frac{\psi^2}{\mu_{j^*}} \mathbbm{1}_{\{J = J'\}}\right) P\left\{\left.\max_{1 \leq j \leq j^*} W_j - \mu_{j^*} \leq \psi \,\right|\, \rho, \rho'\right\}\right) 
        \end{align*}
        where \(\rho, \rho' \overset{iid}{\sim} \pi\) and \(J, J'\) are corresponding random indices, and \(W_j \,|\, \rho, \rho' \overset{ind}{\sim} \Poisson\left(\frac{\rho_j \rho_j'}{\mu_{j^*}}\right)\).
        We can now conclude by repeating the same steps as in the proof of Proposition~\ref{prop:asymptotic_poisson_mgf}.
    \end{proof}
        \begin{proof}[Proof of (ii) in Theorem \ref{thm:multinomial_sharp_constant}]
        Fix \(\xi < 1\). 
        Note since \(n' = n(1+o(1))\) that \(\frac{\log j^*}{nq_0^{-\max}(j^*)} \to 0\) implies \(\frac{\log j^*}{\mu_{j^*}} \to 0\) and \(\frac{\log j^*}{nq_0^{-\max}(j^*)} \to \infty\) implies \(\frac{\log j^*}{\mu_{j^*}} \to \infty\).
        Then by Proposition \ref{prop:flattening_sharp_constant}, equation \ref{eq_TV_chi2_mult_sharp_constant}, and Lemmas \ref{lemma:multinomial_conditional_sharp_cst},~\ref{lem:sharp_constant_prior_mult},~\ref{lem:control_chi2_multinomial_sharp_cst} 
        
        \begin{align*}
            \lim_{p \to \infty} \mathcal{R}_{\mathcal{M}}(\epsilon, q_0) &\geq \lim_{p \to \infty} \left(1 - \dTV\left(P_0, P_\pi\right)\right) \geq \lim_{p \to \infty}\left(1 - \frac{1}{2}\sqrt{\chi^2\left(\tilde{P}_{\pi} \,||\, \tilde{P}_0\right)}\right) = 1,
        \end{align*}
        
        as claimed. 
    \end{proof}
    
\section{Auxiliary, technical tools}
    Throughout this section, let \(h : [-1, \infty) \to \R\) denote the function with \(h(u) = (1+u)\log(1+u) - u\) for \(u > -1\) and \(h(-1) = 1\).

    \begin{lemma}[Bennett's inequality for Poisson random variables]\label{lemma:Bennett}
        Suppose \(Y \sim \Poisson(\rho)\) where \(\rho > 0\). If \(u \geq 0\), then \(P\left\{Y \geq \rho(1+u)\right\} \leq \exp\left(-\rho h(u)\right)\). 
        If \(0 \leq u < 1\), then \(P\left\{Y \leq \rho(1-u)\right\} \leq \exp\left(-\rho h(-u)\right)\). 
        In particular, for any \(u \geq 0\) we have \(P\left\{|Y - \rho| \geq \rho u\right\} \leq 2\exp\left(-\rho h(u)\right)\). 
    \end{lemma}
    \begin{proof}
        The results will follow from Chernoff's method. To show the first claim, let \(u \geq 0\). For any \(t \geq 0\), observe 
        \begin{align*}
            P\left\{Y \geq \rho(1 + u)\right\} = P\left\{e^{tY} \geq e^{t\rho(1+u)}\right\}
            \leq e^{-t\rho(1+u)} E(e^{tY})
            = \exp\left(-t\rho(1+u) + \rho(e^{t}-1)\right). 
        \end{align*}
        Selecting \(t = \log(1+u)\), we obtain the bound \(P\left\{Y \geq \rho(1 + u)\right\} \leq e^{-\rho h(u)}\). The second claim is obtained similarly. Let \(0 \leq u < 1\). Now note for any \(t \leq 0\), 
        \begin{equation*}
            P\left\{Y \leq \rho(1-u)\right\} = P\left\{tY \geq t\rho(1-u)\right\} = P\left\{e^{tY} \geq e^{t\rho(1-u)}\right\} \leq \exp\left(-t\rho(1-u) + \rho(e^t-1)\right). 
        \end{equation*}
        Selecting \(t = \log(1-u)\) yields \(P\left\{Y \leq \rho(1-u)\right\} \leq e^{-\rho h(-u)}\), as desired. To show the final claim, observe for \(u \geq 0\), we have 
        \begin{equation*}
            P\left\{|Y-\rho| \geq \rho u\right\} \leq P\left\{Y \geq \rho(1+u)\right\} + P\left\{Y \leq \rho(1-u)\right\} \leq e^{-\rho h(u)} + e^{-\rho h(-u)}\mathbbm{1}_{\{0\leq u < 1\}} \leq 2e^{-\rho h(u)}
        \end{equation*}
        since \(0 \leq h(u) \leq h(-u)\) for \(u \geq 0\). The proof is complete. 
    \end{proof}
    \begin{corollary}\label{corollary:constant_rate_max}
        Suppose \(Y_1,...,Y_d \overset{iid}{\sim} \Poisson(\rho)\) where \(\rho > 0\). If \(\eta \in (0, 1)\), then
        \begin{equation*}
            P\left\{\max_{1 \leq j \leq d} Y_j - \rho > \rho h^{-1}\left(\frac{\log\left(\frac{d}{\eta}\right)}{\rho}\right)\right\} \leq \eta. 
        \end{equation*}
    \end{corollary}

    \begin{lemma}\label{lemma:bennett_cancellation}
    Suppose \(\nu, c, \xi> 0\). If \(c^2 \xi > \nu\), then 
    \begin{equation*}
        e^{\frac{c^2\xi^2}{\nu}} P\left\{ \Poisson\left(\frac{(\nu + c\xi)^2}{\nu}\right) \leq \nu + \xi\right\} \leq \exp\left(-\nu h\left(\frac{\xi}{\nu}\right) + 2\nu\left(1 + \frac{\xi}{\nu}\right)\log\left(1 + \frac{c\xi}{\nu}\right)\right).
    \end{equation*}
\end{lemma}
\begin{proof}
    To bound the Poisson probability, we would like to use Lemma \ref{lemma:Bennett}. Note \(\frac{(\nu + c\xi)^2}{\nu} = \nu + 2c\xi + \frac{c^2\xi^2}{\nu} > \nu + \xi\) since \(\xi > c^{-2}\nu\), and so we can indeed apply Lemma \ref{lemma:Bennett} to obtain 
    \begin{align*}
        e^{\frac{c^2\xi^2}{\nu}} P\left\{ \Poisson\left(\frac{(\nu + c\xi)^2}{\nu}\right) \leq \nu + \xi\right\} &= e^{\frac{c^2\xi^2}{\nu}} P\left\{\Poisson\left(\frac{(\nu + c\xi)^2}{\nu}\right) \leq \frac{(\nu + c\xi)^2}{\nu} \left(\frac{\nu \xi + \nu^2}{(\nu + c\xi)^2}\right) \right\} \\
        &\leq \exp\left(\frac{c^2\xi^2}{\nu} - \frac{(\nu + c\xi)^2}{\nu} h\left(-1+\frac{\nu\xi + \nu^2}{(\nu + c\xi)^2}\right)\right). 
    \end{align*}
    By direct calculation, observe 
    \begin{align*}
    &\frac{(\nu + c\xi)^2}{\nu} h\left(-1+\frac{\nu\xi + \nu^2}{(\nu + c\xi)^2}\right) \\
    &= \frac{(\nu + c\xi)^2}{\nu} \left(\frac{\nu\xi+\nu^2}{(\nu + c\xi)^2} \log\left(\frac{\nu\xi+\nu^2}{(\nu + c\xi)^2} \right) + 1 - \frac{\nu\xi+\nu^2}{(\nu+c\xi)^2}\right) \\
    &= \nu\left(1 + \frac{\xi}{\nu}\right) \log\left(1 + \frac{\xi}{\nu}\right) - 2\nu \left(1 + \frac{\xi}{\nu}\right) \log\left(1 + \frac{c\xi}{\nu}\right) + \frac{(\nu + c\xi)^2}{\nu} - \nu\left(1 + \frac{\xi}{\nu}\right) \\
    &= \nu h\left(\frac{\xi}{\nu}\right) - 2\nu \left(1 + \frac{\xi}{\nu}\right)\log\left(1 + \frac{c\xi}{\nu}\right) + 2c\xi + \frac{c^2\xi^2}{\nu}. 
    \end{align*}
    Therefore, 
    \begin{align*}
        \exp\left(\frac{c^2\xi^2}{\nu} - \frac{(\nu + c\xi)^2}{\nu} h\left(-1+\frac{\nu\xi + \nu^2}{(\nu + c\xi)^2}\right)\right) &= \exp\left(-\nu h\left(\frac{\xi}{\nu}\right) + 2\nu \left(1 + \frac{\xi}{\nu}\right)\log\left(1 + \frac{c\xi}{\nu}\right) - 2c\xi\right),
    \end{align*}
    and so the claimed result follows because \(-2c\xi \leq 0\).
\end{proof}

    \begin{lemma}[Bennett's inequality for bounded random variables - Theorem 2.9 \cite{boucheron_concentration_2013}]\label{lemma:bennett_bounded}
        Let \(Z_1,...,Z_n\) be independent random variables with finite variance such that \(|Z_i - E(Z_i)| \leq b\) for some \(b > 0\) almost surely for all \(i \leq n\). Let \(v = \sum_{i=1}^{n} \Var(Z_i)\). If \(t > 0\), then
        \begin{equation*}
            P\left\{\left|\sum_{i=1}^{n} (Z_i - E(Z_i))\right| \geq t\right\} \leq 2\exp\left(-\frac{v}{b^2}h\left(\frac{bt}{v}\right)\right).
        \end{equation*}
    \end{lemma}
    \begin{corollary}\label{corollary:bennett_binomial}
        If \(Y \sim \Binomial(n, \pi)\) where \(n \geq 1\) and \(\pi \in [0, 1]\), then 
        \begin{equation*}
            P\left\{|Y - n\pi| \geq u \cdot n\pi(1-\pi)\right\} \leq 2 e^{-n\pi(1-\pi) h(u)}
        \end{equation*}
        for \(u \geq 0\). 
    \end{corollary}

    \begin{definition}[Lambert function]\label{def:Lambert}
        The Lambert function \(\W : [0, \infty) \to \R\) is defined to be \(\W(x) = y\) where \(y\) is the solution to the equation \(ye^y = x\). 
    \end{definition}
    
    \begin{lemma}\label{lemma:Lambert_order}
        For \(x \geq 1\), we have \(\W(x) \asymp \log(ex)\). 
    \end{lemma}


    \begin{lemma}\label{lemma:Lambert_identities}
        Let \(\W\) denote the Lambert function given by Definition \ref{def:Lambert}. Then \(e^{\W(x\log x)} = x\) for \(x \geq e^{-1}\) and \(\lim_{x \to \infty} \frac{\W(x)}{\log x} = 1\). 
    \end{lemma}

    \begin{lemma}\label{lemma:h_inverse}
        The function \(h\) restricted to the domain \([0, \infty)\) admits an inverse \(h^{-1}\) which satisfies 
        \begin{equation*}
            h^{-1}(y) \asymp 
            \begin{cases}
                \sqrt{y} &\textit{if } y \leq 1, \\
                \frac{y}{\log(ey)} &\textit{if } y > 1.  
            \end{cases}
        \end{equation*}
        \begin{proof}
            It follows immediately from the fact \(h\) is a strictly increasing function that it admits an inverse \(h^{-1}\). Consider that \(h(e-1) = 1\), and so \(h(x) \leq 1\) for all \(x \leq e-1\) and \(h(x) > 1\) for all \(x > e-1\). Since we have the Taylor expansion \(\log(1+x) = x + \frac{x^2}{2} + O(x^3)\), it follows that \(h(x) \asymp x^2\) for \(x \leq e-1\). Therefore, \(h^{-1}(y) \asymp \sqrt{y}\) for \(y \leq 1\). For \(x \geq e-1\), we have \(x \leq \frac{e-1}{e} (1+x)\log(1+x)\) and so \(h(x) \geq e^{-1}(1+x)\log(1+x)\). Trivially \(h(x) \leq (1+x)\log(1+x)\), so for \(x \geq e-1\) we have 
            \begin{equation*}
                e^{-1}(1+x) \log(e^{-1}(1+x)) \leq h(x) \leq (1+x)\log(1+x).    
            \end{equation*}
            Let \(g_1(x) = e^{-1}(1+x)\log(e^{-1}(1+x))\) and \(g_2(x) = (1+x)\log(1+x)\). These functions are also strictly increasing and thus admit inverses. Since \(g_1 \leq h \leq g_2\), we have \(g_1^{-1} \geq h^{-1} \geq g_2^{-1}\). Consider that \(1 + g_2^{-1}(y) = e^{\W(y)}\) where \(\W\) is the Lambert function (see Definition \ref{def:Lambert}). By definition and Lemma \ref{lemma:Lambert_order}, we have \(1+g_2^{-1}(y) = e^{\W(y)} = \frac{y}{\W(y)}\). Consider that \(\frac{y}{\W(y)}\) is strictly increasing in \(y\). Since \(\W(1) \leq \frac{2}{3}\), at \(y = 1\) we have \(1 + g_2^{-1}(1) = \frac{1}{\W(1)} \geq \frac{3}{2}\). Therefore, it follows that \(g_2^{-1}(y) = \frac{y}{\W(y)} - 1 \asymp \frac{y}{\W(y)} \asymp \frac{y}{\log(ey)}\) for \(y > 1\). By a similar argument, \(g_1^{-1}(y) \asymp \frac{y}{\log(ey)}\) for \(y > 1\). Therefore, \(h^{-1}(y) \asymp \frac{y}{\log(ey)}\) for \(y > 1\) as claimed. 
        \end{proof}
    \end{lemma}

    \begin{lemma}\label{lemma:h_asymptotics}
        Let \(h^{-1}\) denote the inverse of the function \(h\) restricted to \([0, \infty)\). Then \(\lim_{x \to 0} \frac{h^{-1}(x)}{\sqrt{2x}} = 1\) and \(\lim_{x \to \infty} \frac{h^{-1}(x)\log x }{x} = 1\). 
    \end{lemma}
    \begin{proof}
        To prove the first claim, first note \(h^{-1}\) is continuous, strictly increasing, and \(h^{-1}(0) = 0\). Consider by Taylor expansion \(h(y) = \frac{y^2}{2} + o(y^3) = \frac{y^2}{2}\left(1 + o(y)\right)\) as \(y \to 0\). Letting \(y = h^{-1}(x)\) note that \(y \to 0\) as \(x \to 0\). Since \(x = h(y) = \frac{y^2}{2}\left(1 + o(y)\right)\), it follows 
        \begin{align*}
            \lim_{x \to 0} \frac{h^{-1}(x)}{\sqrt{2x}} = \lim_{x \to 0} \frac{y}{\sqrt{2\cdot \frac{y^2}{2}(1+o(y))}} = 1. 
        \end{align*}
        The proof of the second claim is similar. Consider \(h(y) = (y \log y) \left(1 + O\left(\frac{1}{\log y}\right)\right)\) as \(y \to \infty\). Letting \(y = h^{-1}(x)\) note that \(y \to \infty\) as \(x \to \infty\). By Lemma \ref{lemma:Lambert_identities}, we have 
        \begin{align*}
            \lim_{x \to \infty} \frac{h^{-1}(x) \log x}{x} = \lim_{x \to \infty} \frac{h^{-1}(x)}{e^{\W(x)}} = \lim_{x \to \infty} \frac{y}{e^{\W(y \log y)}} = 1
        \end{align*}
        as desired. 
    \end{proof}

    \begin{lemma}[Data-processing inequality - Theorem 7.4 \cite{polyanskiy_information_2024}]\label{lemma:dpi}
        Consider a channel that produces \(Y\) given \(X\) based on the conditional law \(P_{Y|X}\). Let \(P_Y\) (resp. \(Q_Y\)) denote the distribution of \(Y\) when \(X\) is distributed as \(P_X\) (resp. \(Q_X\)). For any \(f\)-divergence \(D_f(\cdot\,||\, \cdot)\), we have \(D_f(P_Y\,||\,Q_Y) \leq D_f(P_X\,||\,Q_X)\). 
    \end{lemma}

    Let \((\mathcal{X}, \mathcal{A})\) be a measurable space on which \(P\) and \(Q\) are two probability measures. Suppose \(\nu\) is a \(\sigma\)-finite measure on \((\mathcal{X}, \mathcal{A})\) such that \(P\ll \nu\) and \(Q \ll \nu\). Define \(p = dP/d\nu\) and \(q = dQ/d\nu\). 

    \begin{definition}[Hellinger distance - Definition 2.3 \cite{tsybakov_introduction_2009}]\label{def:hellinger}
        The Hellinger distance between \(P\) and \(Q\) is defined as \(H(P, Q) = \left(\int (\sqrt{p} - \sqrt{q})^2 \, d\nu \right)^{1/2}\). 
    \end{definition}
   
    \begin{lemma}\label{lemma:poisson_TV}
        Suppose \(\mu, \delta \geq 0\). The total variation distance between \(\Poisson(\mu)\) and \(\Poisson(\mu + \delta)\) satisfies \(\dTV(\Poisson(\mu), \Poisson(\mu+\delta)) \leq \sqrt{\delta}\).
    \end{lemma}
    \begin{proof}
        The probability mass function of \(\Poisson(\mu)\) is given by \(p(k) := P\left\{\Poisson(\mu) = k\right\} = \frac{e^{-\mu}\mu^k}{k!}\) for \(k = 0,1,2,...\). Let \(q\) denote the probability mass function of \(\Poisson(\mu+\delta)\). Then
        \begin{equation*}
            \dTV(p, q) \leq H(p, q) = \sqrt{2}\left(1 - \sum_{k=0}^{\infty} \frac{e^{-\mu - \frac{\delta}{2}}}{k!} (\mu^2 + \mu\delta)^{k/2}  \right)^{1/2} = \sqrt{2}\left(1 - \exp\left(\sqrt{\mu^2 + \mu \delta} - \mu - \frac{\delta}{2}\right)\right)^{1/2}.
        \end{equation*}
        The first inequality is standard \cite{tsybakov_introduction_2009}. It is immediate that we have \(\exp\left(\sqrt{\mu^2 + \mu \delta} - \mu - \frac{\delta}{2}\right) \geq e^{-\frac{\delta}{2}}\), and so \(H(\Poisson(\mu), \Poisson(\mu+\delta)) \leq \sqrt{2}\left(1 - e^{-\frac{\delta}{2}}\right)^{1/2} \leq \sqrt{\delta}\) using the inequality \(1-e^{-x} \leq x\).
    \end{proof}

    \begin{lemma}\label{lemma:conditional_TV}
        Suppose \(P\) is a probability measure and \(E\) is an event. If \(\tilde{P}\) is the measure conditional on \(E\), that is, \(\tilde{P}(A) = \frac{P(A \cap E)}{P(E)}\), then \(\dTV(P, \tilde{P}) \leq 2 P(E^c)\). 
    \end{lemma}
    \begin{proof}
        It follows directly from the definition of total variation that \(\dTV(P, \tilde{P}) = \sup_{A} |P(A) - \tilde{P}(A)| \leq P(E^c) + \sup_A \left|P(A \cap E) - \frac{P(A \cap E)}{P(E)}\right| \leq P(E^c) + P(E) \cdot \frac{1-P(E)}{P(E)} \leq 2P(E^c)\). 
    \end{proof}

\end{document}